\documentclass[12pt, letterpaper]{article}
\usepackage{amsmath}
\usepackage{amsthm}
\usepackage{amsfonts}
\usepackage{amssymb}
\usepackage{color}
\usepackage{enumitem}
\newtheorem{theorem}{Theorem}[section]
\newtheorem{corollary}{Corollary}[section]
\newtheorem{lemma}{Lemma}[section]
\newtheorem{definition}{Definition}[section]
\newtheorem{proposition}{Proposition}[section]
\newtheorem{example}{Example}[section]
\newtheorem{examples}{Examples}[subsection]
\newtheorem{remark}{Remark}[section]
\newcommand\blue{\color{blue}}
\usepackage{calc}
\usepackage{tikz}
\usetikzlibrary{decorations.markings}
\tikzstyle{vertex}=[circle, draw, inner sep=0pt, minimum size=6pt]
\newcommand{\vertex}{\node[vertex]}

\title{Drawing Graphs using TikZ in {\LaTeX}}
\usepackage{graphicx}
\allowdisplaybreaks
\usepackage{float}
\usepackage{subcaption}
\definecolor{cadmiumgreen}{rgb}{0.0, 0.42, 0.24}

\newcommand\red{\color{red}}
\begin{document}
\title{ Monge solutions and uniqueness in multi-marginal optimal transport via graph theory\thanks{BP is pleased
		to acknowledge support from Natural Sciences and Engineering Research
		Council of Canada Grant 04658-2018. The work of AVJ was completed in partial fulfillment of the requirements for a doctoral degree in mathematics at the University of Alberta.}}
\author{BRENDAN PASS \thanks{%
		Department of Mathematical and Statistical Sciences, University of Alberta,
		Edmonton, Alberta, Canada pass@ualberta.ca.}\\	 
	\and 
	ADOLFO VARGAS-JIM\'{E}NEZ\thanks{%
		Department of Mathematical and Statistical Sciences, University of Alberta,
		Edmonton, Alberta, Canada vargasji@ualberta.ca.}}
\date{\today}
\maketitle
\begin{abstract}
We study a multi-marginal optimal transport problem with surplus $b(x_{1}, \ldots, x_{m})=\sum_{\{i,j\}\in P} x_{i}\cdot x_{j}$, where $P\subseteq Q:=\{\{i,j\}: i, j \in \{1,2,...m\}, i \neq j\}$. We reformulate this problem by associating each surplus of this type with a graph with $m$ vertices whose set of edges is indexed by $P$.  We then establish uniqueness and Monge solution results for two general classes of surplus functions. Among many other examples, these classes encapsulate  the Gangbo and \'{S}wi\c{e}ch surplus \cite{Gangbo} and the surplus $\sum_{i=1}^{m-1}x_{i}\cdot x_{i+1} + x_{m}\cdot x_{1}$ studied in an earlier work by the present authors \cite{Pass5}.
% whose edges are the elements of $P$.
 \end{abstract}
\section{\textbf{Introduction
 }}
Multi-Marginal optimal transport is the problem of correlating probability measures to maximize a given surplus function. There are two formulations of this problem: the Kantorovich formulation and the Monge formulation. In the Kantorovich formulation, given Borel probability measures $\mu_{i}$ on open bounded sets $X_{i}\subseteq \mathbb{R}^{n}$, with $i=1, \ldots, m$,  and $b$ a real-valued surplus function  on the product space $X_{1}\times  \ldots \times X_{m}$, the goal is to maximize
 \begin{equation}\label{KP}
\displaystyle \int_{X_{1}\times \ldots \times X_{m}} b(x_{1},\ldots, x_{m})d\mu,\tag{KP}
 \end{equation}
 among all Borel probability measures $\mu$ on the product space  $X_{1}\times  \ldots \times X_{m}$  whose marginals are the $\mu_i$; that is, for each fixed $i\in\{1,\ldots,m\}$, $\mu(X_{1}\times\ldots \times X_{i-1}\times A \times X_{i+1}\times \ldots \times X_{m})=\mu_{i}(A)$ for any Borel set $A\subseteq X_{i}$.\par
 In the Monge formulation, one seeks to maximize
 \begin{equation}\label{MP}
 \displaystyle \int_{X_{1}} b(x_{1},T_{2}x_{1},\ldots, T_{m}x_{1})d\mu_{1},\tag{MP}
 \end{equation}
 among all $(m-1)$-tuples of maps $(T_{2}, \ldots, T_{m})$ such that  $(T_{i})_{\sharp} \mu_{1}=\mu_{i}$ for all $i=2, \ldots, m$, where $(T_{i})_{\sharp} \mu_{1}$ denotes the \textit{image measure} of $\mu_{1}$ through $T_{i}$, defined by $(T_{i})_{\sharp} \mu_{1}(A) = \mu_{1}(T_{i}^{-1}(A))$, for any Borel set $A\subseteq X_{i}$. It is well known that problem \eqref{KP} is a relaxation of problem \eqref{MP}, as  for any $(m-1)$-tuple of maps $(T_{2}, \ldots, T_{m})$ satisfying the image measure constraint in \eqref{MP}, we can define $\mu=(Id,T_{2}, \ldots, T_{m})_{\sharp} \mu_{1}$, which satisfies the constraint in \eqref{KP} and 
$$\displaystyle \int_{X_{1}\times \ldots \times X_{m}} b(x_{1},\ldots, x_{m})d\mu=\displaystyle \int_{X_{1}} b(x_{1},T_{2}x_{1},\ldots, T_{m}x_{1})d\mu_{1}.$$
 Under mild conditions, a maximizer to \eqref{KP} exists.\par
When $m=2$, the classical optimal transport problems of Monge and Kantorovich arise in \eqref{KP} and \eqref{MP}. This case  has been widely studied  and it is reasonably well understood; in particular, if $\mu_{1}$ is absolutely continuous with respect to Lebesgue measure $\mathcal{L}^n$ and the \emph{twist} condition holds (for each fixed $x_{1}$, the map $x_{2}\mapsto D_{x_{1}}b(x_{1},x_{2})$ is injective), the Kantorovich solution $\mu$ induces a Monge solution and it is unique \cite{Brenier}\cite{McCann}\cite{Gangbo1}\cite{Gangbo2}. The classical optimal transportation problem has deep connections with many different areas of mathematics, including analysis, probability, PDE and  geometry, and a great wealth  of applications in, for example, economics,  fluid mechanics,  physics, among many other areas \cite{Filippo}\cite{Villani}\cite{Villani2}. 

 A  wide variety of applications has also recently emerged for the $m \geq 3$ case, including, for example,  matching in economics \cite{Carlier}\cite{McCann2}\cite{Pass6},  density functional theory in computation \cite{BBCP2021}\cite{BBCP2022}\cite{Gori}\cite{BCP2018}\cite{Codina}, interpolating among distributions in machine learning and statistics \cite{BentoMi, ZemelPanaretos} (see also \cite{Pass3} for an overview and additional references). However, determining whether solutions to the multi-marginal Kantorovich problem \eqref{KP} are unique and of Monge form % the uniqueness of a Kantorovich solution and the existence of an  $(m-1)$-tuple of maps $(T_{2}, \ldots, T_{m})$, satisfying  $\mu=(Id,T_{2}, \ldots, T_{m})_{\sharp} \mu_{1}$, 
has  proven much more challenging, as the answer depends on the form of the surplus $b$ in subtle ways which are still not understood. A fairly general condition, called \emph{twist on splitting sets}, together with absolute continuity of $\mu_1$, was shown in \cite{Kim} to guarantee unique Monge solutions, unifying results in \cite{ Carlier03, Gangbo,   Heinich, PassPinamontiVedovato20, KimPass2015, Pass4, Pass6}. However, this condition is cumbersome and difficult to verify for many surpluses, and, as our recent work reveals, there  exist surplus functions outside of this class for which Monge solution and uniqueness results hold (albeit stronger conditions on the marginals are required than for surplus functions which are twisted on splitting sets) \cite{Pass5}. One of our goals here is to shed some light on the general problem of classifying which surplus functions yield unique Monge solutions, by investigating a certain class of surpluses.

% Some efforts have been made to establish general conditions implying unique Monge solutions \cite{Pass4}\cite{Kim}, but, we might find surplus functions providing solutions of Monge type not  covered by these frameworks \cite{Pass5}. One of our goals here is to shed some light on this issue by investigating a certain class of surplus functions.\par
One of the best known surplus functions in the multi-marginal setting is the Gangbo and \'{S}wi\c{e}ch surplus function \cite{Gangbo}
\begin{equation}\label{Gangbo cost}
\sum_{\substack{1\leq i<j\leq m}} x_{i}\cdot x_{j}.
\end{equation}
 In their seminal work, they prove that every solution to \eqref{KP} is concentrated on a graph of a measurable map, thus obtaining a unique solution to the Monge-Kantorovich problem (in the subsequently developed terminology of \cite{Kim}, \eqref{Gangbo cost} is twisted on splitting sets). In \cite{AC2011}, Agueh and Carlier proved that solving the multi-marginal Kantorovich problem  with a weighted version of \eqref{Gangbo cost} is equivalent to finding the barycenter of the marginals $\mu_1, \ldots, \mu_m$.
% In this work, we encapsulates \eqref{Gangbo cost} in a broader problem; we study the surplus functions of the form
On the other hand, our recent paper \cite{Pass5} focused on \emph{cyclic costs} of the form
	
	\begin{equation}\label{cycle cost}
	\sum_{i=1}^{m-1}x_{i}\cdot x_{i+1} + x_{m}\cdot x_{1},
	\end{equation}
	whose origin lies in the time discretization of Arnold's variational interpretation of the incompressible Euler equation \cite{Arnold66}\cite{Brenier89}; we showed that unique Monge solutions are obtained for $m \leq 4$ (although when $m=4$, $\mu_2$, in addition to $\mu_1$, is required to be absolutely continuous -- in this case \eqref{cycle cost} violates the twist on splitting sets condition from \cite{Kim}), whereas higher dimensional solutions exist when $m \geq 5$.
	In the present work, we encapsulate the surplus functions \eqref{Gangbo cost} and \eqref{cycle cost} by studying a more general form in which arbitrary interaction structures between the variables are permitted.  More precisely, we set
\begin{equation}\label{Main cost}
b(x_{1}, \ldots, x_{m})=\sum_{\{i,j\}\in P} x_{i}\cdot x_{j},
\end{equation}
 where $P\subseteq Q:=\{\{i,j\}: i, j \in \{1,2,...m\}, i \neq j\}$ (note  that \eqref{Main cost} takes shape \eqref{Gangbo cost} when $P=Q$, and shape \eqref{cycle cost} when  $P= \big\{ \{ i,i+1\}: i=1, \ldots, m-1\big \}\cup \big\{ \{1,m\}\big \}$; our main goal is then to identify conditions on $P$ which lead to Monge solutions. For this, we exploit a natural connection to graph theory;
%We focused on identifying and classifying a wide range of P's providing Monge solutions, where, to facilitate the classification we use graph theory; 
in particular, we associate the surplus function \eqref{Main cost} with a graph whose vertices we label $\{v_1,...,v_m\}$ and whose set of edges is indexed by $P$. For instance, it is evident that surplus \eqref{Gangbo cost} is associated to a complete graph with vertices  $\{v_{1}, \ldots, v_{m}\}$, denoted by $C_{m}$. See figure below for the case $m=7$.
\begin{figure}[H]  
\begin{center}
\begin{tikzpicture}[x=1.5cm, y=1.5cm]
	\vertex (v1) at (51:1) [label=51:$v_1$]{};
	\vertex (v2) at (103:1) [label=103:$v_2$]{};
	\vertex (v3) at (154:1) [label=154:$v_3$]{};
	\vertex (v4) at (206:1) [label=206:$v_4$]{};
	\vertex (v5) at (257:1) [label=257:$v_5$]{};
	\vertex (v6) at (309:1) [label=309:$v_6$]{};
	\vertex (v7) at (360:1) [label=360:$v_7$]{};
	\path 
		
		(v1) edge (v6)
		(v1) edge (v4)
		(v6) edge (v7)
		(v6) edge (v4)
		(v6) edge (v2)
		(v6) edge (v3)
		(v7) edge (v1)
		(v7) edge (v2)
		(v7) edge (v4)
		(v7) edge (v3)
		(v7) edge (v5)
		(v5) edge (v1)
		(v5) edge (v3)
		(v5) edge (v4)
		(v5) edge (v2)
		(v4) edge (v2)
		(v4) edge (v3)
		(v2) edge (v3)
		(v1) edge (v2)
		(v1) edge (v3)
		(v5) edge (v6)
		;
\end{tikzpicture}
\caption{$C_{m}$ \text{with} $m=7$.} \label{fig:M1}  
\end{center}
\end{figure}
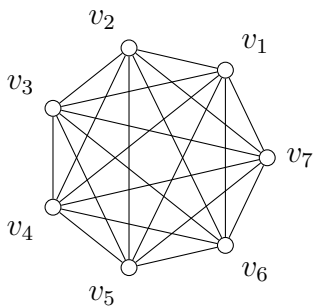 
 In this setting every subgraph with $m$ vertices $G$ of $C_{m}$ is associated to a surplus $\sum_{\{v_i,v_j\} \in E(G)} x_i\cdot x_j$, where $$E(G) =\{\{v_i,v_j\}: \text{ }G \text{ has an edge between } v_i\text{ and }v_j, v_i\neq v_j\}.$$ %formed by a sum  of terms of the form , so, studying the optimal transportation problem for \eqref{Main cost} is equivalent to studying the optimal transportation problem for the surplus functions associated to the subgraphs of $C_{m}$. %Now, since the surplus associated to any not connected subgraph of $C_{m}$ can be expressed  as the sum of two surplus functions with not common variables, and then, we would not obtain a Monge type solution, we restrict our attention to the subgraphs of $C_{m}$ that are connected (see definition in Subsection 2.2).  
For instance, the "border" of $C_{m}$, that is, the cycle graph with vertex sequence $(v_1,\ldots, v_m,v_1)$ (see definition in Subsection 2.2 and Figure 2-b for the case $m=7$) is associated to  surplus \eqref{cycle cost}. Hence, according to our work in \cite{Pass5}, we can conclude that if $m\geqslant 4$, the only cycle graph providing unique Monge solution is the graph cycle with $m=4$ (under regularity conditions of the first and fourth marginal).  See Figure 2-a below. 
\begin{figure}[H]   
\begin{center}
\begin{subfigure}[b]{0.4\linewidth}
\begin{tikzpicture}[x=1.5cm, y=1.5cm]
	\vertex (v1) at (90:1) [label=90:$v_{1}$]{};
	\vertex (v2) at (180:1) [label=180:$v_{2}$]{};
	\vertex (v3) at (270:1) [label=270:$v_{3}$]{};
	\vertex (v4) at (360:1) [label=360:$v_{4}$]{};
	\path 
		(v1) edge (v2)
		(v2) edge (v3)
		(v3) edge (v4)
		(v4) edge (v1)
		;	
\end{tikzpicture}
\caption{Graph cycle with $m=4$.} \label{Graph cycle.}
  \end{subfigure}
  \qquad \qquad
\begin{subfigure}[b]{0.4\linewidth}
\begin{tikzpicture}[x=1.5cm, y=1.5cm]
	\vertex (v1) at (51:1) [label=51:$v_1$]{};
	\vertex (v2) at (103:1) [label=103:$v_2$]{};
	\vertex (v3) at (154:1) [label=154:$v_3$]{};
	\vertex (v4) at (206:1) [label=206:$v_4$]{};
	\vertex (v5) at (257:1) [label=257:$v_5$]{};
	\vertex (v6) at (309:1) [label=309:$v_6$]{};
	\vertex (v7) at (360:1) [label=360:$v_7$]{};
	\path 
		(v6) edge (v7)
		(v7) edge (v1)
		(v5) edge (v4)
		(v4) edge (v3)
		(v2) edge (v3)
		(v1) edge (v2)
		(v5) edge (v6)
		;
	
\end{tikzpicture}
\caption{Cycle graph with $m=7$.} \label{fig:M2}
\end{subfigure}
\caption{}
\end{center}
\end{figure}
 %for  the surplus  \begin{equation}\label{cycle cost}
%\sum_{i=1}^{m-1}x_{i}\cdot x_{i+1} + x_{m}\cdot x_{1},
%\end{equation}
% where \eqref{Main cost} takes this shape when $P= \big\{ ( i,i+1): i=1, \ldots, m-1\big \}\cup \big\{ ( 1,m)\big \}$.

%This surplus was studied by the present authors in \cite{Pass5}, whose origin lies in the time discretization of Arnold's variational interpretation of the incompressible Euler equation \cite{Arnold66}. We proved that, when $m=4$ and the first and last marginals are absolutely continuous, all solutions in \eqref{KP} are of Monge type. Additionally, we exhibit examples showing that when $m> 4$ solutions may not be of Monge type, even when all the marginals are absolutely continuous. Hence, we can conclude that the only cycle graph providing unique Monge solution is the graph cycle with $m=4$:
%\begin{figure}[H]
%\begin{center}
%\begin{tikzpicture}[x=1.5cm, y=1.5cm]
%	\vertex (x1) at (90:1) [label=90:$x_{1}$]{};
%	\vertex (x2) at (180:1) [label=180:$x_{2}$]{};
%	\vertex (x3) at (270:1) [label=270:$x_{3}$]{};
%	\vertex (x4) at (360:1) [label=360:$x_{4}$]{};
%	\path 
%		(x1) edge (x2)
%%		(x2) edge (x3)
%		(x3) edge (x4)
%		(x4) edge (x1)
%		;	
%\end{tikzpicture}
%\caption{Graph cycle} \label{Graph cycle.}
%\end{center}
%\end{figure} 
%In recent years have been some interest in connecting graph theory with optimal transport  in a computational setting.
The connection between multi-marginal surplus functions and  graphs described above recently appeared in a computational setting in \cite{Haasler2}, where a regularized (through an entropy term) multi-marginal optimal transport problem with surplus associated to a tree was studied.  %A recent work, considered the regularized (throughout an entropy term)  multi-marginal optimal transport problem with surplus associated to a tree  . 
 Although the scope of that work is restricted to a more basic graph structure (only trees were considered), the edges $\{v_i, v_j\}$ are associated to more general symmetric surplus $c_{ij}(x_i,x_j)$. Also, \cite{Haasler1}  established an equivalence of the regularized  multi-marginal optimal transport  and the inference problem for a probabilistic graphical model when both problems are associated to a common graph structure. On the other hand,  the same relationship was noted in \cite{Eckstein}, where connectedness of the graph played an important role in solving a  % the cost functional of a
  one dimensional multi-marginal martingale optimal transport problem under various assumptions; see Theorem 5.3 in \cite{Eckstein}.% was associated to a connected graph, which under some assumptions on the marginals, its optimal measures induce optimal measures of a multi-marginal optimal transport (associated to the same graph). .\\
% The marginals on these works, as in this paper, are defined over the vertices; however, there are works where the marginals are defined over the edges \cite{Chen}\cite{Chow}\cite{Conforti}. \\

 As we shall see in sections \ref{Section 3} and  \ref{Section 4},  our main results (Theorem \ref{Theorem 2} and Theorem \ref{Theorem 1}, as well as the related Propositions \ref{Proposition 2} and \ref{Proposition 3} ), provide a broad class of graphs providing unique Monge solutions; some of these are classical, well known graphs, whereas others are less standard and more exotic.  In particular, we highlight in Corollary \ref{cor: one missing edge} a special subclass of graphs encompassed by our theory, offering a generalization of the  Gangbo and \'{S}wi\c{e}ch result which we find conceptually appealing: the class in which each vertex is connected to all,  except at most one, of the other vertices.   Generally speaking, the graphs for which we establish Monge solution results come in two complementary classes; one  (see Section \ref{Section 3}) result from the extraction from the complete graph of subgraphs with a particular structure, while the other (see Section  \ref{Section 4}) is obtained by joining complete graphs in a special way.

 %where two  graphs operations will play an important role: Extraction and union of graphs (see definitions in Subsection 2.2). Furthermore, we introduce a special class of graphs that facilitates the description of our main theorems, we call these graphs semi-disjoint collection of graphs (see definition 2.2). Some parts of the proof of our main results are based on the essential ideas of the proof of the main theorem in \cite{Pass5}. 
We would like to emphasize that, in addition to the regularity assumption on $\mu_{1}$, which is standard in optimal transport, many of our results require extra regularity conditions on certain other marginals; these assumptions are not typical in optimal transport theory, but are necessary in our setting, since many surplus functions counterexamples to Monge solutions and uniqueness exist in their absence (see the second assertion of Proposition \ref{prop: non ToSS}).  Note that these examples confirm that the framework developed here reaches well beyond the twist on splitting sets theory, the most general currently known condition implying Monge solution and uniqueness results for multi-marginal problems; indeed, Proposition \ref{prop: non ToSS} verifies that the twist on splitting sets condition is violated by a wide variety of surplus functions, many of which fall within the scope of either Theorem \ref{Theorem 2} or the results in section 4 (Theorem \ref{Theorem 1} and Propositions \ref{Proposition 2} and \ref{Proposition 3}).

%Our main results apply to many graphs which are are not covered by the twist on splitting sets theory in \cite{Kim}, which is the most general theory known implying Monge solutions and uniqueness results.  \par
 In the next section, we recall and formulate some definitions and concepts from the theory of optimal transportation and graph theory, as well as establish some preliminary results connecting these two areas, one of which in particular will be used later on the paper.
 In section 3, we establish and prove our first main result and provide some examples. In section 4 we state and prove our second main result and provide examples. The short fifth section contains a (standard) proof of the fact that when every solution is of Monge type, as we prove for the graphs considered in sections 3 and 4, the solution must in fact be unique. The final section is reserved for discussion, including examples of graphs which fall outside the scope of the results proved in this paper, and for which the Monge solution and uniqueness questions therefore remain open.
\section{\textbf{Definitions and preliminaries
 }}
In this section, we formulate the main concepts used in this work.
\subsection{The dual problem}
For the surplus function \eqref{Main cost}, set
\begin{equation*}\label{Dp}
\mathcal{U}=\left\lbrace (u_{1}, u_{2},\ldots, u_{m})\in \prod_{i=1}^{m}L^{1}(\mu_{i}):b(x_{1},\ldots,x_{m})\leq \sum_{i=1}^{m}u_{i}(x_{i}), \forall (x_{1},\ldots,x_{m})\in X_{1}\times \ldots \times X_{m}
\right\rbrace.
\end{equation*}

 The dual of \eqref{KP} asks to minimize on $\mathcal{U}$ the map:
 
\begin{equation}\label{DP}
(u_{1}, u_{2},\ldots, u_{m})\mapsto\sum_{i=1}^{m}\int_{X_{i}}u_{i}(x_{i})d\mu_{i}(x_i)\tag{DP}.
  \end{equation}
%There is a well known sub-class of $\mathcal{U}$ key for this work, whose elements are described in the following definition.
The following subclass of $\mathcal{U}$ plays a key role in multi-marginal optimal transport theory:
\begin{definition}
An $m$-tuple of functions $(u_{1}, u_{2},\ldots, u_{m})$ is $b$-conjugate if for all i, $$u_{i}(x_{i})= sup_{x_j\in X_{j}, j\neq i}\Big ( b(x_{1},\ldots,x_{m}) -\sum_{j\neq i}u_{j}(x_{j})\Big ). $$
\end{definition}
Since in our setting the surplus $b$ is locally Lipschitz  and semi-convex, it follows that for any $b$-conjugate $m$-tuple $(u_{1}, u_{2},\ldots, u_{m})$, $u_{i}$ is locally Lipschitz  and semi-convex for each $i$ \cite{McCann}.

We now introduce a duality theorem that shows the connection between \eqref{DP} and \eqref{KP}. We refer to \cite{Gangbo}\cite{Pass4} for a proof that solutions can be taken to be  $b$-conjugate; the remaining assertions can be found in \cite{Kellerer}.
\begin{theorem}
Assume $X_{i}$ is compact for every $i$ and let $spt(\mu)$ be the support of $\mu$. Hence, there exists a solution $\mu$ to the Kantorovich problem and a $b$-conjugate solution $(u_{1}, u_{2},\ldots,u_{m})$ to its dual. The minimum and maximum  values in \eqref{DP} and \eqref{KP} respectively agree, and  $\sum_{i=1}^{m}u_{i}(x_{i})=  b(x_{1},\ldots,x_{m})$ for each $(x_{1},\ldots,x_{m})\in spt(\mu)$.
\end{theorem}
\subsection{Some graph theory}
First, let us recall some definitions from Graph Theory. An \textit{undirected simple graph } $G$ is an ordered pair $(V(G), E(G))$, consisting of a finite \textit{set of vertices} $V(G)$ and a \textit{set of edges} $E(G)\subseteq\left\lbrace \{v,w\}:\text{ $v,w \in V(G)$ and $v\neq w$} \right\rbrace$. Throughout this work, every graph $G$ is an undirected simple graph. A \textit{trail} is a finite sequence $\left\lbrace\{v_{i_1}, v_{i_2}\}, \{v_{i_2}, v_{i_3}\},\ldots, \{v_{i_l}, v_{i_{l+1}}\}\right\rbrace $ of pairwise distinct edges which joins a sequence of vertices. A \textit{path} is a trail in which all vertices are distinct: $v_{i_{j}}\neq v_{i_k}$ for all $j \neq k$.  A \textit{cycle graph} is a trail in which the first and last vertex are the only one repeated: $\left\lbrace\{v_{i_1}, v_{i_2}\}, \{v_{i_2}, v_{i_3}\},\ldots, \{v_{i_l}, v_{i_{l+1}}\}\right\rbrace, $  where $\left\lbrace\{v_{i_1}, v_{i_2}\}, \{v_{i_2}, v_{i_3}\},\ldots, \{v_{i_{l-1}}, v_{i_{l}}\}\right\rbrace $  is a path and $ v_{i_{l+1}}=v_{i_1}$. A \textit{tree} is a graph where any two distinct vertices are connected by a unique path. A graph $G$ is \textit{connected} if for every $v,w \in V(G)$ there exists a path in the graph joining them.  We will denote by $I(V(G))$ the \textit{set of indices }of $V(G)$ (that is, for $V(G) =\{v_1,...,v_m\}$, $I(V(G))=\{1,2,...,m\}$) and $|V(G)|$ the cardinality of $V(G)$.
 %and its set of vertices $V(G)$ is a subset of $V=\{x_{1},\ldots, x_{m}\}$, where $x_{i}\in X_{i}$ for all $i=1, \ldots, m$, for some $m$.
\par
A \textit{subgraph}  $S$ of a graph $G$  is a graph whose sets of vertices and edges are subsets of $V(G)$ and $E(G)$ respectively. In this case, we call the graph $G\setminus S:= (V(G), E(G)\setminus E(S))$ the \textit{extraction} of $S$ from $G$. Note that if $G$ is complete and $V(G) =V(S)$, $G\setminus S$ coincides with the complement of $S$;  that is, $G\setminus S=\left( V(S), E(S^{c})\right)$, where $E(S^{c}):= \left\lbrace \{v,w\}:v,w\in V(S)\;\; and \;\;\{v,w\}\notin E(S)\right\rbrace.$
\par
 Given $v,w\in V(G)$, $v$ and $w$ are called \textit{adjacent} if $\{v,w\}\in E(G)$. The \textit{ open neighborhood} of a vertex $v$, denoted $N_{G}(v)$ (or simply $N(v)$ if there is not danger of confusion), is the set of vertices that are adjacent to $v$; that is,  $$N(v)= \big\{w\in V(G):   \{v,w\}\in E(G) \big\}.$$ 
 The \textit{closed neighborhood} of a vertex $v$, denoted $\overline{N}_{G}(v)$(or simply $\overline{N}(v)$), is the set $N(v)\cup \{v\}$.\\
%If for every $(x,y)\in V(G)\times V(G)$ with $x\neq y$ we have $\{x,y\}\in E(G)$, then $G$ is said to be \textit{complete}.  
 A graph $G$ is \emph{complete} if $N(v)=V(G)\setminus \{v\}$ for every $v\in V(G)$. A \textit{clique} $\tilde G =(V(\tilde G), E(\tilde G))$ of a graph  $G=(V(G), E(G))$ is a complete subgraph of $G$; that is,  $\tilde G$ is a subgraph of $G$ and satisfies $N_{\tilde G }(v)=V(\tilde G )\setminus \{v\}$ for every $v\in V(\tilde G )$ . A clique is \emph{maximal} if it is not a proper subgraph of any other clique of $G$.
% \marginpar{{\blue It seems that a clique is not a graph, it is a subset of $V(G)$ that induces a graph.\\
 %}} {\red  A set $Q\subseteq V(G)$ is called a \emph{clique} in $G$, if its induced subgraph is complete; that is, the subgraph  of $G$ with set of vertices $Q$ and set of edges $\left\lbrace \{x,y\}\in E(G): x,y \in Q \right\rbrace$ is complete.}  

 The \textit{union} of two given graphs $G_{1}$ and $G_{2}$ (denoted by $G_{1}\cup G_{2}$) is the graph with set of vertices $V(G_{1})\cup V(G_{2})$ and  edges $E(G_{1})\cup E(G_{2})$. A \textit{complete $k$-partite graph} $G$ is a graph whose set of vertices $V(G)$, can be partitioned into $k$ subsets $V_{1}, V_{2},\ldots, V_{k}$ such that for every $v\in V_{j}$, $N(v)=\bigcup_{\underset{\alpha\neq j}{\alpha=1}}^{k}V_{\alpha}$ for any fixed $j\in \{1, \ldots, k\}$.  A complete $k$-partite graph $G$ is denoted as $K_{m_{1},\ldots,m_{k}}$, where $|V_{j}|=m_{j}$ for every $j\in \{1, \ldots, k\}$.
\begin{remark}
Note that the definition of a graph implies $v\notin N(v)$ for every  $v\in V(G)$.  Furthermore, from now on, if $S$ is a subgraph of $G$ and $v\in V(S)$, we will write $N(v)$ for the open neighborhood of $v$ in $G$; that is, $N(v)=N_{G}(v)$. Similarly, $\overline{N}(v)=\overline{N}_{G}(v)$.
\end{remark}
The next two concepts are used to facilitate the description of our main results, established in Theorem \ref{Theorem 2} and \ref{Theorem 1}.
\begin{definition}\label{Def 1}  
 We say that a subset $A$ of $V(G)$  is the \emph{inner hub} of $G$, if $A=V(S_1)\bigcap V(S_2)$ for any two maximal cliques $S_1$ and $S_2$ of $G$.

%A collection of graphs $\{ S_{j}\}_{j=1}^{l}$ is called semi-disjoint, if for every $j$, $S_{j}$ is complete and there exists a subset $A$ such that $ V(S_{j})\cap V(S_{k})=A$ for all $j\neq k$, $j,k\in \{1, \ldots, l\}$.
\end{definition}
\begin{example}\label{Example 1}
The picture below shows the graph $G:=S_{1}\cup S_{2} \cup S_{3}$, where $S_1, S_2$ and $S_3$ are complete graphs with $V(S_{1})=\{v_{6}, v_{7},v_{8},v_{9}\}$, $V(S_{2})=\{v_{6}, v_{7},v_{8},v_{4},v_{5}\}$ and $V(S_{3})=\{v_{6}, v_{7},v_{8},v_{1},v_{2},v_{3},v_{10}\}$. Clearly, $\{S_{1}, S_{2},S_{3}\}$ is the collection of maximal cliques of $G$, and the inner hub of $G$ is the set formed by the vertices of the triangle colored blue; that is, $A=\{v_{6},v_{7},v_{8}\}$.
\end{example}
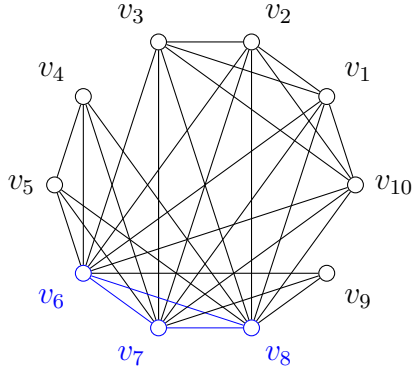
\begin{figure}[H]  
\begin{center}
\begin{tikzpicture}[x=2cm, y=2cm]
	\vertex (v1) at (36:1) [label=36:$v_{1}$]{};
	\vertex (v2) at (72:1) [label=72:$v_{2}$]{};
	\vertex (v3) at (108:1) [label=108:$v_{3}$]{};
	 \vertex (v4) at (144:1) [label=144:$v_{4}$]{};
	\vertex (v5) at (180:1) [label=180:$v_{5}$]{};
	{\blue \vertex (v6) at (216:1) [label=216:$v_{6}$]{};
	\vertex (v7) at (252:1) [label=252:$v_{7}$]{};
	\vertex (v8) at (288:1) [label=288:$v_{8}$]{};}
	\vertex (v9) at (324:1) [label=324:$v_{9}$]{};
	\vertex (v10) at (360:1) [label=360:$v_{10}$]{};
	\path 
		(v1) edge (v2)	
		(v1) edge (v3)
		(v1) edge (v6)
		(v1) edge (v7)
		(v1) edge (v8)
		(v1) edge (v10)
		(v2) edge (v3)
		(v2) edge (v6)
		(v2) edge (v7)
		(v2) edge (v8)
		(v2) edge (v10)
		(v3) edge (v6)
		(v3) edge (v7)
		(v3) edge (v7)
		(v3) edge (v8)
		(v3) edge (v10)
		(v4) edge (v5)
		(v4) edge (v6)
		(v4) edge (v7)
		(v4) edge (v8)
		(v5) edge (v6)
		(v5) edge (v7)
		(v5) edge (v8)
        (v6) edge[ultra thin, draw=blue,-](v7)
		(v6) edge [ultra thin, draw=blue,-](v8)
		(v6) edge (v9)
		(v6) edge (v10)
		(v7) edge [ultra thin, draw=blue,-](v8)
		(v7) edge (v9)
		(v7) edge (v10)
		(v8) edge (v9)
		(v8) edge (v10)	
		;
\end{tikzpicture}
 \caption{Graph $G=S_{1}\cup S_{2} \cup S_{3.}$} \label{fig:M1}  
\end{center}
\end{figure}
Not all graphs have an inner hub. Letting $\{ S_{j}\}_{j=1}^{l}$ be the maximal cliques of a graph $G$, it is clear that  $G =\bigcup_{j=1}^lS_j$; $A$ is the inner hub of $G$ if $ V(S_{j})\cap V(S_{k})=A$ for all $j\neq k$, $j,k\in \{1, \ldots, l\}$.  Note that we allow the inner hub $A$ to be the empty set, which is the case when $G$ is disconnected and each connected component is complete.  At the other extreme, we could have $A=G$, which is the case when $G$ is complete.\\
%\marginpar{BP: Each $S_i$ is a clique of $S$.  I don't like semi-disjoint.  $A$ could be called the hub.  What about just saying the collection has Hub $A$? }
%We call $A$ the \textit{intersection set }of $\{ S_{j}\}_{j=1}^{l}$.
% and $S_{1}, \ldots, S_{n}$ the branches of $S$. 

\begin{definition}\label{Def 2}
Let $\{ S_{1j}\}_{j=1}^{l_{1}}$ and $\{ S_{2j}\}_{j=1}^{l_{2}}$ be  the collection of maximal cliques of given graphs $G_1$ and $G_2$, respectively. Assume that $G_1$ and $G_2$ have inner hubs $A_{1}$ and $A_{2}$, respectively. We say the graphs $G_{1}$ and $G_{2}$ are glued on a clique if:
\begin{enumerate}
\item There are $j\in \{1, \ldots, l_{1}\}$ and $k\in\{1,\ldots, l_{2}\}$ such that $S_{1j}=S_{2k}$.
\item $V(G_{1})\cap V(G_{2})=V(S_{1j})$
\end{enumerate}
%We call $S_{1i}$ the intersection branch of $S_{1}$ and $S_{2}$.
\end{definition}
\begin{example}\label{Example 2}
The picture below shows the graph $S_{1}\cup S_{2} \cup S_{3}$, where $\{S_{1}, S_{2}, S_{3}\}$ is the collection of maximal cliques in Example \ref{Example 1}, glued with the graph $S_{1}^{\prime}\cup S_{2}^{\prime} \cup S_{3}^{\prime}\cup S_{4}^{\prime}$ on $S_2$, where $S_{1}^{\prime},S_{2}^{\prime},S_{3}^{\prime}$ and $S_{4}^{\prime}$  are complete graphs with $V(S_{1}^{\prime})=\{v_{4}, v_{5},v_{14}\}$, $V(S_{2}^{\prime})=\{v_{6}, v_{7},v_{8},v_{4},v_{5}\}=V(S_{2})$, $V(S_{3}^{\prime})=\{v_{4}, v_{5},v_{15},v_{16}\}$ and $V(S_{4}^{\prime})=\{v_{4}, v_{5},v_{11},v_{12},v_{13}\}$. Note that the inner hub of $S_{1}^{\prime}\cup S_{2}^{\prime} \cup S_{3}^{\prime}\cup S_{4}^{\prime}$, whose collection of maximal cliques is $\{ S_{j}^{\prime}\}_{j=1}^{4}$,  is formed by the vertices of the edges colored red; that is, $A^{\prime}=\{v_{4},v_{5}\}$.
\end{example}
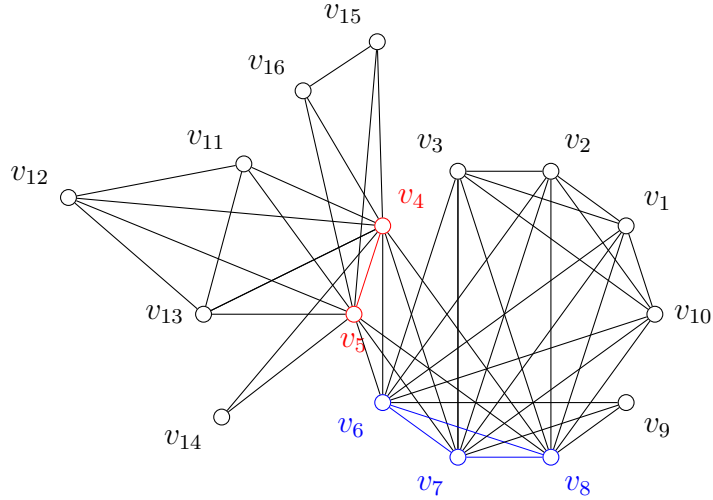
\begin{figure}[H]  
\begin{center}
\begin{tikzpicture}[x=2cm, y=2cm]
	\vertex (v1) at (36:1) [label=36:$v_{1}$]{};
	\vertex (v2) at (72:1) [label=72:$v_{2}$]{};
	\vertex (v3) at (108:1) [label=108:$v_{3}$]{};
	{\red \vertex (v4) at (144:1) [label=60:$v_{4}$]{};}
	\vertex (v15) at (115:2) [label=115:$v_{15}$]{};
   \vertex (v16) at (132:2) [label=132:$v_{16}$]{};
	\vertex (v11) at (150:2) [label=150:$v_{11}$]{};
	\vertex (v12) at (165:3) [label=165:$v_{12}$]{};
	\vertex (v13) at (180:2) [label=180:$v_{13}$]{};
	{\red\vertex (v5) at (180:1) [label=268:$v_{5}$]{};}
	\vertex (v14) at (200:2) [label=200:$v_{14}$]{};
	{\blue \vertex (v6) at (216:1) [label=216:$v_{6}$]{};
	\vertex (v7) at (252:1) [label=252:$v_{7}$]{};
	\vertex (v8) at (288:1) [label=288:$v_{8}$]{};}
	\vertex (v9) at (324:1) [label=324:$v_{9}$]{};
	\vertex (v10) at (360:1) [label=360:$v_{10}$]{};
	\path 
	    (v4) edge (v11)
		(v4) edge (v12)
		(v4) edge (v13)
		(v4) edge (v14)
		(v4) edge (v15)
		(v4) edge (v16)
		(v5) edge (v11)
		(v5) edge (v12)
		(v5) edge (v13)
		(v5) edge (v14)
		(v5) edge (v15)
		(v5) edge (v16)
		(v4) edge (v13)
		(v11) edge (v12)
		(v11) edge (v13)
		(v12) edge (v13)
		(v15) edge (v16)
		(v1) edge (v2)	
		(v1) edge (v3)
		(v1) edge (v6)
		(v1) edge (v7)
		(v1) edge (v8)
		(v1) edge (v10)
		(v2) edge (v3)
		(v2) edge (v6)
		(v2) edge (v7)
		(v2) edge (v8)
		(v2) edge (v10)
		(v3) edge (v6)
		(v3) edge (v7)
		(v3) edge (v7)
		(v3) edge (v8)
		(v3) edge (v10)
		(v4) edge [ultra thin, draw=red,-](v5)
		(v4) edge (v6)
		(v4) edge (v7)
		(v4) edge (v8)
		(v5) edge (v6)
		(v5) edge (v7)
		(v5) edge (v8)
        (v6) edge[ultra thin, draw=blue,-](v7)
		(v6) edge [ultra thin, draw=blue,-](v8)
		(v6) edge (v9)
		(v6) edge (v10)
		(v7) edge [ultra thin, draw=blue,-](v8)
		(v7) edge (v9)
		(v7) edge (v10)
		(v8) edge (v9)
		(v8) edge (v10)	
		;
\end{tikzpicture}
 \caption{Graph $\left(\bigcup_{j=1}^{3}S_{j}\right)\bigcup \left(\bigcup_{j=1}^{4}S_{j}^{\prime}\right)$} \label{fig:M2}  
\end{center}
\end{figure}

\subsection{Preliminary results connecting graph theory and multi-marginal optimal transport}

In this subsection, we establish some initial results connecting solutions of the multi-marginal optimal transport problem \eqref{KP} with cost \eqref{Main cost} and the structure of the corresponding graph.  These include a couple of very basic observations (Proposition \ref{prop: non ToSS}), as well as  a technical lemma which will be used throughout the paper (Lemma \ref{Lemma:1}).
\begin{proposition} \label{prop: non ToSS} Let $G$ be the graph corresponding to some $P \subseteq Q$ and $b$ the suplus \eqref{Main cost}.
\begin{enumerate}
		\item  Assume $G$ is not connected and let $x_i$ be any vertex such that there is no path between $x_1$ and $x_i$, and assume that $\mu_i$ is not a dirac mass. Then there exist non Monge solutions to 
		\eqref{KP}, and, if $\mu_1$ is not a dirac mass, the solution to \eqref{KP} is not unique.
		\item Assume $\{v_1, v_i\}$ is not an edge of $G$ for some $i$, and all the marginals are dirac measures except $\mu_1$ and $\mu_i$, with $\mu_1$ absolutely continuous with respect to Lebesgue measure. Then, there exist solutions of non-Monge form to \eqref{KP} and the solution to \eqref{KP} is not unique.

	\end{enumerate}
	\begin{proof}
%		For both the twist on splitting sets part of assertion 1, and assertion 2, it suffices to find marginals $\mu_1,\ldots,\mu_m$ such that $\mu_1$ is absolutely continuous with respect to Lebesgue measure, so that solutions $\gamma$ to \eqref{KP} are of non-Monge form and are non-unique, since if the surplus satisfies the twist on splitting sets condition, these examples would violate Theorem 3.1 in \cite{Kim}.
		Consider the first assertion. Let $G_1$ be the connected component of $G$ satisfying $v_1\in Z:=V(G_1)$, and $G_2$ the graph union of the other components of $G$, with $W:=V(G_2)$. Then the surplus \eqref{Main cost} takes the separable form:
		$$
		b(x_1,\ldots, x_m) =b_{Z}(x_{Z}) +b_{W}(x_{W}),
		$$
		where we decompose $x=(x_{Z},x_{W})$ into components $x_{Z}$ and $x_{W}$ whose indices of their coordinates lie in $I(Z)$ and $I(W)$, respectively, and $b_{Z}(x_{Z}) =\sum_{\{v_s,v_t\}\in E(G_1)}x_s\cdot x_t$, $b_{W}(x_{W}) =\sum_{\{v_s,v_t\}\in E(G_2)}x_s\cdot x_t$. Solutions to $\eqref{KP}$ are then exactly measures $\gamma$ whose projections $\gamma_Z$ and $\gamma_{W}$ onto the appropriate subspaces are optimal for the multi-marginal optimal transport problem with costs $b_Z$ and $b_{W}$, respectively, and the appropriate marginals.  In particular, the dependence structure between $\gamma_Z$ and $\gamma_W$ is completely arbitrary, and so, if $\mu_i$ is not a dirac mass for some $v_i \in W$, we immediately get the existence of non-Monge solutions (for instance, the product measure  $\gamma_Z \otimes\gamma_{W}$), and if in addition $\mu_1$ is  not a dirac mass, solutions are non-unique. 
		
		Turning to assertion 2, without loss of generality, assume $\{v_1, v_2\}$ is not an edge of $G$. Take $\mu_1$  be absolutely continuous with respect to Lebesgue measure, $\mu_2$ be any measure other than a dirac mass (so that $\mu_2$ charges at least two points) and let all other marginals be dirac masses, $\mu_i=\delta_{\bar x_i}$.  In this case, measures $\gamma$ whose marginals are the $\mu_i$ all take the form
		$\gamma= \sigma(x_1,x_2) \otimes \delta_{\bar x_3}\otimes \ldots\otimes  \delta_{\bar x_m}$, where $\sigma \in P(X_1 \times X_2)$ has marginals $\mu_1$ and $\mu_2$.  For any such $\gamma$, we have
		\begin{eqnarray*}
		\int_{X_1 \times X_2 \times\ldots\times X_m}b(x_1,\ldots,x_m)d\gamma(x_1,\ldots x_m)&=&\int_{X_1 \times X_2}b(x_1,x_2,\bar x_3,\ldots,\bar x_m)d\sigma(x_1,x_2)\\
		&=&\int_{X_1}b_1(x_1,\bar x_3,\ldots,\bar x_m)d\mu_1(x_1)\\
		&& +\int_{X_2}b_2(x_2,\bar x_3,\ldots,\bar x_m)d\mu_2(x_2)\\
		\end{eqnarray*}
	where $b_1(x_1,x_3,\ldots,x_m) =\sum_{\{v_s,v_t\}\in E(G)\; s,t\neq 2}x_s\cdot x_t$ and $b_2(x_2,x_3,\ldots,x_m) =\sum_{s\in I\left(N(v_2)\right)}x_2 \cdot x_s$.  Thus, the Kantorovich functional is independent of $\sigma$, and so any $\sigma$ with marginals $\mu_1$ and $\mu_2$ is optimal.  We conclude that solutions are non-unique, and can be of non-Monge form (as is the case when, for example, $\sigma =\mu_1 \times \mu_2$ is  the product measure).
	\end{proof}
	\end{proposition}
	 Let us now establish an immediate consequence of the above proposition. For this, let us invoke the main definitions in \cite{Kim}.
\begin{definition}
A set $S\subseteq \prod_{i=1}^{m}X_i$ is called a $b$-splitting set if there are Borel functions $u_i:X_i \mapsto \mathbb{R}$ such that 
\begin{equation*}
\sum_{i=1}^{m}u_i(x_{i})\geq b(x_1, \ldots,x_m)
\end{equation*}
for every $(x_1, \ldots, x_m)\in \prod_{i=1}^{m}X_i$, and whenever $(x_1, \ldots, x_m)\in S$ equality holds. 
\end{definition}
\begin{definition}\label{Art3:44}
Let $b$ be a continuous semi-convex surplus function. It is called twisted on $b$-splitting sets, whenever for each fixed $x_1\in X_1$ and $b$-splitting set $S\subseteq \{x_1\}\times X_2\times \ldots X_m$, the map 
$$(x_{2}, \ldots, x_{m})\mapsto D_{x_{1}}b(x_{1}^{0},x_2,\ldots, x_{m}) $$ is injective on
the subset of  $S$ where $D_{x_1}b(x_1, x_2, \ldots, x_m)$ exists. 
\end{definition}
\begin{remark}\label{Art3:45}
The main result in \cite{Kim} establishs that if $b$ is twisted on $b$-splitting sets, then every solution to $(\ref{KP})$ is induced by a map, whenever $\mu_1$ is absolutely continuous with respect to local coordinates. 
\end{remark}
\begin{corollary}
Under the hypothesis in any of assertion 1 or assertion 2 of Proposition \ref{prop: non ToSS}, the surplus $b$ is not twisted on $b$-splitting sets.
\begin{proof}
Suppose $b$ is twisted on $b$-splitting sets. From the above remark every solution to (KP) is induced by map; that is, every solution to (KP) is of Monge type. This clearly contradicts Proposition \ref{prop: non ToSS}, completing the proof of the corollary.
\end{proof}   
\end{corollary}

Clearly, in light of the first assertion there is no hope of obtaining Monge solution results for disconnected graphs (except in the trivial case when each $\mu_i$ with $x_i$ not connected to $\mu_1$ is a dirac mass, in which case the problem reduces to a problem on the connected component containing $x_1$.)  We therefore will focus on connected graphs throughout this paper.  On the other hand, our work in \cite{Pass5} suggests that at least for some surplus functions where $\{v_1, v_i\}$ is not an edge of $G$ for some $i$, unique Monge solutions may exist when extra regularity conditions on the marginals are imposed (even though the twist on splitting sets condition fails).   Our results in the following sections confirm that this is indeed the case.

The proofs of our main results will require the following technical lemma.
\begin{lemma}\label{Lemma:1}
Let $G$ be a graph, with $V(G)=\{v_{1}, \ldots, v_{m}\}$, and  $b(x_1,\ldots,x_m) = \sum_{\{v_s,v_t\}\in E(G)}x_s \cdot x_t$ be the surplus associated to $G$. %Let $\gamma$ be a solution to the Kantorovich problem with surplus $b$ and marginals $\mu_{1},\ldots, \mu_{m}$, and 
Let $(u_{1}, \ldots, u_{m})$ a $b$-conjugate $m$-tuple. Set $$ W:=\left\lbrace (x_{1}, \ldots, x_{m})\in X_{1}\times\ldots \times X_{m}: \sum_{i=1}^{m}u_{i}(x_{i})=  b(x_{1},\ldots,x_{m})\right\rbrace.$$ Fix $x_{1}^{0}\in X_{1}$ and for convenience of notation set $x_{1}^{1}=x_{1}^{2}=x_{1}^{0}$.  Let  $(x_{1}^{1}, x_{2}^{1}, \ldots, x_{m}^{1}), (x_{1}^{2}, x_{2}^{2}, \ldots, x_{m}^{2}) \in W$.
\begin{enumerate}
\item \label{Part 1} Assume there are sets  $V_{1}, V_{2}\subseteq V(G)$ such that $N(v_s)=V_{2}$  for every $s\in I(V_{1})$, and set 
\[y_{s}= \begin{cases} 
      x_{s}^{1} & \text{if}\quad s\in \{1, \ldots, m\}\setminus I(V_{1}) \\
      x_{s}^{2} & \text{if}\quad  s\in I(V_{1}). \\

   \end{cases}
\]
If 
\begin{equation}\label{eqn:1.2-21}
\sum_{{s\in I(V_{2})}}x_{s}^{1}=\sum_{{s\in  I(V_{2})}}x_{s}^{2},  
\end{equation}
then $y:=(y_{1}, y_{2}, \ldots, y_{m})\in W$.
%Assume $Du_{1}(x_{1}^{0})$ exists. If
%\begin{equation}\label{eqn:1.2-50}
%\sum_{{s\in I(E_{1})}}x_{s}^{1}=\sum_{{s\in  I(E_{1})}}x_{s}^{2},  
%\end{equation}
%and \eqref{eqn:1.2-21} holds true, $x_{s}^{1}=x_{s}^{2}$ on $I(E_{1})$.
\item \label{Part 2} For all $t\in \{1,\ldots, m\}$ we have 
\begin{equation}\label{eqn:1.2-5}
\left(x_{t}^{2}-x_{t}^{1}\right)\cdot\sum_{s\in I(N(v_t))}\left(x_{s}^{1}-x_{s}^{2}\right )\leq 0.
\end{equation}
%Furthermore, if equality holds in \eqref{eqn:1.2-5} we get
%\begin{equation}\label{eqn:1.2-34}
%x_{t}^{1}\in  \text{Argmax}\Big\{x_{t} \mapsto \Big (\sum_{s\in I(N(v_t))}x_{s}^{2}\Big )\cdot x_{t}-u_{t}(x_{t})\Big\}
%\end{equation}
%and 
%\begin{equation}\label{eqn:1.2-35}
%x_{t}^{2}\in  \text{Argmax}\Big\{x_{t} \mapsto \Big (\sum_{s\in I(N(v_t))}x_{s}^{1}\Big )\cdot x_{t}-u_{t}(x_{t})\Big\}.
%\end{equation}
%\item \label{Part 3} If there are sets $A,B\subset V(G)$ such that $C_{x_{s}}(G)=B$  for every $s\in I(A)$ and 
%\begin{equation}\label{eqn:1.2-22}
%\sum_{{s\in I(A)}}x_{s}^{1} + \sum_{{s\in I(B)}}x_{s}^{1} =\sum_{{s\in  I(A)}}x_{s}^{2} + \sum_{{s\in  I(B)}}x_{s}^{2},
%\end{equation}
\item \label{Part 4} If there exists $t\in \{1,\ldots, m\}$ such that 
\begin{equation}\label{eqn:1.2-1}
\sum_{s\in I(\overline{N}(v_t))}x_{s}^{1}=\sum_{s\in I(\overline{N}(v_t))}x_{s}^{2}, 
\end{equation}
 then $x_{t}^{1}=x_{t}^{2}$.
 \item \label{Part 5} Assume $x_{p}^{1}=x_{p}^{2}$ and $Du_{p}(x_{p}^{1})$ exists for some $p\in \{1, \ldots, m\}$. %Set 
%\begin{equation}
%W_{x_{0}}:= \left\lbrace (x_{2}, \ldots, x_{m}): (x_{0}, x_{2}, \ldots, x_{m}) \in W \right\rbrace,
%\end{equation}
%and let $( x_{2}^{1}, \ldots, x_{m}^{1}), ( x_{2}^{2}, \ldots, x_{m}^{2}) \in W_{x_{0}}$.
 \begin{enumerate}
 \item \label{Part a} For every $t\in \{2,\ldots, m\}\setminus \{p\}$ satisfying 
 \begin{equation}\label{eqn:1.2-6}
\overline{N}(v_p)=\overline{N}(v_t), 
 \end{equation}
 we have $x_{t}^{1}=x_{t}^{2}$.
 \item \label{Part b} Assume there are sets $F_{1}, F_{2}, F_{3}$ such that $F_{1}, F_{2}\subseteq N(v_p)$ and $N(v_s)=F_{2}\cup F_{3}$  for every $s\in I(F_{1})$. If $x_{s}^{1}=x_{s}^{2}$ for every $s\in I( N(v_p) \setminus F_{1}\cup F_{2})\cup I(F_{3})$,
%\begin{equation}\label{eqn:1.2-22}
%\sum_{{s\in I(E_{1})}}x_{s}^{1} + \sum_{{s\in I(E_{2})}}x_{s}^{1} =\sum_{{s\in  I(E_{1})}}x_{s}^{2} + \sum_{{s\in  I(E_{2})}}x_{s}^{2},
%\end{equation}
then $x_{s}^{1}=x_{s}^{2}$ for every $s\in I(F_{1})$.
 \end{enumerate}

% If 
%\begin{equation}\label{eqn:1.2-2}
%D_{x_{1}}b(x_{1}^{1}, x_{2}^{1}, \ldots, x_{m}^{1})=\sum_{s\in I(C_{x_{1}}(G))}x_{s}^{1}=\sum_{s\in I(C_{x_{1}}(G))}x_{s}^{2}=D_{x_{1}}b(x_{1}^{2}, x_{2}^{2}, \ldots, x_{m}^{2}),
%\end{equation}
\end{enumerate}
\end{lemma}
\begin{proof}
Since for every $s \in I(V_{1})$ we have $N(v_s)=V_{2}$, and  $v\notin N(v)$ for all $v\in V(G)$, we get $V_1\cap V_2=\emptyset$. Hence, we can write  $$b(x_{1},\ldots, x_{m})=g(x_{1}, \ldots, x_{m}) + \Big (\sum_{s\in I(V_{2})}x_{s}\Big )\cdot \Big (\sum_{s\in I(V_{1})}x_{s}\Big ), $$ where $g(x_{1}, \ldots, x_{m})$ does not depend on $\{x_{s}\}_{s\in I(V_{1})}$. Hence
$$\{x_{s}^{2}\}_{s\in I(V_{1})}\in  \text{Argmax}\Big\{\{x_{s}\}_{s\in I(V_{1})} \mapsto \Big (\sum_{s\in I(V_{2})}x_{s}^{2}\Big )\cdot \Big (\sum_{s\in I(V_{1})}x_{s}\Big )-\sum_{s\in I(V_{1})}u_{s}(x_{s})\Big\},$$
as  $(x_{1}^{2}, x_{2}^{2}, \ldots, x_{m}^{2})\in W$. Then, If \eqref{eqn:1.2-21} holds we get 
$$\{x_{s}^{2}\}_{s\in I(V_{1})}\in  \text{Argmax}\Big\{\{x_{s}\}_{s\in I(V_{1})} \mapsto \Big (\sum_{s\in I(V_{2})}x_{s}^{1}\Big )\cdot \Big (\sum_{s\in I(V_{1})}x_{s}\Big )-\sum_{s\in I(V_{1})}u_{s}(x_{s})\Big\},$$
which implies  $y\in W$, as $(x_{1}^{1}, x_{2}^{1}, \ldots, x_{m}^{1})\in W$. This completes the proof of part \ref{Part 1}.  Using the arguments of the previous proof, and taking $V_{1}=\{v_{t}\}$ and $V_{2}=N(v_t)$ for any fixed $t\in \{2, \ldots, m\}$, we deduce
$$x_{t}^{2}\in  \text{Argmax}\Big\{x_{t} \mapsto \Big (\sum_{s\in I(N(v_t))}x_{s}^{2}\Big )\cdot x_{t}-u_{t}(x_{t})\Big\}.$$ Similarly,
$$x_{t}^{1}\in  \text{Argmax}\Big\{x_{t} \mapsto \Big (\sum_{s\in I(N(v_t))}x_{s}^{1}\Big )\cdot x_{t}-u_{t}(x_{t})\Big\}.$$
Then,
\begin{equation}\label{eqn:1.2-4}
\Big (\sum_{s\in I(N(v_t))}x_{s}^{2}\Big )\cdot x_{t}^{1}-u_{t}(x_{t}^{1})\leq \Big (\sum_{s\in I(N(v_t))}x_{s}^{2}\Big )\cdot x_{t}^{2}-u_{t}(x_{t}^{2})
\end{equation}
and
\begin{equation}\label{eqn:1.2-3}
\Big (\sum_{s\in I(N(v_t))}x_{s}^{1}\Big )\cdot x_{t}^{2}-u_{t}(x_{t}^{2})\leq \Big (\sum_{s\in I(N(v_t))}x_{s}^{1}\Big )\cdot x_{t}^{1}-u_{t}(x_{t}^{1}).
\end{equation}

Adding \eqref{eqn:1.2-4} and \eqref{eqn:1.2-3}  (and eliminating terms) we obtain inequality \eqref{eqn:1.2-5}, completing the proof of the second part. The proof of part \ref{Part 4} follows immediately from part \ref{Part 2}, as if there exists $t\in \{2,\ldots, m\}$ satisfying \eqref{eqn:1.2-1} we get $$x_{t}^{1} + \sum_{s\in I(N(v_t))}x_{s}^{1}=x_{t}^{2} + \sum_{s\in I(N(v_t))}x_{s}^{2},$$
hence, $\sum_{s\in I(N(v_t))}\left(x_{s}^{1}-x_{s}^{2}\right )=x_{t}^{2}-x_{t}^{1}$. Substituting it into inequality \eqref{eqn:1.2-5} we get  $\Vert x_{t}^{2}-x_{t}^{1}\Vert^{2} \leq 0$; that is, $x_{t}^{2}=x_{t}^{1}$. To prove part \ref{Part 5}, first note that 
\begin{equation}\label{eqn:1.2-49}
\sum_{s\in I(N(v_p))}x_{s}^{1}=D_{x_{p}}b(x_{1}^{1},\ldots, x_{m}^{1})=Du_{p}(x_{p}^{1})=Du_{p}(x_{p}^{2})=D_{x_{p}}b(x_{1}^{2},\ldots, x_{m}^{2})=\sum_{s\in I(N(v_p))}x_{s}^{2}.
\end{equation}
%for any $( x_{2}^{1}, \ldots, x_{m}^{1}), ( x_{2}^{2}, \ldots, x_{m}^{2}) \in W_{x_{0}}$.
Hence, 
%\begin{equation}\label{eqn:1.2-7}
%\sum_{s\in I(C_{x_{1}}(G))}x_{s}^{1}=\sum_{s\in I(C_{x_{1}}(G))}x_{s}^{2}.
%\end{equation}
for any $t\in \{2,\ldots, m\}\setminus \{p\}$ satisfying \eqref{eqn:1.2-6} we obtain
\begin{flalign}
\sum_{s\in I\left( \overline{N}(v_t)\right)}x_{s}^{1}& =\sum_{s\in I\left( \overline{N}(v_p)\right)}x_{s}^{1} \nonumber\\
    &= \sum_{s\in I\left(\overline{N}(v_p)\right)}x_{s}^{2}\qquad\qquad \text{by \eqref{eqn:1.2-49} and the equality $x_{p}^{1}=x_{p}^{2}$.}\nonumber\\
    &= \sum_{s\in I\left( \overline{N}(v_t)\right)}x_{s}^{2}\nonumber\\
\end{flalign} 
 Then, by part \ref{Part 4} we conclude $x_{t}^{1}=x_{t}^{2}$, completing the proof of part \ref{Part a}. To prove part \ref{Part b} observe that $F_{1}\cap F_{2}=\emptyset$, as  $N(v_s)=F_{2}\cup F_{3}$ and $v_{s}\notin N(v_s)$  for every $s\in I(F_{1})$. Then, from \eqref{eqn:1.2-49} we get
\begin{equation*}
\sum_{{s\in I(F_{1})}}x_{s}^{1} + \sum_{{s\in I(F_{2})}}x_{s}^{1} + \sum_{{s\in I(N(v_p)\setminus F_{1}\cup F_{2})}}x_{s}^{1} =\sum_{{s\in  I(F_{1})}}x_{s}^{2} + \sum_{{s\in  I(F_{2})}}x_{s}^{2} + \sum_{{s\in I( N(v_p) \setminus F_{1}\cup F_{2})}}x_{s}^{2},
\end{equation*} 
 as $F_{1}, F_{2}\subseteq N(v_p)$. Since $x_{s}^{1}=x_{s}^{2}$ for every $s\in I( N(v_p) \setminus F_{1}\cup F_{2})$, the above equality reduces to
  \begin{equation}\label{eqn:1.2-22}
\sum_{{s\in I(F_{1})}}x_{s}^{1} + \sum_{{s\in I(F_{2})}}x_{s}^{1} =\sum_{{s\in  I(F_{1})}}x_{s}^{2} + \sum_{{s\in  I(F_{2})}}x_{s}^{2},
\end{equation} 
and applying part \ref{Part 2} we get 
%Also, if equality holds in \eqref{eqn:1.2-5}, we must have equality in \eqref{eqn:1.2-3} and \eqref{eqn:1.2-4}, and then, we get \eqref{eqn:1.2-34} and \eqref{eqn:1.2-35}, 
\begin{equation*}
\left(x_{t}^{2}-x_{t}^{1}\right)\cdot\sum_{s\in I(F_{2}\cup F_{3})}\left(x_{s}^{1}-x_{s}^{2}\right )\leq 0,
\end{equation*}
for any $t\in I(F_{1})$. Summing over $t\in I(F_{1})$ and using the equalities $x_{s}^{1}=x_{s}^{2}$ on $I(F_{3})$, we obtain
$$\sum_{t\in I(F_{1})}\left(x_{t}^{2}-x_{t}^{1}\right )\cdot\sum_{s\in I(F_{2})}\left(x_{s}^{1}-x_{s}^{2}\right )\leq 0.$$
Furthermore, by \eqref{eqn:1.2-22} we get $\sum_{t\in I(F_{1})}\left(x_{t}^{2}-x_{t}^{1}\right )=\sum_{s\in I(F_{2})}\left(x_{s}^{1}-x_{s}^{2}\right )$. Substituting it into the above inequality we get
%\Vert\sum_{t\in I(E_{1})}\left(x_{t}^{2}-x_{t}^{1}\right )\Vert^{2} \leq 0 \quad \text{and}\quad 
 $$\Vert\sum_{s\in I(F_{2})}\left(x_{s}^{1}-x_{s}^{2}\right )\Vert^{2}\leq 0;  \quad \text{that is,}\quad$$
% \sum_{{t\in I(E_{1})}}x_{t}^{1}=\sum_{{t\in  I(E_{1})}}x_{t}^{2} \quad \text{and}\quad 
\begin{equation}\label{eqn:1.2-33}
\sum_{{s\in I(F_{2})}}x_{s}^{1}=\sum_{{s\in  I(F_{2})}}x_{s}^{2}.
\end{equation}
Now, fix $t\in I(F_{1})$ and set  $V_{1}=F_{1}\setminus \{v_{t}\}$, $V_{2}=F_{2}\cup F_{3}$ and $y=(y_{1}, y_{2}, \ldots, y_{m})$ such that
\[y_{s}= \begin{cases} 
      x_{s}^{1} & \text{if}\quad s\in \{1, \ldots, m\}\setminus  I(V_{1}) \\
      x_{s}^{2} & \text{if}\quad  s\in I(V_{1}).\\
    
   \end{cases}
\]

Since  $x_{s}^{1}=x_{s}^{2}$ on $I(F_{3})$, \eqref{eqn:1.2-33} can be written as $\sum_{{s\in I(F_{2}\cup F_{3})}}x_{s}^{1}=\sum_{{s\in  I(F_{2}\cup F_{3})}}x_{s}^{2}$. Therefore, by part \ref{Part 1} we get $y\in W$, as $N(v_s)=V_{2}$ for every $s\in I(V_{1})$. Hence, 
$$\sum_{s\in I(N(v_p))}y_{s}=Du_{p}(y_{p})=Du_{p}(x_{p}^{1})=Du_{p}(x_{p}^{2})=\sum_{s\in I(N(v_p))}x_{s}^{2},$$
or equivalently,
\begin{equation*}
y_{t} + \sum_{{s\in I(F_{1})\setminus \{t\}}}y_{s} + \sum_{{s\in I(F_{2})}}y_{s}+ \sum_{{s\in I( N(v_p) \setminus F_{1}\cup F_{2})}}y_{s} = x_{t}^{2} + \sum_{{s\in  I(F_{1})\setminus \{t\}}}x_{s}^{2} + \sum_{{s\in  I(F_{2})}}x_{s}^{2} + \sum_{{s\in I( N(v_p) \setminus F_{1}\cup F_{2})}}x_{s}^{2}.
\end{equation*} 
From \eqref{eqn:1.2-33}, construction of $y$ and the equalities $x_{s}^{1}=x_{s}^{2}$ on $I( N(v_p) \setminus F_{1}\cup F_{2})$, we get $x_{t}^{1}=x_{t}^{2}$, completing the proof of part \ref{Part b}. 
\end{proof}
\section{Monge solutions under extraction of graphs}\label{Section 3}
%Note that the results obtained in Section 3 do not encapsulate the surplus associated to the graph cycle with vertices $\{x_{1}, x_{2}, x_{3}, x_{4}\}$ (see Figure \ref{cycle cost}).

%A natural approach to get the graph in Figure \ref{cycle cost} is extracting the edges $\{x_{1}, x_{3}\}$ and $\{x_{2}, x_{4}\}$ (which are complete graphs) from $C_{4}$. It turns out that the extraction operation is a feasible operation to get graphs providing unique Monge solutions. 
The main theorem of this section establishs that, roughly speaking, the extraction from $C_{m}$ of a subgraph with an inner hub provides a unique Monge solution, possibly under an additional regularity condition on one of the marginals. 

We will present several examples of graphs obtained in this way later on, but for now we mention that the graph in Figure \ref{cycle cost}-a is obtained by extracting the edges $\{x_{1}, x_{3}\}$ and $\{x_{2}, x_{4}\}$ from $C_{4}$, which can be interpreted as maximal cliques of the graph with edges $\{x_{1}, x_{3}\}$ and $\{x_{2}, x_{4}\}$, and inner hub $A=\emptyset$. 
\subsection{Monge solutions}
We now state and prove our first main result.
\begin{theorem}\label{Theorem 2}
Let $\{ S_{j}\}_{j=1}^{l}$ be the collection of maximal cliques of a given subgraph $S$ of $C_m$ with inner hub $A$, for some $m\in \mathbb{N}$. Let $G:=C_{m}\setminus S$ be connected, $b$ the surplus function associated to $G$ and $\mu_{i}$ be probability measures over $X_{i}$, $i=1,\ldots, m$, with $\mu_{1}$ absolutely continuous with respect to  $\mathcal{L}^{n}$. Assume that one of the following conditions is met:
\begin{enumerate}[label=(\roman*)]
\item $v_{1}\in V(G)\setminus V(S)$,
\item  There exists $p\in I(N_{G}(v_1))$ such that $ A\subseteq N_{G}(v_p)$, with $\mu_{p}$ is absolutely continuous with respect to  $\mathcal{L}^{n}$, and, if $S$ is not complete, $v_{1}\notin A$.
\end{enumerate}
Then every solution to the Kantorovich problem (KP) is induced by a map. 
\end{theorem}
\begin{proof}
Let $\gamma$ be a solution to the Kantorovich problem with surplus $b$ and $(u_{1}, \ldots, u_{m})$ a $b$-conjugate solution to its dual. Consider:
\begin{align*}
\widetilde{W}&=\Big\{ (x_{1},\ldots, x_{m}): Du_{1}(x_{1})\quad\text{exists,}\quad \text{and} \quad\sum_{i=1}^{m}u_{i}(x_{i})=b(x_{1},\ldots, x_{m})\Big\}.
\end{align*}
The function $u_{1}$ is differentiable $\mathcal{L}^{n}$-a.e, as it is Lipschitz continuous. Hence, it is differentiable $\mu_{1}$ a.e,  as $\mu_{1}$ is absolutely continuous. It follows that $\gamma (\widetilde{W})=1$. Similarly, if in addition there exists $p\in \{2, \ldots, m\}$ such that $\mu_{p}$ is absolutely continuous with respect to  $\mathcal{L}^{n}$, we get $u_{p}$ is differentiable $\mu_{p}$ a.e and $\gamma (\widetilde{W_{p}})=1$, where 
\begin{align*}
\widetilde{W_{p}}&=\Big\{ (x_{1},\ldots, x_{m}): Du_{1}(x_{1})\quad \text{and}\quad Du_{p}(x_{p})\quad \text{exist,}\quad \text{and} \quad\sum_{i=1}^{m}u_{i}(x_{i})=b(x_{1},\ldots, x_{m})\Big\}.
\end{align*}
Fix $x_{1}^{0}\in spt(\mu_{1})$, where $u_{1}(x_{1})$ is differentiable, and $(x_{2}^{0},\ldots, x_{m}^{0})$ such that $(x_{1}^{0},\ldots, x_{m}^{0})\in \widetilde{W}$.
% We will show that there exists a unique $(m-1)$-tuple $(x_{2}^{0},\ldots, x_{m}^{0})$ such that $(x_{1}^{0},\ldots, x_{m}^{0})\in \widetilde{W}$.
 Note that  %Since the map $x_{1}\mapsto b(x_{1},x_{2}^{0},\ldots, x_{m}^{0})$ is convex and satisfy $$b(x_{1},x_{2}^{0},\ldots, x_{m}^{0})\leq u_{1}(x_{1}) + \sum_{i=2}^{m}u_{i}(x_{i}^{0})$$ for every $x_{1}$, $\partial_{x_{1}}b(x_{1},x_{2}^{0},\ldots, x_{m}^{0})$ exists everywhere, where  $\partial_{x_{1}}b(x_{1},x_{2}^{0},\ldots, x_{m}^{0})$ denotes the subdifferential of $b$ with respect to $x_{1}$. It follows from the definition of the set $W$ that $$\partial_{x_{1}}b(x_{1}^{0},x_{2}^{0},\ldots, x_{m}^{0})\subseteq \partial_{x_{1}}u_{1}(x_{1}^{0})=\left\lbrace D_{x_{1}}u_{1}(x_{1}^{0})\right\rbrace;$$ 
$b$ is differentiable with respect to $x_{1}$ at $(x_{1}^{0},\ldots, x_{m}^{0})$ and it satisfies 
\begin{equation}\label{Theorem 1.2-1}
Du_{1}(x_{1}^{0})= D_{x_{1}}b(x_{1}^{0},\ldots, x_{m}^{0}).
\end{equation}
\par
We will show that the map $$(x_{2},\ldots, x_{m})\mapsto D_{x_{1}}b(x_{1}^{0},x_{2}\ldots, x_{m})$$ is injective on $\widetilde{W}_{x_{1}^{0}}:=\left\lbrace( x_{2}, \ldots, x_{m}): (x_{1}^{0}, x_{2}, \ldots, x_{m})\in \widetilde{W}\right\rbrace$, if $v_{1}\in V(G)\setminus V(S)$, or on  $\widetilde{W}_{x_{1p}^{0}}:=\left\lbrace( x_{2}, \ldots, x_{m}): (x_{1}^{0}, x_{2}, \ldots, x_{m})\in \widetilde{W_{p}}\right\rbrace$, if there exists $p \in I \left( N(v_1)\right)$ such that $\mu_{p}$ is absolutely continuous with respect to  $\mathcal{L}^{n}$ and $A\subseteq N(v_p)$; this will imply that the equation \eqref{Theorem 1.2-1} defines $(x_{2}^{0},\ldots, x_{m}^{0})$ uniquely from $x_{1}^{0}$, which will complete the proof. Let $(x_{1}^{0},x_{2}^{1},\ldots, x_{m}^{1}),(x_{1}^{0},x_{2}^{2},\ldots, x_{m}^{2})\in \widetilde{W}$ and assume 
\begin{equation}\label{eqn:1.2-23}
D_{x_{1}}b(x_{1}^{0},x_{2}^{1},\ldots, x_{m}^{1})=\sum_{s\in I(N(v_1))}x_{s}^{1}=\sum_{s\in I(N(v_1))}x_{s}^{2}=D_{x_{1}}b(x_{1}^{0},x_{2}^{2},\ldots, x_{m}^{2}).
\end{equation}
We want to prove $x_{s}^{1}=x_{s}^{2}$ for every $s\in \{2, \ldots, m\}$.\\
  Set $x_{1}^{0}:=x_{1}^{1}:=x_{1}^{2}$ and $B_{j}=V(S_{j})\setminus A$, with $j\in \{1, \ldots, l\}$. First, note that $S=\bigcup_{j=1}^{l}S_j $ and 
\begin{equation}\label{eqn:1.2-41}
N(v_s)=V(G)\setminus \{v_{s}\}\quad \text{for any $s\in I(V(G)\setminus V(S))$, }\quad \quad\qquad
\end{equation}
\begin{equation}\label{eqn:1.2-42}
N(v_s)=V(G)\setminus V(S_{j})\quad\text{for any $s\in I(B_{j}),\quad  j\in\{1,\ldots,l\}$},\quad 
\end{equation}
\begin{equation}\label{eqn:1.2-26}
 N(v_s)=V(G)\setminus V(S), \quad\text{for any $s\in I(A)$}.\qquad\qquad\qquad\quad\quad\enspace
\end{equation}
Let us consider two cases:
\begin{enumerate}[label=\textbf{Case \arabic*.}]

\item Assume $v_{1}\in V(G)\setminus V(S)=\{v_{1}, \ldots,v_{m}\}\setminus V(S)$, then by \eqref{eqn:1.2-41} we get 
$\overline{N}(v_1)=V(G)=\overline{N}(v_s)$  for any $s\in I(V(G)\setminus V(S))\setminus \{1\}$. It follows from part \ref{Part a} of Lemma \ref{Lemma:1} that
\begin{equation}\label{eqn:1.2-24}
x_{s}^{1}=x_{s}^{2}\quad \text{for all}\quad s\in I(V(G)\setminus V(S))\setminus \{1\}.
\end{equation}
Fix $j\in \{1, \ldots, l\}$, and let us consider two sub-cases:
%By \eqref{eqn:1.2-23} we can write 
%\begin{equation}\label{eqn:1.2-25}
%\sum_{s\in I(V(S_{j}))}x_{s}^{2} + \sum_{s\in I\left( V(G)\setminus V(S_{j})\right)}x_{s}^{2}=\sum_{s\in I(V(S_{j}))}x_{s}^{1}+\sum_{s\in I\left( V(G)\setminus V(S_{j})\right)}x_{s}^{1}.
%\end{equation}
\begin{enumerate}
\item \label{Sub-case a} If $A=\emptyset$, then $V(S_{j})=B_{j}$. By defining $F_{1}= B_{j}$ and $F_{2}= \left( V(G)\setminus B_{j}\right) \setminus \{v_{1}\}$, we get
$ F_{1}\cup F_{2}=V(G)\setminus \{v_{1}\}=N(v_1)$. Also, from \eqref{eqn:1.2-42} we have  $N(v_s)= F_{2}\cup F_{3} $ for all $s\in I(F_{1})$, where $F_{3}=\{v_{1}\}$. Then, we can apply part \ref{Part b} of Lemma \ref{Lemma:1} to get $x_{s}^{1}=x_{s}^{2}$ for every $s\in I(B_{j})$; that is, $x_{s}^{1}=x_{s}^{2}$ on $\bigcup_{j=1}^{l}I(B_{j})=\bigcup_{j=1}^{l}I(V(S_{j}))=I(V(S))$. Combining this result with \eqref{eqn:1.2-24} we get $x_{s}^{1}=x_{s}^{2}$ on $I(V(G))\setminus \{1\}=\{2, \ldots, m\}$. This completes the proof of sub-case (a).
\item Assume $A\neq\emptyset$. By setting $V_{2}=V(G)\setminus V(S)$  and $V_{1}=A$ we can use \eqref{eqn:1.2-24} to get equality \eqref{eqn:1.2-21}, and then, by \eqref{eqn:1.2-26} and part \ref{Part 1} of Lemma \ref{Lemma:1} we get $y:=(y_{1}, y_{2}, \ldots, y_{m})\in \widetilde{W}$, where

\[y_{s}= \begin{cases} 
      x_{s}^{1} & \text{if}\quad s\in \{1, \ldots, m\}\setminus I(A) \\
      x_{s}^{2} & \text{if}\quad  s\in I(A). \\

   \end{cases}
\]
Now, set $F_{1}= B_{j}$, $F_{2}= \left( V(G)\setminus V(S_{j})\right) \setminus \{v_{1}\}$ and $F_{3}=\{v_{1}\}$. Clearly, $F_{1}, F_{2}\subseteq N(v_1) =V(G)\setminus \{v_{1}\}$ and $F_{1}\cup F_{2}=V(G)\setminus \left( A\cup \{v_{1}\}\right)$, then $N(v_1)\setminus \left( F_{1}\cup F_{2}\right)=A.$
Furthermore, $y$ and $(x_{1}^{0},x_{2}^{2},\ldots, x_{m}^{2})$ trivially satisfies $y_{s}=x_{s}^{2}$ on $I(A)$, and  by \eqref{eqn:1.2-42}, $N(v_s)= F_{2}\cup F_{3}$ for every $s\in I(F_{1})$. Hence, by part \ref{Part b} of Lemma \ref{Lemma:1} we get $x_{s}^{1}=y_{s}=x_{s}^{2}$ on $I(B_{j})$, which proves that, using \eqref{eqn:1.2-24} and the equality $\bigcup_{j=1}^{l}B_{j}=V(S)\setminus A$, $x_{s}^{1}=x_{s}^{2}$ on $I(V(G)\setminus A)\setminus \{1\}$. We combine this result with \eqref{eqn:1.2-26} and  part \ref{Part b} of Lemma \ref{Lemma:1} to get $x_{s}^{1}=x_{s}^{2}$ on $I(A)$; all the conditions needed to apply this part of the lemma are trivially satisfied by setting $F_{1}=A$, $F_{2}= \left(V(G)\setminus V(S)\right)\setminus \{v_{1}\}$, $F_{3}=\{v_{1}\}$ and $p=1$. We conclude  that $x_{s}^{1}=x_{s}^{2}$ on $I(V(G)\setminus \{v_{1}\})=\{ 2,\ldots, m\}$, completing the proof of sub-case (b).
%$$\sum_{s\in I(C_{x_{1}}(G))}y_{s}=Du_{1}(x_{1}^{0})=\sum_{s\in I(C_{x_{1}}(G))}x_{s}^{2},$$
%the above expression may be written as 
%\begin{equation*}
%\sum_{s\in I(B_{i})}x_{s}^{1} + \sum_{s\in I\left( V(G)\setminus V(S_{i})\right)}x_{s}^{1}=\sum_{s\in I(B_{i})}x_{s}^{2}+\sum_{s\in I\left( V(G)\setminus V(S_{i})\right)}x_{s}^{2}.
%\end{equation*}
%We then use \eqref{eqn:1.2-42} and part \eqref{Part 5}(ii) of Lemma \ref{Lemma:1} to get $x_{s}^{1}=x_{s}^{2}$ on $I(B_{i})$, which prove, using \eqref{eqn:1.2-24}, that $x_{s}^{1}=x_{s}^{2}$ on $I(V(G)\setminus A)$.

 %Using this, \eqref{eqn:1.2-23} and the equality $C_{x_{1}}(G)=V(G)\setminus \{x_{1}\}$ we can write
%\begin{equation*}
%\sum_{s\in I(A)}x_{s}^{2} + \sum_{s\in I(V(G)\setminus V(S))}x_{s}^{2}=\sum_{s\in I(A)}x_{}^{1} + \sum_{s\in I(V(G)\setminus V(S))}x_{}^{1}.
%\end{equation*}
%\begin{equation}\label{eqn:1.2-25}
%\sum_{s\in I(B_{i})}x_{s}^{2} + \sum_{s\in \bigcup_{\underset{j\neq i}{j=1}}^{n} I(B_{j})}x_{s}^{2}=\sum_{s\in I(B_{i})}x_{s}^{1}+\sum_{s\in \bigcup_{\underset{j\neq i}{j=1}}^{n} I(B_{j})}x_{s}^{1},
%\end{equation}
%Hence, from \eqref{eqn:1.2-26} and part \eqref{Part 5}(ii) of Lemma \ref{Lemma:1} we immediately deduce $x_{s}^{1}=x_{s}^{2}$ on $I(A)$, completing the proof for case 1.
\end{enumerate}
This completes the proof of case 1.
\item Assume $v_1\in V(S)$ and let $p\in I(N(v_1))$ be such that $\mu_{p}$ is absolutely continuous with respect to  $\mathcal{L}^{n}$ and $ A\subseteq N(v_p)$. Assume $Du_{p}(x_{p}^{1})$ and $Du_{p}(x_{p}^{2})$ exist. If $S$ is complete, $S_j=S$ for all $j=1,\ldots, l$, and so, $A=V(S)$ . Using \eqref{eqn:1.2-23} and \eqref{eqn:1.2-26} we obtain

\begin{equation}\label{eqn:1.2-70}
\sum_{s\in I\left( V(G)\setminus A\right)}x_{s}^{1}=\sum_{s\in I\left( V(G)\setminus A\right)}x_{s}^{2}. 
\end{equation}
Also, using \eqref{eqn:1.2-26} and part 1 of Lemma 2.1 we get  $z:=(z_{1}, \ldots, z_{m})\in \widetilde{W_{p}}$, where
\begin{equation}\label{eqn:1.2-53}
z_{s}= \begin{cases} 
      x_{s}^{1} & \text{if}\quad s\in \{1, \ldots, m\}\setminus I(A) \\
      x_{s}^{2} & \text{if}\quad  s\in I(A). \\

   \end{cases}
   \end{equation}
   
Fix $t\in I(V(G)\setminus A)$. Then
\begin{flalign*}
\sum_{s\in I\left(\overline{N}(v_t)\right)}z_{s}& = z_{t} + \sum_{s\in I\left(N(v_t)\right)}z_{s}\nonumber\\
&=z_{t} + \sum_{s\in I\left(V(G)\setminus \{v_{t}\}\right)}z_{s}\qquad\qquad\qquad\qquad \text{by \eqref{eqn:1.2-41}}  \nonumber\\
&=\sum_{s\in I\left(V(G)\right)}z_{s}\nonumber\\
  &=\sum_{s\in I\left(A\right)}z_{s}+\sum_{s\in I\left(V(G)\setminus A\right)}z_{s}\nonumber\\
   &=\sum_{s\in I\left(A\right)}x_{s}^{2}+\sum_{s\in I\left(V(G)\setminus A\right)}x_{s}^{1}\qquad\qquad\qquad \text{by construction of $z$}\nonumber\\
    &= \sum_{s\in I\left(A\right)}x_{s}^{2}+\sum_{s\in I\left(V(G)\setminus A\right)}x_{s}^{2}\qquad\qquad\qquad \text{by \eqref{eqn:1.2-70}} \nonumber\\
    &=\sum_{s\in I\left(V(G)\right)}x_{s}^{2}\nonumber\\
    &= x_{t}^{2} + \sum_{s\in I\left(V(G)\setminus \{v_{t}\}\right)}x_{s}^{2}\nonumber\\
    &=x_{t}^{2} + \sum_{s\in I\left(N(v_t)\right)}x_{s}^{2}\\
    &= \sum_{s\in I\left(\overline{N}(v_t)\right)}x_{s}^{2}.
\end{flalign*}
It follows that $z_{s}=x_{s}^{2}$ on $I(V(G)\setminus A)$,  by part \ref{Part 4} of Lemma \ref{Lemma:1}; that is,
\begin{equation}\label{eqn:1.2-71}
x_{s}^{1}=x_{s}^{2}\;\;\text{on}\;\; I(V(G)\setminus A),
\end{equation}

  by construction of $z$. Now, to prove  that $x_{s}^{1}=x_{s}^{2}$ on $I(A)$ we use part \ref{Part b} of Lemma \ref{Lemma:1}. For this, set $F_1=A$, $F_2=V(G)\setminus (A\cup \{v_p\})$ and $F_3=\{v_p\}$. Since $v_p\in N(v_1)=V(G)\setminus V(S)$, hence $N(v_p)=V(G)\setminus \{v_p\}$ by \eqref{eqn:1.2-41}, and so $F_1, F_2\subseteq N(v_p)$. Note that $F_1\cup F_2=N(v_p)$ and by \eqref{eqn:1.2-26}, $F_2\cup F_3=V(G)\setminus A=N(v_s)$ for every $s\in I(F_1)$. Also, from \eqref{eqn:1.2-71}, $x_{p}^{1}=x_{p}^{2}$. This allow us to apply part \ref{Part b} of Lemma \ref{Lemma:1} to get $x_{s}^{1}=x_{s}^{2}$ on $I(A)$. Hence,  $x_{s}^{1}=x_{s}^{2}$ on $I(V(G))=\{1, \ldots,m\}$.\\
Let us know assume that $S$ is not complete, then $v_{1}\notin A$ by assumption, which implies that $v_{1}\in B_{k}$ for some $k\in \{ 1, \ldots, l\}$. We first claim that   $x_{s}^{1}=x_{s}^{2}$ on $\bigcup_{\underset{j\neq k}{j=1}}^{l}I(B_{j})= I(V(S)\setminus V(S_{k}))$. Indeed, from  \eqref{eqn:1.2-42} and \eqref{eqn:1.2-23} we get
\begin{equation}\label{eqn:1.2-27}
\sum_{s\in I\left( V(G)\setminus V(S_{k})\right)}x_{s}^{1}=\sum_{s\in I\left( V(G)\setminus V(S_{k})\right)}x_{s}^{2}.
\end{equation}
It follows that, by setting $V_{1}= B_{k}$ and $V_{2}=  V(G)\setminus V(S_{k})$, we can use \eqref{eqn:1.2-42} and part \ref{Part 1} of  Lemma \ref{Lemma:1} to get $y:=(y_{1}, \ldots, y_{m})\in \widetilde{W_{p}}$, where
\begin{equation}\label{eqn:1.2-53}
y_{s}= \begin{cases} 
      x_{s}^{1} & \text{if}\quad s\in \{1, \ldots, m\}\setminus I(B_{k}) \\
      x_{s}^{2} & \text{if}\quad  s\in I(B_{k}). \\

   \end{cases}
   \end{equation}
Fix $j\in \{1, \ldots, l\}$, with $j\neq k$. Set $F_{1}=B_{j}$, $F_{2}= V(G)\setminus \left(V(S_{k})\cup B_{j}\right)$ and $F_{3}= B_{k}$. Note that from  \eqref{eqn:1.2-42} we get  $F_{1}\cup F_{2}= V(G)\setminus V(S_{k})=N(v_1)$ and $F_{2}\cup F_{3}= V(G)\setminus V(S_{j})=N(v_s)$, for any $s\in I(F_{1})$. Since $y$ and $(x_{1}^{0},x_{2}^{2},\ldots, x_{m}^{2})$ satisfies $y_{s}=x_{s}^{2}$ on $I(F_{3})$, we can apply part \ref{Part b} of Lemma \ref{Lemma:1} to get $x_{s}^{1}=y_{s}=x_{s}^{2}$ on $I(B_{j})$; that is, %Furthermore,. Since $\{ S_{j}\}_{j=1}^{l}$ is semi-disjoint, $F_{1}\cap B_{i}=\emptyset$, and $F_{2}\cap B_{i}=\emptyset$, as $B_{i}\subseteq V(S_{i})$. Then, $F_{1}, F_{2}\subseteq C_{x_{1}}(G) =V(G)\setminus V(S_{i})$ and 
%\begin{equation}\label{eqn:1.2-46}
 %\text{we can choose}\quad  x_{s}^{1}=x_{s}^{2}\quad \text{for every}\quad s\in I(B_{i}),
%\end{equation}
 %and by \eqref{eqn:1.2-27},
%\begin{equation}\label{eqn:1.2-28}
%\sum_{s\in I\left( V(G)\setminus A\right)}x_{s}^{2}=\sum_{s\in I\left( V(G)\setminus A\right)}x_{s}^{1},
%\end{equation}
%or equivalently,
%\begin{equation*}
%\sum_{s\in I(B_{j})}x_{s}^{2} + \sum_{s\in I\left( V(G)\setminus V(S_{j})\right)}x_{s}^{2}=\sum_{s\in I(B_{j})}x_{s}^{1}+\sum_{s\in I\left( V(G)\setminus V(S_{j})\right)}x_{s}^{1},
%\end{equation*}
%for any fixed $j\in \{1, \ldots, l\}$, with $j\neq i$. Since \eqref{eqn:1.2-42} holds true for $j$, we can apply part \eqref{Part 3}(ii) of Lemma \ref{Lemma:1} to get $x_{s}^{1}=x_{s}^{2}$ for every $s\in I(B_{j})$; that is,
\begin{equation}\label{eqn:1.2-43}
 x_{s}^{1}=x_{s}^{2}\quad\text{on}\quad \bigcup_{\underset{j\neq k}{j=1}}^{l}I(B_{j})= I(V(S)\setminus V(S_{k})).
\end{equation}
Next, as in Case 1, let us consider the following sub-cases:
\begin{enumerate}
\item If $A= \emptyset$, then $V(S_{k})=B_{k}$. Also, for any $t\in I(V(G)\setminus V(S))$ we get
\begin{flalign*}
\sum_{s\in I\left(\overline{N}(v_t)\right)}y_{s}& =y_{t} + \sum_{s\in I\left(N(v_t)\right)}y_{s}\nonumber\\
&=y_{t} + \sum_{s\in I\left(V(G)\setminus \{v_{t}\}\right)}y_{s}\qquad\qquad\qquad\qquad \text{by \eqref{eqn:1.2-41}}  \nonumber\\
&=\sum_{s\in I\left(V(G)\right)}y_{s}\nonumber\\
  &=\sum_{s\in I\left(B_{k}\right)}y_{s}+\sum_{s\in I\left(V(G)\setminus B_{k}\right)}y_{s}\nonumber\\
   &=\sum_{s\in I\left(B_{k}\right)}x_{s}^{2}+\sum_{s\in I\left(V(G)\setminus B_{k}\right)}x_{s}^{1}\qquad\qquad\qquad \text{by construction of $y$}\nonumber\\
    &= \sum_{s\in I\left(B_{k}\right)}x_{s}^{2}+\sum_{s\in I\left(V(G)\setminus B_{k}\right)}x_{s}^{2}\qquad\qquad\qquad \text{by \eqref{eqn:1.2-27}} \nonumber\\
    &=\sum_{s\in I\left(V(G)\right)}x_{s}^{2}\nonumber\\
    &= x_{t}^{2} + \sum_{s\in I\left(V(G)\setminus \{v_{t}\}\right)}x_{s}^{2}\nonumber\\
    &=x_{t}^{2} + \sum_{s\in I\left(N(v_t)\right)}x_{s}^{2}\\
    &= \sum_{s\in I\left(\overline{N}(v_t)\right)}x_{s}^{2}
\end{flalign*}

Thus, by part \ref{Part 4} of Lemma \ref{Lemma:1} we obtain $x_{s}^{1}=x_{s}^{2}$ on $I(V(G)\setminus V(S))$, as $y_{s}=x_{s}^{1}$ on $I(V(G)\setminus V(S))$. Combining this with \eqref{eqn:1.2-43} we deduce
\begin{equation}\label{eqn:1.2-44}
 x_{s}^{1}=x_{s}^{2}\quad\text{on}\quad I(V(G)\setminus V(S_{k}))=I(V(G)\setminus B_{k})=I(N(v_1)).
\end{equation}
To prove that $x_{s}^{1}=x_{s}^{2}$ on $I(B_{k})$ we use part \ref{Part b} of Lemma \ref{Lemma:1}. Let us first recall that $p\in I(N(v_1))$, and then, the above result tell us that $x_{p}^{1}=x_{p}^{2}$. 
% Then 
%\begin{equation}\label{eqn:1.2-45}
%\sum_{{s\in I(C_{x_{p}}(G))}}x_{s}^{1}=D_{x_{p}}u_{p}(x_{p}^{1})=D_{x_{p}}u_{p}(x_{p}^{2})=\sum_{{s\in I(C_{x_{p}}(G))}}x_{s}^{2},
%\end{equation}
%as
%$$x_{p}^{i}\in  \text{Argmax}\Big\{x_{p} \mapsto \Big (\sum_{{s\in I(C_{x_{p}}(G))}}x_{s}^{i}\Big )\cdot x_{p}-u_{p}(x_{p})\Big\},$$
%for $i=1,2.$
Furthermore, $p\in I(V(G)\setminus V(S))$ or  $p\in I(B_{j})$ for some $j\in \{1, \ldots, l\}$, $k\neq j$. Thus, from \eqref{eqn:1.2-41}, \eqref{eqn:1.2-42} and the disjointness  of $B_{k}$ and $B_{j}$, we deduce $B_{k}\subseteq N(v_p)$. %It follows from \eqref{eqn:1.2-44} and \eqref{eqn:1.2-45} that 
%\begin{equation*}
%\sum_{s\in I\left( B_{i}\right)}x_{s}^{2}=\sum_{s\in I\left( B_{i}\right)}x_{s}^{1}.
%\end{equation*}
Now, set $F_{1}=B_{k}$, $F_{2}=N(v_p)\setminus B_{k}$ and $F_{3}=V(G)\setminus N(v_p)$. Then, $F_{1}, F_{2}\subseteq N(v_p)$, $F_{1}\cup F_{2}=N(v_p)$ and $F_{2}\cup F_{3}= V(G)\setminus B_{k}= N(v_s)$, for every $s\in I(F_{1})$. Also, from \eqref{eqn:1.2-44} we get $x_{s}^{1}=x_{s}^{2}$ on $I(F_{3})$, as it is evident that $B_{k}\cap F_{3}=\emptyset$. We can then apply part \ref{Part b} of Lemma \ref{Lemma:1} to obtain $x_{s}^{1}=x_{s}^{2}$ on $I\left( B_{k}\right)$, which combined with \eqref{eqn:1.2-44} allow us to have  $x_{s}^{1}=x_{s}^{2}$ on $\{2, \ldots, m\}$. This completes the proof of sub-case (a).
\item Assume  $A\neq \emptyset$. Let us first prove that $x_{s}^{1}=x_{s}^{2}$ on $I(V(G)\setminus V(S))$; this will be achieved via part \ref{Part 4} of Lemma \ref{Lemma:1}.
\par 
Using  \eqref{eqn:1.2-27} and \eqref{eqn:1.2-43}, we can write 
\begin{equation*}
\sum_{s\in I\left( V(G)\setminus V(S)\right)}x_{s}^{1}=\sum_{s\in I\left( V(G)\setminus V(S)\right)}x_{s}^{2},
\end{equation*}
and defining $y$ as in \eqref{eqn:1.2-53} we can equivalently  write 
\begin{equation*}
\sum_{s\in I\left( V(G)\setminus V(S)\right)}y_{s}=\sum_{s\in I\left( V(G)\setminus V(S)\right)}x_{s}^{2}.
\end{equation*}
Hence, from \eqref{eqn:1.2-26} and part \ref{Part 1} of Lemma \ref{Lemma:1} we get $y^{\prime}=(y^{\prime}_{1},\ldots, y^{\prime}_{m})\in \widetilde{W_{p}}$, where 
\[y_{s}^{\prime}=\begin{cases} 
      y_{s} & \text{if}\quad s\in I(V(G) \setminus A) \\
      x_{s}^{2} & \text{if}\quad  s\in I(A) \\
   \end{cases}
\]
 \[\qquad \quad=\begin{cases} 
      x_{s}^{1} & \text{if}\quad s\in I(V(G) \setminus V(S_{k})) \\
      x_{s}^{2} & \text{if}\quad  s\in I(V(S_{k})). \\
   \end{cases}
\]
Then, 
$$\sum_{s\in I(V(G)\setminus V(S_{k}))}y_{s}^{\prime}=\sum_{s\in I(N(v_1))}y_{s}^{\prime}=Du_{1}(x_{1}^{0})=\sum_{s\in I(N(v_1))}x_{s}^{2}=\sum_{s\in I(V(G)\setminus V(S_{k}))}x_{s}^{2}.$$
By construction of $y^{\prime}$, one has,
\begin{equation*}
\sum_{s\in I\left( V(G)\right)}y_{s}^{\prime}=\sum_{s\in I\left( V(G)\right)}x_{s}^{2}.
\end{equation*}
Hence, using \eqref{eqn:1.2-41} we clearly might  express it as 
\begin{equation*}
\sum_{s\in I\left( \overline{N}(v_t)\right)}y_{s}^{\prime}=\sum_{s\in I\left( \overline{N}(v_t)\right)}x_{s}^{2},
\end{equation*}
for every fixed $t\in I(V(G)\setminus V(S))$. We can now apply  part \ref{Part 4} of Lemma \ref{Lemma:1} and get $y_{t}^{\prime}=x_{t}^{2}$ on $I(V(G)\setminus V(S))$, which implies $x_{t}^{1}=x_{t}^{2}$ on $I(V(G)\setminus V(S))$, since $I(V(G)\setminus V(S))\subseteq I(V(G)\setminus V(S_{k}))$ and $y_{t}^{\prime}=x_{t}^{1}$ on $ I(V(G)\setminus V(S_{k}))$. Thus, from  \eqref{eqn:1.2-43},
\begin{equation}\label{eqn:1.2-52}
 x_{s}^{1}=x_{s}^{2}\quad\text{on}\quad I(V(G)\setminus V(S_{k}))=I(N(v_1)).
\end{equation}
 \par
 It only remains to prove that $x_{s}^{1}=x_{s}^{2}$ on $I(V(S_{k}))$. Since $p\in I(N(v_1))$, from the above equalities  $p\in I\left( V(G)\setminus V(S)\right)$ or $p\in I(B_{j})$ for some $j\neq k$, and  $x_p^{1}=x_p^{2}$. Then $N(v_p)=V(G)\setminus \{v_{p}\}$ or $N(v_p)=V(G)\setminus V(S_{j})$. It follows that $N(v_p)=V(G)\setminus \{v_{p}\}$, as $A\subseteq N(v_p)$ and $A\cap \left( V(G)\setminus V(S_j)\right)=\emptyset$. Now, set $F_{1}=A$, $F_{2}=V(G)\setminus \left( V(S)\cup \{v_{p}\}\right)$ and $F_{3}=\{v_{p}\}$. It is clear that $F_{1},F_{2}\subseteq N(v_p)$ and $F_{2}\cup F_{3}=V(G)\setminus V(S)=N(v_s)$ for every $s\in I(A)$. Furthermore, for $y$ defined as in \eqref{eqn:1.2-53} and $(x_{1}^{0}, x_{2}^{2}, \ldots, x_{m}^{2})$, we have $y_{s}=x_{2}^{2}$ on $I(B_{k})$.  It follows from \eqref{eqn:1.2-52} and construction of $y$ that $y_{s}=x_{2}^{2}$ on $I(V(G)\setminus A)$; in particular, $y_{s}=x_{2}^{2}$ on $I\left(N(v_p)\setminus F_{1}\cup F_{2}\right)\subseteq I(V(G)\setminus A) $. Hence, $x_{s}^{1}=y_{s}=x_{2}^{2}$ on $I(A)$ by part \ref{Part b} of Lemma \ref{Lemma:1}, and then, we can easily obtain $x_{s}^{1}=x_{s}^{2}$ on $I(B_{k})$, by applying again part \ref{Part b} of Lemma \ref{Lemma:1}; this time we set $F_{1}=B_{k}$, $F_{2}=\left( V(G)\setminus V(S_{k})\right)\setminus \{v_{p}\}$ and $F_{3}=\{v_{p}\}$.  This completes the proof of sub-case (b)\par
% On the other hand, if $p\in I(B_{j})$ for some $j\neq k$, then $N(v_p)=V(G)\setminus V(S_{j})$. Here, we set $F_{1}=B_{k}$, $F_{2}=N(v_p)\setminus B_{k}$ and $F_{3}=B_{j}$. Note that $$F_{2}\cup F_{3}=\left(N(v_p)\setminus B_{k}\right) \cup B_{j}=\left( V(G)\setminus (B_{k}\cup V(S_{j}))\right) \cup B_{j}=V(G)\setminus V(S_{k})=N(v_s),$$
% for every $s\in I(F_{1})$. Also, from \eqref{eqn:1.2-52}, $x_{s}^{1}=x_{s}^{2}$ on $I(F_{3})$. This implies  $x_{s}^{1}=x_{s}^{2}$ on $I(B_{k})$  by part \ref{Part b} of Lemma \ref{Lemma:1}. Finally, applying one more time this part of the lemma we get $x_{s}^{1}=x_{s}^{2}$ on $I(A)$, by setting $F_{1}=A$, $F_{2}=\left( V(G)\setminus V(S)\right)\setminus \{v_{p}\}$ and $F_{3}=\{v_{p}\}$. .  %ag by 
 %First, we obtain $x_{s}^{1}=x_{s}^{2}$ on $I(B_{k})$  by mimicking the proof  of these equalities in the previous sub-case; this time we define $F_{1}=B_{k}$, $F_{2}=C_{x_{p}}(G)\cap \left(V(G)\setminus V(S_{k})\right)$ and $F_{3}=\left(V(G)\setminus V(S_{k})\right)\setminus C_{x_{p}}(G)$.  This immediately let us obtain $x_{s}^{1}=x_{s}^{2}$ on $I(A)$, since the conditions to apply part \ref{Part b} of Lemma \ref{Lemma:1} are trivially satisfied by defining $ F_{1}^{*}=A$, $F_{2}^{*}=\left(V(G)\setminus V(S)\right)\setminus \{x_{1}\}$ and $F_{3}^{*}=\{x_{1}\}$; therefore, $x_{s}^{1}=x_{s}^{2}$ on $\{2, \ldots, m\}$.  This completes the proof of sub-case (b). 
\end{enumerate}
\end{enumerate}
This completes the proof of the theorem.
\end{proof}

 Before presenting some examples, we note the following consequence of the preceding theorem.  
\begin{corollary}\label{cor: one missing edge}
Let $G$ be a subgraph of $C_{m}$ with $|G|=m$, and  satisfying $|N(v)|\in \{m-1, m-2\}$ for every $v\in V(G)$.  Assume that $\mu_{1}$ is absolutely continuous with respect to Lebesgue measure, and that \emph{either} $|N(v_1)| =m-1$ \emph{or} that $\mu_{i}$ is absolutely continuous with respect to Lesbesgue measure for some $i \in I(N(v_1))$.  Then every solution to the Kantorovich problem \eqref{KP} is induced by a map.
\end{corollary}  
\begin{proof}Note that if $G\neq C_{m}$, then $G=C_{m}\setminus \bigcup S_{j=1}^{l}$, for some disjoint collection of complete graphs $\{S_{j}\}_{j=1}^{l}$, where $|V(S_{j})|=2$ for every $j$ (that is, every $S_j$ consists on a single edge). Clearly, the graph $\bigcup_{j=1}^{l} S_j$ has inner hub $A=\emptyset$ and maximal cliques $S_1, \ldots, S_l$. The result then follows from Theorem \ref{Theorem 2}.
\end{proof}
Note that the Gangbo-Swiech surplus corresponds to a complete graph, or,  equivalently, to the graph $G$ satisfying $|N(v)| =m-1$ for each $v \in V(G)$; the Corollary is then a generalization to the case where each vertex can be missing at most one edge connecting it to the other vertices.
\subsection{Examples}
Here, we illustrate the result obtained in Theorem \ref{Theorem 2} throughout several  examples.
\begin{enumerate}[label=(\roman*)]
\item Let $G$ be a complete $k$-partite graph with set partition $\{V_{1}, \ldots, V_{k}\}$ and $m:=|V(G)|=|\bigcup_{j=1}^{k}V_{j}|$. Write $\bigcup_{j=1}^{k}V_{j}=\{v_{1}, \ldots, v_{m}\}$, and let $S_{1}, \ldots, S_{k}$ be $k$ complete graphs with sets of vertices $V_{1}, \ldots, V_{k}$, respectively. Note that $G:=C_{m}\setminus \bigcup_{j=1}^{k}S_{j}$ and $ N(v_1)=\bigcup_{\underset{j\neq \alpha}{j=1}}^{k}V_{j}$, for some $\alpha \in \{1, \ldots, k\}$. Hence, by assuming $\mu_{1}$ and $\mu_{p}$ absolutely continuous, for some $p\in N(v_1)$, we can conclude by Theorem \ref{Theorem 2} that the graph $G$ gives a unique Monge solution, as we can interpret $\{S_{j}\}_{j=1}^{k}$ as the collection of maximal cliques of the graph $\bigcup_{j=1}^{k}S_{j}$. Here, $A=\emptyset$ is clearly the inner hub of $\bigcup_{j=1}^{k}S_{j}$. \\
 A special case is the complete graph $C_k$; several other examples of $k$-partite graphs are below.%Special cases of complete $k$-partite graphs includes a complete graph $C_{k}$, as well as,
% and $K_{1,k}$ (\textit{known as the star Graph}), as well as,
\begin{itemize}
\item \underline{Complete bipartite graphs $K_{m,n}$:}

\begin{figure}[H]   
\begin{center}
\begin{subfigure}[b]{0.4\linewidth}
\begin{tikzpicture}[x=1.5cm, y=1.5cm]
	\vertex (v1) at (60:1) [label=60:$v_{1}$]{};
	\vertex (v2) at (120:1) [label=120:$v_{2}$]{};
	\vertex (v3) at (180:1) [label=180:$v_{3}$]{};
	\vertex (v4) at (240:1) [label=240:$v_{4}$]{};
	\vertex (v5) at (300:1) [label=300:$v_{5}$]{};
	\vertex (v6) at (360:1) [label=360:$v_{6}$]{};
	\path 
		(v1) edge (v2)
		(v2) edge (v3)
		(v3) edge (v4)
		(v4) edge (v5)
		(v6) edge (v5)
		(v6) edge (v1)
		(v1) edge (v4)
		(v2) edge (v5)
		(v6) edge (v3)
		;	
\end{tikzpicture}
\caption{Graph $K_{3,3}$. Known as the \textit{Utility graph}.} 
  \end{subfigure}
  \qquad \qquad
\begin{subfigure}[b]{0.4\linewidth}
\begin{tikzpicture}[x=1.5cm, y=1.5cm]
	\vertex (v1) at (45:1) [label=45:$v_{1}$]{};
	\vertex (v2) at (90:1) [label=90:$v_{2}$]{};
	\vertex (v3) at (135:1) [label=135:$v_{3}$]{};
	\vertex (v4) at (180:1) [label=180:$v_{4}$]{};
	\vertex (v5) at (225:1) [label=225:$v_{5}$]{};
	\vertex (v6) at (270:1) [label=270:$v_{6}$]{};
	\vertex (v7) at (315:1) [label=315:$v_{7}$]{};
	\vertex (v8) at (360:1) [label=360:$v_{8}$]{};
	\path 
		(v1) edge (v2)
		(v2) edge (v3)
		(v3) edge (v4)
		(v4) edge (v5)
		(v6) edge (v5)
		(v6) edge (v7)
		(v7) edge (v8)
		(v1) edge (v8)
		(v2) edge (v5)
		(v2) edge (v7)
		(v1) edge (v6)
		(v3) edge (v6)
		(v4) edge (v1)
		(v4) edge (v7)
		(v3) edge (v8)
		(v5) edge (v8)
		;	
\end{tikzpicture}
\caption{Graph $K_{4,4}$. Known as the \textit{Cayley graph}.} 
\end{subfigure}
\caption{}
\end{center}
\end{figure}
\begin{figure}[H]  
\begin{center}
\begin{tikzpicture}[x=2cm, y=2cm]
	\vertex (v1) at (36:1) [label=36:$v_{1}$]{};
	\vertex (v2) at (72:1) [label=72:$v_{2}$]{};
	\vertex (v3) at (108:1) [label=108:$v_{3}$]{};
	 \vertex (v4) at (144:1) [label=144:$v_{4}$]{};
	\vertex (v5) at (180:1) [label=180:$v_{5}$]{};
	\vertex (v6) at (216:1) [label=216:$v_{6}$]{};
	\vertex (v7) at (252:1) [label=252:$v_{7}$]{};
	\vertex (v8) at (288:1) [label=288:$v_{8}$]{};
	\vertex (v9) at (324:1) [label=324:$v_{9}$]{};
	\vertex (v10) at (360:1) [label=360:$v_{10}$]{};
	\path 
		
		(v1) edge (v7)
		(v1) edge (v8)
		(v1) edge (v6)
		(v1) edge (v9)
		
		(v2) edge (v6)
		(v2) edge (v7)
		(v2) edge (v8)
		(v2) edge (v9)
		
		(v3) edge (v6)
		(v3) edge (v7)
		(v3) edge (v9)
		(v3) edge (v8)
		
		(v4) edge (v6)
		(v4) edge (v7)
		(v4) edge (v8)
		(v4) edge (v9)
		
		(v5) edge (v6)
		(v5) edge (v7)
		(v5) edge (v8)
		(v5) edge (v9)

		(v6) edge (v10)
	
		(v9) edge (v10)	
		(v7) edge (v10)
		
		(v8) edge (v10)	
		;
\end{tikzpicture}
 \caption{Bipartite graph with set partition $\{V_{1}, V_{2}\}$, where  $V_{1}=\{v_{1}, v_{2}, v_{3}, v_{4}, v_{5}, v_{10}\}$ and $V_{2}=\{v_{6}, v_{7}, v_{8}, v_{9}\}$. } \label{fig:M1}  
\end{center}
\end{figure}
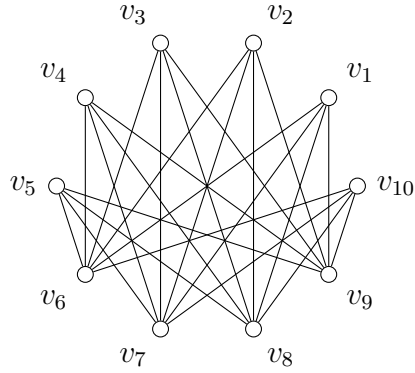

\item \underline{Complete Tripartite graphs $K_{m,n, p}$:}
\begin{figure}[H]   
\begin{center}
\begin{subfigure}[b]{0.4\linewidth}
\begin{tikzpicture}[x=1.5cm, y=1.5cm]
	\vertex (v1) at (90:1) [label=90:$v_{1}$]{};
	\vertex (v2) at (180:1) [label=180:$v_{2}$]{};
	\vertex (v3) at (270:1) [label=270:$v_{3}$]{};
	\vertex (v4) at (360:1) [label=360:$v_{4}$]{};
	\vertex (v5) at (360:0) [label=45:$v_{5}$]{};
	\path 
		(v1) edge (v2)
		(v2) edge (v3)
		(v3) edge (v4)
		(v4) edge (v1)
		(v1) edge (v4)
		(v4) edge (v3)
		(v2) edge (v5)
		(v5) edge (v4)
		(v1) edge (v5)
		(v5) edge (v3)
		;	
\end{tikzpicture}
\caption{Graph $K_{1,2,2}$. Known as the \textit{5-wheel graph}.}
  \end{subfigure}
  \qquad \qquad
\begin{subfigure}[b]{0.4\linewidth}
\begin{tikzpicture}[x=1.5cm, y=1.5cm]
	\vertex (v1) at (90:1) [label=90:$v_{1}$]{};
	\vertex (v2) at (180:1) [label=180:$v_{2}$]{};
	\vertex (v3) at (270:1) [label=270:$v_{3}$]{};
	\vertex (v4) at (360:1) [label=360:$v_{4}$]{};
	\path 
		(v1) edge (v2)
		(v2) edge (v3)
		(v3) edge (v4)
		(v4) edge (v1)
		(v2) edge (v4)
		;	
\end{tikzpicture}
\caption{Graph $K_{1,1,2}$. Known as the \textit{Diamond graph}.}
\end{subfigure}
\end{center}
\end{figure}

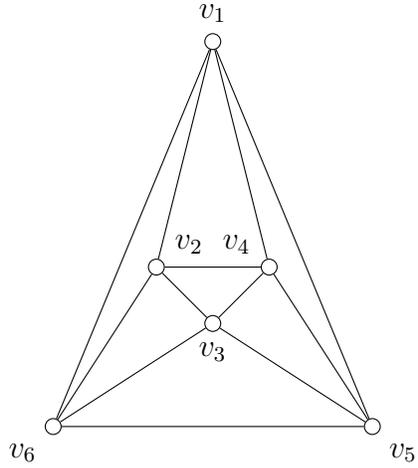
\begin{figure}[H]  
\begin{center}
\begin{tikzpicture}[x=1.5cm, y=1.5cm]
	\vertex (v1) at (90:2) [label=90:$v_{1}$]{};
	\vertex (v2) at (180:0.5) [label=20:$v_{2}$]{};
	\vertex (v3) at (270:0.5) [label=270:$v_{3}$]{};
	\vertex (v4) at (360:0.5) [label=160:$v_{4}$]{};
	\vertex (v5) at (315:2) [label=315:$v_{5}$]{};
	\vertex (v6) at (225:2) [label=225:$v_{6}$]{};
	\path 
		(v1) edge (v2)
		(v2) edge (v3)
		(v3) edge (v4)
		(v4) edge (v1)
		(v2) edge (v4)
		(v1) edge (v6)
		(v6) edge (v5)
		(v5) edge (v1)
		(v2) edge (v6)
		(v3) edge (v6)
		(v3) edge (v5)
		(v4) edge (v5)
		;	
\end{tikzpicture}
\caption{Graph $K_{2,2,2}$. Known as the \textit{Octahedral graph}.}
\end{center}
\end{figure}   
\end{itemize}

\item A notable special case of Corollary \ref{cor: one missing edge} occurs when $m$ is even and $|N(v)|=m-2$ for all $v\in V(G)$, in which case $G=K_{2, \ldots, 2(\frac{m}{2} \text{times})}$. This graph is known as the \textit{Cocktail Party Graph}. See example below.
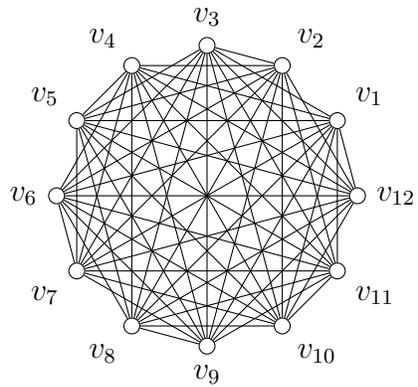
\begin{figure}[H]
\centering
\begin{center}
\begin{tikzpicture}[x=2cm, y=2cm]
\vertex (v1) at (30:1) [label=30:$v_{1}$]{};
\vertex (v2) at (60:1) [label=60:$v_{2}$]{};
\vertex (v3) at (90:1) [label=90:$v_{3}$]{};
\vertex (v4) at (120:1) [label=120:$v_{4}$]{};
\vertex (v5) at (150:1) [label=150:$v_{5}$]{};
\vertex (v6) at (180:1) [label=180:$v_{6}$]{};
\vertex (v7) at (210:1) [label=210:$v_{7}$]{};
 \vertex (v8) at (240:1) [label=240:$v_{8}$]{};
\vertex (v9) at (270:1) [label=270:$v_{9}$]{};
\vertex (v10) at (300:1) [label=300:$v_{10}$]{};
\vertex (v11) at (330:1) [label=330:$v_{11}$]{};
\vertex (v12) at (360:1) [label=360:$v_{12}$]{};
%\vertex (x13) at (312:1) [label=312:$x_{13}$]{};
%\vertex (x14) at (336:1) [label=336:$x_{14}$]{};
%\vertex (x15) at (360:1) [label=360:$x_{15}$]{};
\path 
		%(x1) edge (x2)
		(v1) edge (v3)
		(v1) edge (v4)
		(v1) edge (v5)
		(v1) edge (v6)
		(v1) edge (v7)
		(v1) edge (v8)
		(v1) edge (v9)
		(v1) edge (v10)
		(v1) edge (v11)
		(v1) edge (v12)
		%(x1) edge (x13)
		%(x1) edge (x14)
		%(x1) edge (x15)
		(v2) edge (v3)
		(v2) edge (v4)
		(v2) edge (v5)
		(v2) edge (v6)
		(v2) edge (v7)
		(v2) edge (v8)
		(v2) edge (v9)
		(v2) edge (v10)
		(v2) edge (v11)
		(v2) edge (v12)
		%(x2) edge (x13)
		%(x2) edge (x14)
		%(x2) edge (x15)
		%(x3) edge (x4)
		(v3) edge (v5)
		(v3) edge (v6)
		(v3) edge (v7)
		(v3) edge (v8)
		(v3) edge (v9)
		(v3) edge (v10)
		(v3) edge (v11)
		(v3) edge (v12)
		%(x3) edge (x13)
		%(x3) edge (x14)
		%(x3) edge (x15)
		(v4) edge (v5)
		(v4) edge (v6)
		(v4) edge (v7)
		(v4) edge (v8)
		(v4) edge (v9)
		(v4) edge (v10)
		(v4) edge (v11)
		(v4) edge (v12)
		%(x4) edge (x13)
		%(x4) edge (x14)
		%(x4) edge (x15)
		%(x5) edge (x6)
		(v5) edge (v7)
		(v5) edge (v8)
	    (v5) edge (v9)
	    (v5) edge (v10)
	    (v5) edge (v11)
	    (v5) edge (v12)
	    %(x5) edge (x13)
	    %(x5) edge (x14)
	    %(x5) edge (x15)
        (v6) edge (v7)
		(v6) edge (v8)
		(v6) edge (v9)
		(v6) edge (v10)
		(v6) edge (v11)
		(v6) edge (v12)
		%(x6) edge (x13)
		%(x6) edge (x14)
		%(x6) edge (x15)
		%(x7) edge (x8)
		(v7) edge (v9)
		(v7) edge (v10)
		(v7) edge (v11)
		(v7) edge (v12)
		%(x7) edge (x13)
		%(x7) edge (x14)
		%(x7) edge (x15)
		(v8) edge (v9)
		(v8) edge (v10)
		(v8) edge (v11)
		(v8) edge (v12)	
		%(x8) edge (x13)	
		%(x8) edge (x14)	
		%(x8) edge (x15)
		%(x9) edge (x10)
		(v9) edge (v11)	
		(v9) edge (v12)	
		%(x9) edge (x13)	
		%(x9) edge (x14)
		%(x9) edge (x15)
		(v10) edge (v11)
		(v10) edge (v12)		
		%(x10) edge (x13)		
		%(x10) edge (x14)		
		%x10) edge (x15)
		%(x11) edge (x12)
		%(x11) edge (x13)	
		%(x11) edge (x14)	
		%(x11) edge (x15)
		%(x12) edge (x13)
		%(x12) edge (x14)
		%(x12) edge (x15)
	    %(x13) edge (x14)
	    %(x13) edge (x15)	
	   % (x14) edge (x15)	 							
		;
\end{tikzpicture}
 \caption{A Cocktail Party Graph with $m=12$. } \label{fig:M1}  
\end{center}
\end{figure} 

\item  Theorem \ref{Theorem 2} can be used easily to construct many other, more obscure, graphs leading to Monge solutions.  We construct one of such examples here; set
\begin{flalign*}
V_{1}&=\{ v_{1}, v_{14},v_{15},v_{16},v_{17},v_{18},v_{19},v_{20}\}, \nonumber\\
V_{2}&= \{v_{9},v_{10},v_{11},v_{12},v_{13},v_{14},v_{15},v_{16}\},\nonumber\\
V_{3}&= \{v_{6},v_{14},v_{15},v_{16}\},\nonumber\\
V_{4}&= \{v_{4},v_{14},v_{15},v_{16}\}.
\end{flalign*} 
Consider  $ S_{1}, S_{2}, S_{3}, S_{4}$ complete graphs with $V(S_j)=V_j$, $j=1,2,3,4$. Then, the graph $S=S_1\cup S_2\cup S_3\cup S_4$ has inner hub $A=\{v_{14},v_{15},v_{16}\}$, with maximal cliques  $S_{1}, S_{2}, S_{3}, S_{4}$. See Figure below.
 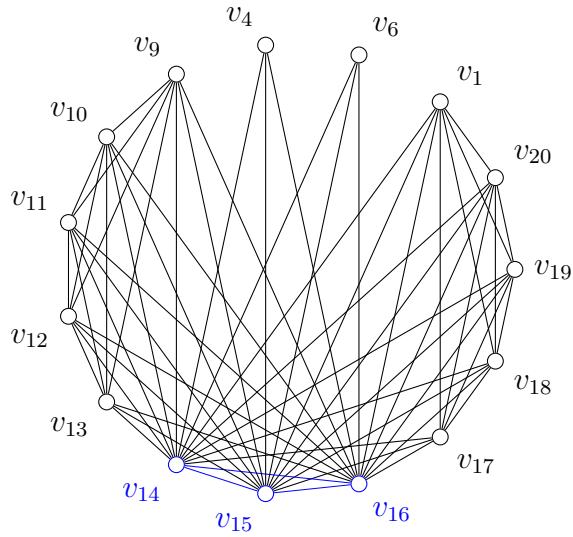
\begin{figure}[H]
\centering
\begin{center}
\begin{tikzpicture}[x=3cm, y=3cm]
\vertex (v20) at (24:1) [label=24:$v_{20}$]{};
\vertex (v1) at (48:1) [label=48:$v_{1}$]{};
\vertex (v6) at (72:1) [label=72:$v_{6}$]{};
\vertex (v4) at (96:1) [label=96:$v_{4}$]{};
\vertex (v9) at (120:1) [label=120:$v_{9}$]{};
\vertex (v10) at (144:1) [label=144:$v_{10}$]{};
\vertex (v11) at (168:1) [label=168:$v_{11}$]{};
 \vertex (v12) at (192:1) [label=192:$v_{12}$]{};
\vertex (v13) at (216:1) [label=216:$v_{13}$]{};
{\blue \vertex (v14) at (240:1) [label=240:$v_{14}$]{};
\vertex (v15) at (264:1) [label=264:$v_{15}$]{};
\vertex (v16) at (288:1) [label=288:$v_{16}$]{};}
\vertex (v17) at (312:1) [label=312:$v_{17}$]{};
\vertex (v18) at (336:1) [label=336:$v_{18}$]{};
\vertex (v19) at (360:1) [label=360:$v_{19}$]{};
\path 
		(v1) edge (v20)
		(v1) edge (v19)
		(v1) edge (v18)
		(v1) edge (v17)
		(v1) edge (v16)
		(v1) edge (v15)
		(v1) edge (v14)
		(v20) edge (v19)
		(v20) edge (v18)
		(v20) edge (v17)
		(v20) edge (v16)
		(v20) edge (v15)
		(v20) edge (v14)
		(v19) edge (v18)
		(v19) edge (v17)
		(v19) edge (v16)
		(v19) edge (v15)
		(v19) edge (v14)
		(v18) edge (v15)
		(v18) edge (v16)
		(v18) edge (v17)
		(v18) edge (v14)
		(v17) edge (v16)
		(v17) edge (v15)
		(v17) edge (v14)
		(v16) edge [ultra thin, draw=blue,-](v14)
		(v14) edge [ultra thin, draw=blue,-](v15)
		(v9) edge (v10)
		(v9) edge (v11)
		(v9) edge (v12)
		(v9) edge (v13)
		(v9) edge (v14)
		(v9) edge (v15)
		(v9) edge (v16)
		(v10) edge (v11)
		(v10) edge (v12)
		(v10) edge (v13)
		(v10) edge (v14)
		(v10) edge (v15)
		(v10) edge (v16)
		(v11) edge (v12)
		(v11) edge (v13)
		(v11) edge (v14)
		(v11) edge (v15)
		(v11) edge (v16)
		(v12) edge (v13)
		(v12) edge (v14)
		(v12) edge (v15)
		(v12) edge (v16)
		(v13) edge (v14)
		(v13) edge (v15)
		(v13) edge (v16)
		(v14) edge (v4)
		(v14) edge (v6)
		(v15) edge (v4)
		(v15) edge (v6)
		(v15) edge [ultra thin, draw=blue,-](v16)
		(v16) edge (v4)
		(v16) edge (v6)

		;
\end{tikzpicture}
 \caption{Graph $S=S_{1}\cup S_{2} \cup S_{3}\cup S_{4}.$} \label{fig:M1}  
\end{center}
\end{figure} 
Then, $G=C_{20}\setminus S$ provides a solution of Monge type. 
\begin{figure}[H]
\centering
\begin{center}
\begin{tikzpicture}[x=3cm, y=3cm]
 \vertex (v1) at (18:1) [label=18:$v_{1}$]{};
\vertex (v2) at (36:1) [label=36:$v_{2}$]{};
\vertex (v3) at (54:1) [label=54:$v_{3}$]{};
\vertex (v4) at (72:1) [label=72:$v_{4}$]{};
\vertex (v5) at (90:1) [label=90:$v_{5}$]{};
\vertex (v6) at (108:1) [label=108:$v_{6}$]{};
\vertex (v7) at (126:1) [label=126:$v_{7}$]{};
 \vertex (v8) at (144:1) [label=144:$v_{8}$]{};
 \vertex (v9) at (162:1) [label=162:$v_{9}$]{};
\vertex (v10) at (180:1) [label=180:$v_{10}$]{};
\vertex (v11) at (198:1) [label=198:$v_{11}$]{};
\vertex (v12) at (216:1) [label=216:$v_{12}$]{};
\vertex (v13) at (234:1) [label=234:$v_{13}$]{};
{\blue\vertex (v14) at (252:1) [label=252:$v_{14}$]{};
\vertex (v15) at (270:1) [label=270:$v_{15}$]{};
\vertex (v16) at (288:1) [label=288:$v_{16}$]{};}
\vertex (v17) at (306:1) [label=306:$v_{17}$]{};
 \vertex (v18) at (324:1) [label=324:$v_{18}$]{};
\vertex (v19) at (342:1) [label=342:$v_{19}$]{};
\vertex (v20) at (360:1) [label=360:$v_{20}$]{};
%\vertex (x21) at (252:1) [label=252:$x_{21}$]{};
%\vertex (x22) at (264:1) [label=264:$x_{22}$]{};
%\vertex (x23) at (276:1) [label=276:$x_{23}$]{};
%\vertex (x24) at (288:1) [label=288:$x_{24}$]{};
%\vertex (x25) at (300:1) [label=300:$x_{25}$]{};
%\vertex (x26) at (312:1) [label=312:$x_{26}$]{};
%\vertex (x27) at (324:1) [label=324:$x_{27}$]{};
%\vertex (x28) at (336:1) [label=336:$x_{28}$]{};
%\vertex (x29) at (348:1) [label=348:$x_{29}$]{};
%\vertex (x30) at (360:1) [label=360:$x_{30}$]{};
\path 
		(v1) edge (v2)
		(v1) edge (v3)
		(v1) edge (v4)
		(v1) edge (v5)
		(v1) edge (v6)
		(v1) edge (v7)
		(v1) edge (v8)
		(v1) edge (v9)
		(v1) edge (v10)
		(v1) edge (v11)
		(v1) edge (v12)
		(v1) edge (v13)
		%(x1) edge (x14)
		%(x1) edge (x15)
		%(x1) edge (x16)
		%(x1) edge (x17)
		%(x1) edge (x18)
		%(x1) edge (x19)
		%(x1) edge (x20)
		(v2) edge (v3)
		(v2) edge (v4)
		(v2) edge (v5)
		(v2) edge (v6)
		(v2) edge (v7)
		(v2) edge (v8)
		(v2) edge (v9)
		(v2) edge (v10)
		(v2) edge (v11)
		(v2) edge (v12)
		(v2) edge (v13)
		(v2) edge (v14)
		(v2) edge (v15)
		(v2) edge (v16)
		(v2) edge (v17)
		(v2) edge (v18)
		(v2) edge (v19)
		(v2) edge (v20)
		(v3) edge (v4)
		(v3) edge (v5)
		(v3) edge (v6)
		(v3) edge (v7)
		(v3) edge (v8)
		(v3) edge (v9)
		(v3) edge (v10)
		(v3) edge (v11)
		(v3) edge (v12)
		(v3) edge (v13)
		(v3) edge (v14)
		(v3) edge (v15)
		(v3) edge (v16)
		(v3) edge (v17)
		(v3) edge (v18)
		(v3) edge (v19)
		(v3) edge (v20)
		(v4) edge (v5)
		%(x4) edge (x6)
		(v4) edge (v7)
		(v4) edge (v8)
		(v4) edge (v9)
		(v4) edge (v10)
		(v4) edge (v11)
		(v4) edge (v12)
		(v4) edge (v13)
		%(x4) edge (x14)
		%(x4) edge (x15)
		%(x4) edge (x16)
		(v4) edge (v17)
		(v4) edge (v18)
		(v4) edge (v19)
		(v4) edge (v20)
		(v5) edge (v6)
		(v5) edge (v7)
		(v5) edge (v8)
	    (v5) edge (v9)
	    (v5) edge (v10)
	    (v5) edge (v11)
	    (v5) edge (v12)
	    (v5) edge (v13)
	    (v5) edge (v14)
	    (v5) edge (v15)
	    (v5) edge (v16)
	    (v5) edge (v17)
	    (v5) edge (v18)
	    (v5) edge (v19)
	    (v5) edge (v20)
        (v6) edge (v7)
		(v6) edge (v8)
		(v6) edge (v9)
		(v6) edge (v10)
		(v6) edge (v11)
		(v6) edge (v12)
		(v6) edge (v13)
		%(x6) edge (x14)
		%(x6) edge (x15)
		%(x6) edge (x16)
		(v6) edge (v17)
		(v6) edge (v18)
		(v6) edge (v19)
		(v6) edge (v20)
		(v7) edge (v8)
		(v7) edge (v9)
		(v7) edge (v10)
		(v7) edge (v11)
		(v7) edge (v12)
		(v7) edge (v13)
		(v7) edge (v14)
		(v7) edge (v15)
		(v7) edge (v16)
		(v7) edge (v17)
		(v7) edge (v18)
		(v7) edge (v19)
		(v7) edge (v20)
		(v8) edge (v9)
		(v8) edge (v10)
		(v8) edge (v11)
		(v8) edge (v12)	
		(v8) edge (v13)	
		(v8) edge (v14)	
		(v8) edge (v15)
		(v8) edge (v16)
		(v8) edge (v17)
		(v8) edge (v18)
		(v8) edge (v19)
		(v8) edge (v20)
		%(x9) edge (x10)
		%(x9) edge (x11)	
		%(x9) edge (x12)	
		%(x9) edge (x13)	
		%(x9) edge (x14)
		%(x9) edge (x15)
		%(x9) edge (x16)
		(v9) edge (v17)
		(v9) edge (v18)
		(v9) edge (v19)
		(v9) edge (v20)
		%(x10) edge (x11)
		%(x10) edge (x12)		
		%(x10) edge (x13)		
		%(x10) edge (x14)		
		%(x10) edge (x15)
		%(x10) edge (x16)
		(v10) edge (v17)
		(v10) edge (v18)
		(v10) edge (v19)
		(v10) edge (v20)
		%(x11) edge (x12)
		%(x11) edge (x13)	
		%(x11) edge (x14)	
		%(x11) edge (x15)
		%(x11) edge (x16)
		(v11) edge (v17)
		(v11) edge (v18)
		(v11) edge (v19)
		(v11) edge (v20)
		%(x12) edge (x13)
		%(x12) edge (x14)
		%(x12) edge (x15)
		%(x12) edge (x16)
		(v12) edge (v17)
		(v12) edge (v18)
		(v12) edge (v19)
		(v12) edge (v20)
	    %(x13) edge (x14)
	    %(x13) edge (x15)
	    %(x13) edge (x16)
	    (v13) edge (v17)	
	    (v13) edge (v18)	
	    (v13) edge (v19)	
	    (v13) edge (v20)		
	    %(x14) edge (x15)
	    %(x14) edge (x16)
	    %(x14) edge (x17)
	    %(x14) edge (x18)
	    %(x14) edge (x19)
	    %(x14) edge (x20)
	    %(x15) edge (x16)
	    %(x15) edge (x17)
	    %(x15) edge (x18)
	    %(x15) edge (x19)
	    %(x15) edge (x20)
	    %(x16) edge (x17)
	    %(x16) edge (x18)
	    %(x16) edge (x19)
	    %(x16) edge (x20)
	    %(x17) edge (x18)
	    %(x17) edge (x19)
	    %(x17) edge (x20)
	    %(x18) edge (x19)
	    %(x18) edge (x20)
	    %(x19) edge (x20)

		;
\end{tikzpicture}
 \caption{Graph $G=C_{20}\setminus S$.} \label{fig:M1}  
\end{center}
\end{figure}
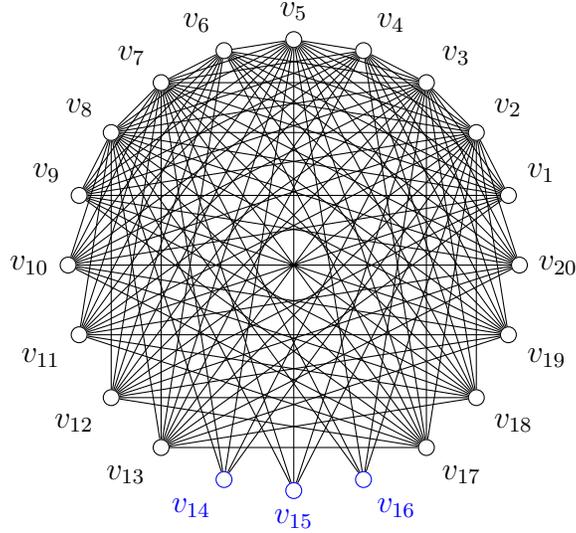 
 
\end{enumerate}

%\section{\textbf{Monge solutions for some union of graphs
 %}}\label{Section 4}
 \section{\textbf{Monge solutions for graphs with inner hubs and gluing of them
 }}\label{Section 4}
 The main result of this section (Theorem \ref{Theorem 1}) ensures that under regularity conditions on two of the marginals, the surplus associated to a graph with inner hub provides a unique solution  for the Monge-Kantorovich problem.\par
 
Before stating the main result of this section, we present the following simple example, which illustrates part of the motivation for Theorem \ref{Theorem 1} and Propositions \ref{Proposition 2} and \ref{Proposition 3}.

 \begin{example}\label{eqn:1.2-54}
 	%Let $X_{k}=B(0,r)\subseteq\mathbb{R}^{n}$ be an open ball, $r>0$ and $k=1,2,3,4$. Set $\mu_{1}=\mu_{2}=\mu_{4}=\frac{\mathcal{L}^{n}}{\mathcal{L}^{n}(B(0,r))}$ and $\mu_{3}=\delta_{0}$ the Dirac measure at the origin. 
 	Let $b$ be the surplus associated to the graph $G$ below.
 	
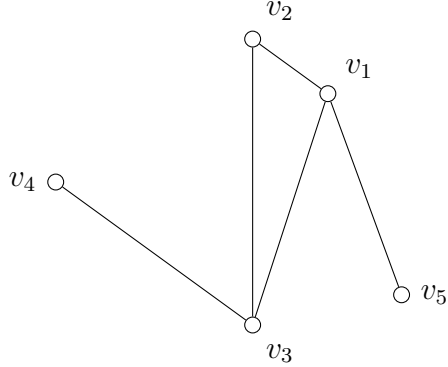
\begin{figure}[H]  
 		\begin{center}
 			\begin{tikzpicture}[x=2cm, y=2cm]
 			\vertex (v1) at (36:1) [label=36:$v_{1}$]{};
 			\vertex (v2) at (72:1) [label=72:$v_{2}$]{};
 			\vertex (v4) at (180:1) [label=180:$v_{4}$]{};
 			\vertex (v3) at (288:1) [label=288:$v_{3}$]{};
 			\vertex (v5) at (330:1.5) [label=360:$v_{5}$]{};
 			\path 
 			(v1) edge (v2)
 			(v1) edge (v3)
 			(v1) edge (v5)
 			(v2) edge (v3)
 			(v3) edge (v4)				
 			;
 			\end{tikzpicture}
 		\end{center}
 	\end{figure}
 The second assertion of Proposition \ref{prop: non ToSS} implies that $b$ is not twisted on splitting sets, and  there are in fact choices $\mu_1, \mu_2, \mu_3$ and $\mu_4$ of marginals such that $\mu_1$ is absolutely continuous with respect to Lebesgue measure and the solution to \eqref{KP} is of non-Monge form and non-unique (explicitly, take $\mu_3$ to be a Dirac mass and the other marginals to be uniform on bounded domains).  However, it is clear that the problem does admit a unique, Monge type solution as soon as both $\mu_1$ and $\mu_3$ are absolutely continuous.  The reason for this is one may solve the three marginal problem with $\mu_1$, $\mu_2$, and $\mu_3$ and the surplus $x_1 \cdot x_2 +x_1 \cdot x_3+x_2 \cdot x_3$ via the Gangbo-\'{S}wi\c{e}ch theorem \cite{Gangbo}, obtaining unique optimal maps $T_2, T_3$, and then solve independently the two marginal problems between $\mu_3$ and $\mu_4$ with surplus $x_3\cdot x_4$, yielding a unique optimal map $\bar T_4$, and between $\mu_1$ and $\mu_5$, with surplus $x_1 \cdot x_5$, yielding a unique optimal map $T_5$.   %(via Theorem \ref{Theorem 2}, for instance, or alternatively by solving the three marginal problem with $\mu_1$, $\mu_2$, and $\mu_3$ via the Gangbo-Swiech theorem \cite{Gangbo} to get $T_2$ and $T_3$, and then solving independently the two marginal problem between $\mu_1$ and $\mu_5$ ot get $T_5$) the four marginals problem with $\mu_1, \mu_2$, $\mu_3$ and $\mu_5$, and the surplus $x_1 \cdot x_2 +x_1 \cdot x_3+x_2 \cdot x_3 + x_1\cdot x_5$,  $T_5$ and then, 
 %solve independently the two marginal problem , obtaining  a unique Monge solution .  
 Since $x_4$ only interacts with $x_3$, and $x_5$ only interacts with $x_1$,  $(T_2,T_3,T_4,T_5):= (T_2,T_3,\bar T_4 \circ T_3, T_5)$ is then the unique Monge solution for the overall problem.
 
 This sort of result is not captured by Theorem \ref{Theorem 2}, as the graph extracted from the complete graph $C_5$ to yield $G$, depicted below:
 
 
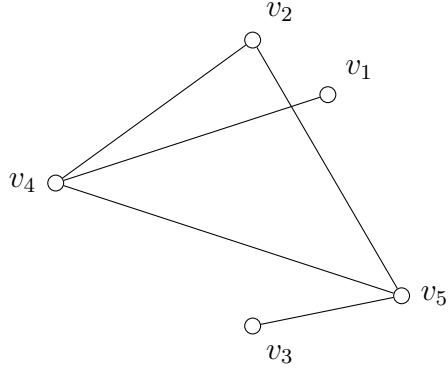
\begin{figure}[H]  
 		\begin{center}
 			\begin{tikzpicture}[x=2cm, y=2cm]
 			\vertex (v1) at (36:1) [label=36:$v_{1}$]{};
 			\vertex (v2) at (72:1) [label=72:$v_{2}$]{};
 			\vertex (v4) at (180:1) [label=180:$v_{4}$]{};
 			\vertex (v3) at (288:1) [label=288:$v_{3}$]{};
 			\vertex (v5) at (330:1.5) [label=360:$v_{5}$]{};
 			\path 
 			(v1) edge (v4)
 			(v2) edge (v4)
 			(v2) edge (v5)
 			(v3) edge (v5)
 			(v4) edge (v5)				
 			;
 			\end{tikzpicture}
 		\end{center}
 	\end{figure} 
does not have an inner hub; we develop in this section a framework that encapsulates simple examples like this one, as well as more complicated ones which cannot be treated with  adhoc  arguments like the one sketched above.

 %	Then, for any $\mu$ in $\Pi(\mu_{1},\mu_{2},\mu_{3},\mu_{4})$ we get
 %	\begin{flalign*}
 %	\displaystyle \int_{X_{1}\times X_{2}\times X_{3}\times X_{4}} b(x_{1},x_{2},x_{3},x_{4}) d\mu & =\displaystyle \int_{X_{1}\times X_{2}\times X_{3}\times X_{4}} x_{1}\cdot x_{2} d\mu \\
 %	&=\displaystyle \int_{X_{1}\times X_{2} } x_{1}\cdot x_{2}  d(P_{1,2})_{\sharp}\mu,\\
 %	\end{flalign*} 
 %	where $P_{1,2}$ is the projection of $X_{1}\times X_{2}\times X_{3}\times X_{4}$ onto $X_{1}\times X_{2}$. Hence, $\mu$ is a solution to the Kantorovich problem if and only if $(P_{1,2})_{\sharp}\mu=\gamma$, where $\gamma$ is the unique solution (induced by a map $T_{1}$) to the Kantorovich problem with surplus $ x_{1}\cdot x_{2}$ and marginals $\mu_{1},\mu_{2}$. In particular, the product measure $\mu=\gamma\otimes\mu_{3}\otimes\mu_{4}$ is a solution to the Kantorovich problem with surplus $b$ and it is not concentrated on a graph of a measurable map.
 
 \end{example}
\subsection{Monge solutions for graphs with inner hubs}
We now proceed to state and prove our second main result.
\begin{theorem}\label{Theorem 1}
Let $G$ be a graph with inner hub $A$ and maximal cliques $S_1, \ldots, S_l$, with $m=|V(G)|$, and $b$ its associated surplus. Let $\mu_{i}$ be probability measures over $X_{i}$, $i=1,\ldots, m$, with $\mu_{1}$ absolutely continuous  with respect to  $\mathcal{L}^{n}$. If there exists $p\in I(A)$ such that $\mu_{p}$ is absolutely continuous with respect to  $\mathcal{L}^{n}$, then every solution to the Kantorovich problem \eqref{KP} with surplus $b$ is induced by a map.
\end{theorem}
\begin{proof}
 Let $\gamma$ be a solution to the Kantorovich problem with surplus $b$ and $(u_{1}, \ldots, u_{m})$ a $b$-conjugate solution to its dual. Set:
\begin{align*}
\widetilde{W_{p}}&=\Big\{ (x_{1},\ldots, x_{m}): Du_{1}(x_{1})\quad \text{and} \quad  Du_{p}(x_{p})\quad\text{exist,}\quad \text{and} \quad\sum_{i=1}^{m}u_{i}(x_{i})=b(x_{1},\ldots, x_{m})\Big\}.
\end{align*}
As in Theorem \ref{Theorem 2}, we obtain $\gamma(\widetilde{W_{p}})=1$. Moreover, by fixing $x_{1}^{0}$ where $u_{1}(x_{1})$ is differentiable, we get for any $(m-1)$-tuple $(x_{2}^{0},\ldots, x_{m}^{0})$ satisfying $(x_{1}^{0},\ldots, x_{m}^{0})\in \widetilde{W_{p}}$, 
\begin{equation*}
Du_{1}(x_{1}^{0})= D_{x_{1}}b(x_{1}^{0},\ldots, x_{m}^{0}).
\end{equation*}
Let us show that the map $$(x_{2},\ldots, x_{m})\mapsto D_{x_{1}}b(x_{1}^{0},x_{2},\ldots, x_{m})$$ is injective on $\widetilde{W}_{x_{1p}^{0}}:=\left\lbrace( x_{2}, \ldots, x_{m}): (x_{1}^{0}, x_{2}, \ldots, x_{m})\in \widetilde{W_{p}}\right\rbrace$.
Indeed, assume
\begin{equation}\label{eqn:1.2-8}
D_{x_{1}}b(x_{1}^{0},x_{2}^{1},\ldots, x_{m}^{1})=\sum_{s\in I(N(v_1))}x_{s}^{1}=\sum_{s\in I(N(v_1))}x_{s}^{2}=D_{x_{1}}b(x_{1}^{0},x_{2}^{2},\ldots, x_{m}^{2}),
\end{equation} 
where $(x_{1}^{0},x_{2}^{1},\ldots, x_{m}^{1}),(x_{1}^{0},x_{2}^{2},\ldots, x_{m}^{2})\in \widetilde{W_{p}}$, and $x_{1}^{0}:=x_{1}^{1}:=x_{1}^{2}$. Recall that $G=\bigcup_{j=1}^{l}S_j$, and without lost of generality assume $v_{1}\in V(S_{1})$.  Then $v_{1}\in B_{1}$ or $v_{1}\in A $, where $B_{j}=V(S_{j})\setminus A$, $j\in \{1, \ldots, l\}$.  For the case $v_{1}\in B_{1}$, we split the proof into several steps.
\begin{enumerate}[label=\textbf{Step \arabic*.}]
\item Since $S_{1}$ is complete, for every $s\in I(B_{1})$, $N(v_s)= V(S_{1})\setminus \{v_{s}\}$, which implies $\overline{N}(v_1)=\overline{N}(v_s)$. Then, by part \ref{Part a} of Lemma \ref{Lemma:1} we get
\begin{equation}\label{eqn:1.2-19}
x_{s}^{1}=x_{s}^{2}\quad \text{for all}\quad s\in I(B_{1}).
\end{equation}
Hence, the equalities  $N(v_1)= V(S_{1})\setminus \{v_{1}\}=\left( B_{1}\cup A\right)\setminus \{v_{1}\}$, and \eqref{eqn:1.2-8}, show that
\begin{equation}\label{eqn:1.2-9}
\sum_{s\in I(A)}x_{s}^{1}=\sum_{s\in I(A)}x_{s}^{2}.
\end{equation}
\item From part \ref{Part 2} of Lemma \ref{Lemma:1}, for every  $t\in I\left( A\right)$ we get
\begin{equation}\label{eqn:1.2-10}
\left(x_{t}^{2}-x_{t}^{1}\right)\cdot\sum_{s\in I(N(v_t))}\left(x_{s}^{1}-x_{s}^{2}\right )\leq 0,
\end{equation}
and by the definition of $A$, 
\begin{flalign}\label{eqn:1.2-12}
N(v_t)& =V(G)\setminus \{v_{t}\} \nonumber\\
    &=\left(\bigcup_{j=1}^{l} V(S_{j})\right)\setminus\{v_{t}\}\nonumber\\
    &= \left( A\setminus\{v_{t}\}\right)\bigcup\left(\bigcup_{j=1}^{l} B_{j}\right).
\end{flalign} 
Thus, we can write \eqref{eqn:1.2-10} as
\begin{equation*}
\left(x_{t}^{2}-x_{t}^{1}\right)\cdot\sum_{s\in I(A)\setminus\{t\}}\left(x_{s}^{1}-x_{s}^{2}\right )+\left(x_{t}^{2}-x_{t}^{1}\right)\cdot\sum_{s\in \bigcup_{j=1}^{l} I(B_{j})}\left(x_{s}^{1}-x_{s}^{2}\right )\leq 0.
\end{equation*}
It follows from \eqref{eqn:1.2-9} that
\begin{equation}\label{eqn:1.2-14}
\Vert x_{t}^{2}-x_{t}^{1}\Vert^{2} +\left(x_{t}^{2}-x_{t}^{1}\right)\cdot\sum_{s\in \bigcup_{j=1}^{l} I(B_{j})}\left(x_{s}^{1}-x_{s}^{2}\right )\leq 0,
\end{equation}
 hence, one easily deduces
 \begin{equation}\label{eqn:1.2-11}
\left(x_{t}^{2}-x_{t}^{1}\right)\cdot\sum_{s\in \bigcup_{j=1}^{l} I(B_{j})}\left(x_{s}^{1}-x_{s}^{2}\right )\leq 0.
\end{equation}
Summing over $t\in I(A)$ we get
 \begin{equation*}
\sum_{t\in I(A)}\left(x_{t}^{2}-x_{t}^{1}\right)\cdot\sum_{s\in \bigcup_{j=1}^{l} I(B_{j})}\left(x_{s}^{1}-x_{s}^{2}\right )\leq 0,
\end{equation*}
and by \eqref{eqn:1.2-9}, we must have equality in \eqref{eqn:1.2-11} for every $t\in I(A)$. Therefore, from \eqref{eqn:1.2-14} we get
\begin{equation}\label{eqn:1.2-15}
x_{t}^{1}=x_{t}^{2}\quad \text{for all}\quad t\in I(A).
\end{equation}   In particular, $x_{p}^{1}=x_{p}^{2}$ and so, $x_{p}^{2}$ belongs to 
$$\text{Argmax}\Big\{x_{p} \mapsto \Big (\sum_{{s\in I(N(v_p))}}x_{s}^{1}\Big )\cdot x_{p}-u_{p}(x_{p})\Big\}\bigcap \text{Argmax}\Big\{x_{p} \mapsto \Big (\sum_{{s\in I(N(v_p))}}x_{s}^{2}\Big )\cdot x_{p}-u_{p}(x_{p})\Big\}.$$ It follows that
$$\sum_{{s\in I(N(v_p))}}x_{s}^{1}=Du_{p}(x_{p}^{2})=\sum_{{s\in I(N(v_p))}}x_{s}^{2},$$
or equivalently, invoking \eqref{eqn:1.2-12}, 
$$\sum_{{s\in  I(A)\setminus\{p\}}}x_{s}^{1} + \sum_{{s\in \bigcup_{j=1}^{l} I(B_{j})}}x_{s}^{1}=\sum_{{s\in  I(A)\setminus\{p\}}}x_{s}^{2} + \sum_{{s\in \bigcup_{j=1}^{l} I(B_{j})}}x_{s}^{2}.$$
It immediately implies by \eqref{eqn:1.2-15} that
\begin{equation}\label{eqn:1.2-13}
\sum_{{s\in \bigcup_{j=1}^{l} I(B_{j})}}x_{s}^{1}=\sum_{{s\in \bigcup_{j=1}^{l} I(B_{j})}}x_{s}^{2}.
\end{equation} 
%\sum_{{x\in  V(A_{1})\setminus\{x_{p_{1}}\}}}x^{1}=\sum_{{x\in  V(A_{1})\setminus\{x_{p_{1}}\}}}x^{2} \quad \text{and}\quad 
%as by \eqref{eqn:1.2-11}, $$\sum_{{s\in  I(A)\setminus\{p\}}}\left(x_{s}^{2}-x_{s}^{1}\right)\cdot\sum_{s\in \bigcup_{i=1}^{l} I(B_{i})}\left(x_{s}^{1}-x_{s}^{2}\right )\leq 0.$$ 
%Using arguments similar to those yielding. We claim that ... Indeed, 
\item Fix $k\in \{2, \ldots, l\}$. From definition \ref{Def 1}, $\left\lbrace B_{j}\right\rbrace_{j=1}^{l}$  is a disjoint collection of sets and every $j\in \{1, \ldots, l\}$ satisfies $N(v_s)=V(S_{j})\setminus \{v_{s}\}=\left( B_{j}\cup A\right)\setminus \{v_{s}\}$, for every  $s\in I(B_{j})$. Since $(x_{1}^{0},x_{2}^{1},\ldots, x_{m}^{1})\in \widetilde{W_{p}}$, we get
\begin{align*}
\left\lbrace x_{s}^{1}\right\rbrace_{s\in \bigcup_{\underset{j\neq k}{j=1}}^{l} I(B_{j})}&\in\text{Argmax}\Bigg\{\left\lbrace x_{s}\right\rbrace_{s\in \bigcup_{\underset{j\neq k}{j=1}}^{l} I(B_{j})} \mapsto \Big (\sum_{{s\in I(A)}}x_{s}^{1}\Big )\cdot \sum_{s\in \bigcup_{\underset{j\neq k}{j=1}}^{l} I(B_{j})}x_{s}\\ &
\qquad\qquad\qquad\qquad\qquad\quad \qquad+ \sum_{\underset{j\neq k}{j=1}}^{l}\sum_{\underset{s<t }{s,t\in  I(B_{j})}}x_{s}\cdot x_{t} -\sum_{s\in \bigcup_{\underset{j\neq k}{j=1}}^{l} I(B_{j})}u_{s}(x_{s})\Bigg\},\end{align*}
and by \eqref{eqn:1.2-15},
\begin{align*}
\left\lbrace x_{s}^{1}\right\rbrace_{s\in \bigcup_{\underset{j\neq k}{j=1}}^{l} I(B_{j})}&\in\text{Argmax}\Bigg\{\left\lbrace x_{s}\right\rbrace_{s\in \bigcup_{\underset{j\neq k}{j=1}}^{l} I(B_{j})} \mapsto \Big (\sum_{{s\in I(A)}}x_{s}^{2}\Big )\cdot \sum_{s\in \bigcup_{\underset{j\neq k}{j=1}}^{l} I(B_{j})}x_{s}\\ &
\qquad\qquad\qquad\qquad\qquad\quad \qquad+ \sum_{\underset{j\neq k}{j=1}}^{l}\sum_{\underset{s<t }{s,t\in  I(B_{j})}}x_{s}\cdot x_{t} -\sum_{s\in \bigcup_{\underset{j\neq k}{j=1}}^{l} I(B_{j})}u_{s}(x_{s})\Bigg\}.
\end{align*}
Hence, setting $y:=(y_{1}, y_{2}, \ldots, y_{m})$ with
\[y_{s}= \begin{cases} 
      x_{s}^{2} & \text{if}\quad s\in \{1, 2, \ldots, m\}\setminus \bigcup_{\underset{j\neq k}{j=1}}^{l} I(B_{j})= I(V(S_{k}))\\
      x_{s}^{1} & \text{if}\quad  s\in \bigcup_{\underset{j\neq k}{j=1}}^{l} I(B_{j}), \\
      
   \end{cases}
\]
we get $y \in \widetilde{W_{p}}$, as $(x_{1}^{0},x_{2}^{2},\ldots, x_{m}^{2})\in \widetilde{W_{p}}$. Since \eqref{eqn:1.2-13} holds true for every  $(x_{1}^{0},x_{2}^{1},\ldots, x_{m}^{1})$, $ (x_{1}^{0},x_{2}^{2},\ldots, x_{m}^{2})\in\widetilde{W_{p}}$; in particular, it is true for $(x_{1}^{0},x_{2}^{1},\ldots, x_{m}^{1})$ and $y$, which implies that
\begin{equation*}
\sum_{{s\in I(B_{k})}}x_{s}^{1}=\sum_{{s\in  I(B_{k})}}x_{s}^{2}.
\end{equation*}
Using \eqref{eqn:1.2-15} we can write the above equality as
\begin{equation*}
\sum_{{s\in I(V(S_{k}))}}x_{s}^{1}=\sum_{{s\in  I(V(S_{k}))}}x_{s}^{2}.
\end{equation*}

 Hence, all the elements of $I(B_{k})$ satisfy \eqref{eqn:1.2-1}, as each $s\in I(B_{k})$ satisfies $N(v_s)=V(S_{k})\setminus \{v_{s}\}$. Then, by part \ref{Part 4} of Lemma \ref{Lemma:1}, $x_{s}^{1}=x_{s}^{2}$ for all $s\in  I(B_{k})$. We thus conclude by \eqref{eqn:1.2-19} and \eqref{eqn:1.2-15} that $x_{s}^{1}=x_{s}^{2}$ for all $s\in \bigcup_{j=1}^{l} I(B_{j})\cup I(A)=I(G)=\{1,2,\ldots, m\}$. This completes the proof for the case $v_{1}\in B_{1}$.
\end{enumerate}
\par
Finally, for the case $v_{1}\in A$, note that every $s\in I(A)$ satisfies $N(v_s)= V(G)\setminus \{v_{s}\}$, hence for any $s\in I(A)$ we get 
$$\overline{N}(v_s)=\{v_{s}\}\cup N(v_s)= V(G)=\left(V(G)\setminus \{v_{1}\}\right)\cup \{v_{1}\}=N(v_1)\cup\{v_{1}\}=\overline{N}(v_1).$$
Therefore, by part \ref{Part a} of Lemma \ref{Lemma:1} we get \eqref{eqn:1.2-15}, and then,  \eqref{eqn:1.2-8} reduces to \eqref{eqn:1.2-13}. The rest of the proof runs exactly as the proof in Step 3, but instead of fixing $k$ in $\{2, \ldots, l\}$, we fix it in $\{1, \ldots, l\}$, completing the proof of the theorem.
\end{proof}
\subsection{Monge solutions for graphs glued on cliques}
We now turn to a natural extension of Theorem \ref{Theorem 1}.  The next proposition states, roughly speaking, that gluing together several graphs with inner hubs via the procedure formulated in Definition \ref{Def 2}, leads to a solution of Monge type.

\begin{proposition}\label{Proposition 2}
Let $S_1$ be a graph with inner hub $A_1$ and  $\{ S_{1 j}\}_{j=1}^{l}$  its collection of maximal cliques.  Let $E\subset \{2, \ldots, l\}$ such that for every $\alpha\in E$, $S_\alpha$ is a graph with inner hub $A_\alpha \neq \emptyset$, and with collection of maximal cliques  $\{ S_{\alpha j}\}_{j=1}^{k_{\alpha}}$. Assume $A_{\alpha}\cap A_{1}= \emptyset$ for every $\alpha \in E$, and set  $G=\bigcup_{\alpha\in E\cup\{1\}}S_{\alpha}$ and  $m=|V(G)|$. Let $\mu_{i}$ be probability measures over $X_{i}$, $i=1,\ldots, m$ and assume:
%Let  $\{ S_{1 j}\}_{j=1}^{l}$ be a semi-disjoint collection of graphs with intersection set $A_{1}\neq \emptyset$. Let $E\subset \{2, \ldots, l\}$ such that for every $\alpha\in E$, $\{ S_{\alpha j}\}_{j=1}^{k_{\alpha}}$ is a semi-disjoint collection of graphs with intersection set $A_{\alpha}\neq \emptyset$ and $A_{\alpha}\cap A_{1}= \emptyset$. Set $S_{\alpha}=\bigcup_{j=1}^{k_{\alpha}} S_{\alpha j}$, with $\alpha \in E  \cup \{1\}$ and $k_{1}=l$, $G=\bigcup_{\alpha\in E\cup\{1\}}S_{\alpha}$ and  $m=|V(G)|$. Let $\mu_{i}$ be probability measures over $X_{i}$, $i=1,\ldots, m$ and assume:
\begin{enumerate}
\item $S_{\alpha}$ and $S_{1}$ are glued on a clique for all $\alpha \in E$.
%and $S_{\alpha}\neq S_{1\alpha }$, for all $\alpha \in E$.
\item $V(S_{\alpha})\bigcap V(S_{\beta})=A_{1}$ for all $\alpha\neq \beta$,  $\alpha,\beta\in E$.
\item For each $\alpha\in E\cup \{1\}$, there exists $p_{\alpha}\in I(A_{\alpha})$ such that $\mu_{p_{\alpha}}$ is absolutely continuous  with respect to $\mathcal{L}^{n}$.
\item $\mu_{1}$ is absolutely continuous with respect to  $\mathcal{L}^{n}$ and $v_{1}\in V( S_{11})$.
\end{enumerate}
Then every solution to the Kantorovich problem \eqref{KP} with surplus associated to $G$ is concentrated on a graph of a measurable map.
\end{proposition}
\begin{proof}
The strategy of the proof is similar to the strategy used in Theorem \ref{Theorem 1}. Let $\gamma$ be a solution to the Kantorovich problem with surplus $b(x_{1}, \ldots,x_{m})$,  where $b$ is the surplus associated to $G$. Let $(u_{1}, \ldots, u_{m})$ be a $b$-conjugate solution to its dual and set
\begin{align*}
\widetilde{W}&=\Big\{ (x_{1},\ldots, x_{m}): Du_{1}(x_{1})\quad \text{and} \quad  Du_{p_{\alpha}}(x_{p_{\alpha}}) \quad\text{exist}\quad\text{for all $\alpha\in E \cup\{1\}$},\\
&\qquad \qquad \qquad\qquad\text{and} \quad\sum_{i=1}^{m}u_{i}(x_{i})=b(x_{1},\ldots, x_{m})\Big\}.
\end{align*}
Fix $x_{1}^{0}\in spt(\mu_{1})$, where $u_{1}(x_{1})$ is differentiable. Then $Du_{1}(x_{1}^{0})= D_{x_{1}}b(x_{1}^{0},\ldots, x_{m}^{0})$, for every $(x_{1}^{0},\ldots, x_{m}^{0})\in \widetilde{W}$.
We want to prove that the map $(x_{2}, \ldots, x_{m})\mapsto D_{x_{1}}b(x_{1}^{0},x_{2}, \ldots,x_{m})$ is injective on $$\widetilde{W}_{x_{1}^{0}}:=\left\lbrace( x_{2}, \ldots, x_{m}): (x_{1}^{0}, x_{2}, \ldots, x_{m})\in \widetilde{W}\right\rbrace.$$ Assume
\begin{equation}\label{eqn:1.2-48}
D_{x_{1}}b(x_{1}^{0},x_{2}^{1},\ldots, x_{m}^{1})=\sum_{s\in I(N(v_1))}x_{s}^{1}=\sum_{s\in I(N(v_1))}x_{s}^{2}=D_{x_{1}}b(x_{1}^{0},x_{2}^{2},\ldots, x_{m}^{2}),
\end{equation} 
with $(x_{1}^{0},x_{2}^{1},\ldots, x_{m}^{1}),(x_{1}^{0},x_{2}^{2},\ldots, x_{m}^{2})\in \widetilde{W}$ and $x_{1}^{0}:= x_{1}^{1}=x_{1}^{2}$. Note that if   $E=\emptyset$, we get $G=S_{1}$, and then, by Theorem \ref{Theorem 1} we get a solution of Monge type.
Assume $E\neq \emptyset$ and set $B_{j}= V(S_{1j})\setminus A_{1}$, where $j \in \{1,\ldots, l\}$. Since $A_{j}\cap A_{1}= \emptyset$ for every $j \in E$, $B_{j}\neq \emptyset$ for every $j \in E$ and 
\begin{equation}\label{eqn:1.2-38}
N(v_s)=V(S_{1})\setminus \{v_{s}\}=\bigcup_{j=1}^{l} B_{j}\cup \left(A_{1}\setminus \{v_{s}\}\right), \quad \text{for every}\quad s\in I(A_{1}).
\end{equation} 
Furthermore, by assumption 1 we can assume without lost of generality that $S_{1\alpha}=S_{\alpha 1}$ for every $\alpha \in E$. As in Theorem \ref{Theorem 1}, we consider two cases, $v_{1}\in B_{1}$ or $v_{1}\in A_{1}$. Let us divide the proof of case $v_{1}\in B_{1}$ into several steps:
\begin{enumerate}[label=\textbf{Step \arabic*.}]
\item We proceed to make a straightforward adaptation of the arguments used in Step 3 of the proof of Theorem \ref{Theorem 1}. First, note that   $N(v_s)=V(S_{11})\setminus\{v_{s}\}$, for every $s\in I(B_{1})$, then, using the differentiability of $u_{p_{1}}(x_{p_{1}})$ at $x_{p_{1}}^{1}$ and $x_{p_{1}}^{2}$,  and the equalities \eqref{eqn:1.2-48} and \eqref{eqn:1.2-38}, we can mirror steps 1 and 2 in the proof of Theorem \ref{Theorem 1} to get:
\begin{equation}\label{eqn:1.2-40}
x_{s}^{1}=x_{s}^{2}\quad \text{for all}\quad s\in I(B_{1}),
\end{equation}
\begin{equation}\label{eqn:1.2-16}
x_{s}^{1}=x_{s}^{2}\quad \text{for all}\quad s\in I(A_{1}),
\end{equation}
and
\begin{equation}\label{eqn:1.2-17}
\sum_{{s\in \bigcup_{j=1}^{l} I(B_{j})}}x_{s}^{1}=\sum_{{s\in \bigcup_{j=1}^{l}I(B_{j})}}x_{s}^{2}.
\end{equation}
\item 
%Fix $\alpha\in \{2, \ldots, l\}$ and set $S_{\beta}=S_{1\beta}$ for any $\beta\in \{2, \ldots, l\}\setminus E$. Define $\mathcal{I}_{1}=\bigcup_{\underset{\beta\neq \alpha}{\beta=2}}^{l} I(V(S_{\beta})\setminus A_{1})=\{t_{1}, \ldots, t_{d}\}$, $\mathcal{I}_{2}=I(B_{1})=\{r_{1}, \ldots, r_{e}\}$ and $\mathcal{I}_{3}=I(V(S_{\alpha})\setminus A_{1})=\{v_{1}, \ldots, v_{f}\}$.
%We claim that $x_{s}^{1}=x_{s}^{2}$ for all $s\in \mathcal{I}_{3}$. Indeed, note that
Fix $\alpha\in \{2, \ldots, l\}$ and set $S_{\beta}=S_{1\beta}$ for any $\beta\in \{ 2, \ldots, l\}\setminus E$. Define $\mathcal{I}_{1}=\bigcup_{\underset{\beta\neq \alpha}{\beta=2}}^{l} I(V(S_{\beta})\setminus A_{1})\cup I(B_{1})=\{t_{1}, \ldots, t_{d}\}$ and $\mathcal{I}_{2}=I(V(S_{\alpha})\setminus A_{1})=\{r_{1}, \ldots, r_{e}\}$.
We claim that $x_{s}^{1}=x_{s}^{2}$ for all $s\in \mathcal{I}_{2}$, this will complete the proof. Indeed, note that
\begin{flalign}\label{eqn:1.2-37}
\{1, \ldots, m\}& =\bigcup_{\beta=2}^{l} I(V(S_{\beta}))\cup I(B_{1}) \nonumber\\
  &= \left(\bigcup_{\underset{\beta\neq \alpha}{\beta=2}}^{l} I(V(S_{\beta}))\cup I(B_{1})\right)\cup I(V(S_{\alpha})) \nonumber\\
   &=\left(\bigcup_{\underset{\beta\neq \alpha}{\beta=2}}^{l} I(V(S_{\beta})\setminus A_{1})\cup I(B_{1})\right)\cup I(V(S_{\alpha})\setminus A_{1})\cup A_{1}\nonumber\\
    &= \mathcal{I}_{1} \cup \mathcal{I}_{2}\cup A_{1}
\end{flalign}
Furthermore, the last union is disjoint by assumptions 1 and 2. Now, let $g_{1}(x_{t_{1}},\ldots,x_{t_{d}})$ and $g_{2}(x_{r_{1}},\ldots,x_{r_{e}})$  be the functions formed by all the terms of $b$ that depend only on the  variables with index in $\mathcal{I}_{1}$ and  $\mathcal{I}_{2}$ respectively. From Definition 2.3 and assumptions 1 and 2,  it is not hard to deduces that  %$\bigcup_{\underset{\beta\neq \alpha}{\beta=2}}^{l} \left(V(S_{\beta})\setminus A_{1}\right)$, $B_{1}$ and $V(S_{\alpha})\setminus A_{1}$ respectively. From Definition 2.3 and assumptions 1 and 2,  it is not hard to deduces that 
\begin{equation*}
\bigcup_{s\in \mathcal{I}_{k}}N(v_s)= \{v_{s}\}_{s\in \mathcal{I}_{k}}\cup A_{1}, k=1,2.
\end{equation*}
Combining the above equalities, \eqref{eqn:1.2-37} and \eqref{eqn:1.2-38} we get
%\begin{align*}
%b(x_{1},\ldots, x_{m})&=g_{1}(x_{t_{1}},\ldots,x_{t_{d}})+ g_{2}(x_{r_{1}},\ldots,x_{r_{e}})+ g_{3}(x_{v_{1}},\ldots,x_{v_{f}}) \\ &
 %+ \Big(\sum_{{s\in I(A_{1})}}x_{s}\Big )\cdot  \sum_{s\in  \bigcup_{j=1}^{l} I(B_{j})}x_{s} + \sum_{\underset{s<t }{s,t\in  I(A_{1})}}x_{s}\cdot x_{t}\\
% &= g_{1}(x_{t_{1}},\ldots,x_{t_{d}})+ g_{2}(x_{r_{1}},\ldots,x_{r_{e}})+ g_{3}(x_{v_{1}},\ldots,x_{v_{f}}) \\ &
 %+ \Big(\sum_{{s\in I(A_{1})}}x_{s}\Big )\cdot   \sum_{s\in  \bigcup_{\underset{j\neq \alpha}{j=2}}^{l} I(B_{j})}x_{s} +  \Big(\sum_{{s\in I(A_{1})}}x_{s}\Big )\cdot\sum_{s\in   I(B_{\alpha}\cup B_{1})}x_{s} + \sum_{\underset{s<t }{s,t\in  I(A_{1})}}x_{s}\cdot x_{t}.
%\end{align*}
\begin{align*}
b(x_{1},\ldots, x_{m})&=g_{1}(x_{t_{1}},\ldots,x_{t_{d}})+ g_{2}(x_{r_{1}},\ldots,x_{r_{e}})\\ &
 + \Big(\sum_{{s\in I(A_{1})}}x_{s}\Big )\cdot  \sum_{s\in  \bigcup_{j=1}^{l} I(B_{j})}x_{s} + \sum_{\underset{s<t }{s,t\in  I(A_{1})}}x_{s}\cdot x_{t}\\
 &= g_{1}(x_{t_{1}},\ldots,x_{t_{d}})+ g_{2}(x_{r_{1}},\ldots,x_{r_{e}}) \\ &
 + \Big(\sum_{{s\in I(A_{1})}}x_{s}\Big )\cdot   \sum_{s\in  \bigcup_{\underset{j\neq \alpha}{j=1}}^{l} I(B_{j})}x_{s} +  \Big(\sum_{{s\in I(A_{1})}}x_{s}\Big )\cdot\sum_{s\in   I(B_{\alpha})}x_{s} + \sum_{\underset{s<t }{s,t\in  I(A_{1})}}x_{s}\cdot x_{t}.
\end{align*}
Note that
\begin{equation}\label{eqn:1.2-39}
 \bigcup_{\underset{j\neq \alpha}{j=1}}^{l} I(B_{j})\subset \mathcal{I}_{1}, 
\end{equation}
and the only terms of $b$ that depend on the variables with index in $\mathcal{I}_{1}$ are $g_{1}(x_{t_{1}},\ldots,x_{t_{d}})$ and $\Big(\sum_{{s\in I(A_{1})}}x_{s}\Big )\cdot   \sum_{s\in  \bigcup_{\underset{j\neq \alpha}{j=1}}^{l} I(B_{j})}x_{s}$. Hence,
\begin{align*}
\left\lbrace x_{s}^{1}\right\rbrace_{s\in \mathcal{I}_{1}}&\in\text{Argmax}\Bigg\{\left\lbrace x_{s}\right\rbrace_{s\in \mathcal{I}_{1}} \mapsto \Big (\sum_{{s\in I(A_{1})}}x_{s}^{1}\Big )\cdot \sum_{s\in  \bigcup_{\underset{j\neq \alpha}{j=1}}^{l} I(B_{j})}x_{s} + g_{1}(x_{t_{1}},\ldots,x_{t_{d}})\\ &
\qquad\qquad\qquad  -\sum_{s\in\mathcal{I}_{1}}u_{s}(x_{s}) + g_{2}(x_{r_{1}}^{1},\ldots,x_{r_{e}}^{1})+  \Big(\sum_{{s\in I(A_{1})}}x_{s}^{1}\Big )\cdot\sum_{s\in   I(B_{\alpha})}x_{s}^{1}\\
& \qquad\qquad\qquad + \sum_{\underset{s<t }{s,t\in  I(A_{1})}}x_{s}^{1}\cdot x_{t}^{1}-\sum_{s\in\mathcal{I}_{2}\cup I(A_{1})}u_{s}(x_{s}^{1})\Bigg\}\\
&=\text{Argmax}\Bigg\{\left\lbrace x_{s}\right\rbrace_{s\in \mathcal{I}_{1}} \mapsto \Big (\sum_{{s\in I(A_{1})}}x_{s}^{1}\Big )\cdot \sum_{s\in  \bigcup_{\underset{j\neq \alpha}{j=1}}^{l} I(B_{j})}x_{s} + g_{1}(x_{t_{1}},\ldots,x_{t_{d}})\\&
\qquad\qquad\qquad  -\sum_{s\in\mathcal{I}_{1}}u_{s}(x_{s})\Bigg\},\\
\end{align*}
and by \eqref{eqn:1.2-16},
\begin{align*}
\left\lbrace x_{s}^{1}\right\rbrace_{s\in \mathcal{I}_{1}}&\in\text{Argmax}\Bigg\{\left\lbrace x_{s}\right\rbrace_{s\in \mathcal{I}_{1}} \mapsto \Big (\sum_{{s\in I(A_{1})}}x_{s}^{2}\Big )\cdot \sum_{s\in  \bigcup_{\underset{j\neq \alpha}{j=1}}^{l} I(B_{j})}x_{s} +  g_{1}(x_{t_{1}},\ldots,x_{t_{d}}) \\&
\qquad\qquad\qquad  -\sum_{s\in\mathcal{I}_{1}}u_{s}(x_{s})\Bigg\}.
\end{align*}
Since  $(x_{1}^{0},x_{2}^{2},\ldots, x_{m}^{2})\in \widetilde{W}$, we obtain  $y:=(y_{1}, y_{2}, \ldots, y_{m})\in\widetilde{W}$, where
\[y_{s}= \begin{cases} 
      x_{s}^{2} & \text{if}\quad s\in \{1, 2, \ldots, m\}\setminus \mathcal{I}_{1}\\
      x_{s}^{1} & \text{if}\quad  s\in\mathcal{I}_{1}\\
      
   \end{cases}
\]
 Therefore, \eqref{eqn:1.2-17} holds true for $y$ and $(x_{1}^{0},x_{2}^{1},\ldots, x_{m}^{1})\in \widetilde{W}$; that is,
 \begin{equation*}
 \sum_{{s\in \bigcup_{j=1}^{l} I(B_{j})}}y_{s}=\sum_{{s\in \bigcup_{j=1}^{l}I(B_{j})}}x_{s}^{1},
\end{equation*}
or equivalently,
\begin{equation*}
\sum_{{s\in I(B_{\alpha})}}y_{s} +  \sum_{s\in  \bigcup_{\underset{j\neq \alpha}{j=1}}^{l} I(B_{j})}y_{s} = \sum_{{s\in I(B_{\alpha})}}x_{s}^{1} + \sum_{s\in  \bigcup_{\underset{j\neq \alpha}{j=1}}^{l} I(B_{j})}x_{s}^{1}.
\end{equation*}
 By the above equality, \eqref{eqn:1.2-39} and construction of $y$ we get
\begin{equation}\label{eqn:1.2-30}
 \sum_{{s\in I(B_{\alpha})}}x_{s}^{2}=\sum_{{s\in  I(B_{\alpha})}}x_{s}^{1}.
\end{equation}
\item Since $B_{\alpha }= V(S_{1\alpha })\setminus A_{1}$, by \eqref{eqn:1.2-16} and the above equality we can write,
\begin{equation}\label{eqn:1.2-20}
\sum_{{s\in I(V(S_{1\alpha }))}}x_{s}^{1}=\sum_{{s\in  I(V(S_{1\alpha }))}}x_{s}^{2}.
\end{equation}

Now, if $\alpha \in \{2, \ldots, m\}\setminus E$, then  $S_{\alpha}=S_{1\alpha}$ and $N(v_s)=V(S_{1\alpha })\setminus \{v_{s}\}$ for any $s\in I(B_{\alpha })$. Hence, from \eqref{eqn:1.2-20} we get \eqref{eqn:1.2-1} on $I(B_{\alpha })$, implying $x_{s}^{1}=x_{s}^{2}$ on $I(V(S_{\alpha}))$, by part \ref{Part 4} of Lemma \ref{Lemma:1} and \eqref{eqn:1.2-16}. \par On the other hand, if $\alpha \in E$, the equality $A_{\alpha} \cap A_{1}=\emptyset$ implies $A_{\alpha} \subseteq V(S_{\alpha 1 })\setminus A_{1}=V(S_{1\alpha })\setminus A_{1}=B_{\alpha }$. It follows by \eqref{eqn:1.2-20} that equality \eqref{eqn:1.2-1} holds  for the elements of $I(B_{\alpha }\setminus A_{\alpha})$, as every $s\in I(B_{\alpha }\setminus A_{\alpha})$ satisfies  $N(v_s)=V(S_{1\alpha })\setminus \{v_{s}\}$. Then, by part \ref{Part 4} of Lemma \ref{Lemma:1},
\begin{equation}\label{eqn:1.2-58}
x_{s}^{1}=x_{s}^{2}\quad \text{for all}\quad s\in I(B_{\alpha }\setminus A_{\alpha}), 
\end{equation} 
and by \eqref{eqn:1.2-30},
 $$\sum_{{s\in I(A_{\alpha })}}x_{s}^{1}=\sum_{{s\in  I(A_{\alpha })}}x_{s}^{2}.$$
Note that by the differentiability of $u_{p_{\alpha}}(x_{p_{\alpha}})$ at $x_{p_{\alpha}}^{1}$ and $x_{p_{\alpha}}^{2}$, we can apply to the graph $S_{\alpha}=\bigcup_{j=1}^{k_{\alpha}} S_{\alpha j}$, the same arguments discussed in Step 2 of the proof of Theorem \ref{Theorem 1}, getting 
\begin{equation}\label{eqn:1.2-56}
 x_{s}^{1}=x_{s}^{2}\quad \text{for all}\quad s\in I(A_{\alpha}), \quad \text{and}\quad \sum_{{s\in \bigcup_{j=1}^{k_{\alpha}} I(B_{\alpha j})}}x_{s}^{1}=\sum_{{s\in \bigcup_{j=1}^{k_{\alpha}}I(B_{\alpha j})}}x_{s}^{2},
\end{equation}
where $B_{\alpha j}=V(S_{\alpha j })\setminus A_{\alpha}$ for all $j\in \{1, \ldots, k_{\alpha}\}$. By the left-hand equality of the above results, \eqref{eqn:1.2-16} and \eqref{eqn:1.2-58}, we get 
\begin{equation}\label{eqn:1.2-59}
 x_{s}^{1}=x_{s}^{2}\;\;\text{on}\;\; I(V(S_{1\alpha }))=I(V(S_{\alpha 1})).
\end{equation}
 Next, we fix $r\in \{2, \ldots, k_{\alpha}\}$ and proceed to apply the same strategy used in step 2: we set $\mathcal{I}_{1}^{\prime}=\{1, \ldots, m\}\setminus I(V(S_{\alpha r }) )=\{e_{1}, \ldots, e_{q}\}$ and $\mathcal{I}_{2}^{\prime}=I(B_{\alpha r})=\{d_{1}, \ldots, d_{f}\}$, and consider $g_{1}^{\prime}(x_{e_{1}},\ldots,x_{e_{q}})$ and $g_{2}^{\prime}(x_{d_{1}},\ldots,x_{d_{f}})$, the functions formed by all the terms of $b$ that depend only on the vertices with index in $\mathcal{I}_{1}^{\prime}$ and  $\mathcal{I}_{2}^{\prime}$ respectively. 
Noting that $\bigcup_{s\in \mathcal{I}_{j}^{\prime}}N(v_s)= \{v_{s}\}_{s\in \mathcal{I}_{j}^{\prime}}\cup A_{\alpha}, j=1,2$, and using the left-hand equality  in \eqref{eqn:1.2-56}, we follow the arguments of Step 2 to get
\begin{align*}
\left\lbrace x_{s}^{1}\right\rbrace_{s\in \mathcal{I}_{1}^{\prime}}&\in\text{Argmax}\Bigg\{\left\lbrace x_{s}\right\rbrace_{s\in \mathcal{I}_{1}^{\prime}} \mapsto \Big (\sum_{{s\in I(A_{\alpha})}}x_{s}^{1}\Big )\cdot \sum_{s\in  \bigcup_{\underset{j\neq r}{j=1}}^{k_{\alpha}} I(B_{\alpha j})}x_{s} +  g_{1}^{\prime}(x_{e_{1}}, \ldots, x_{e_{q}})\\ &
\qquad\qquad\qquad  -\sum_{s\in\mathcal{I}_{1}^{\prime}}u_{s}(x_{s}) + g_{2}^{\prime}(x_{d_{1}}^{1},\ldots,x_{d_{f}}^{1})+  \Big(\sum_{{s\in I(A_{\alpha})}}x_{s}^{1}\Big )\cdot\sum_{s\in   I(B_{\alpha r})}x_{s}^{1}\\
& \qquad\qquad\qquad + \sum_{\underset{s<t }{s,t\in  I(A_{\alpha})}}x_{s}^{1}\cdot x_{t}^{1}-\sum_{s\in\mathcal{I}_{2}^{\prime}\cup I(A_{\alpha})}u_{s}(x_{s}^{1})\Bigg\}\\
&=\text{Argmax}\Bigg\{\left\lbrace x_{s}\right\rbrace_{s\in \mathcal{I}_{1}^{\prime}} \mapsto \Big (\sum_{{s\in I(A_{\alpha})}}x_{s}^{1}\Big )\cdot \sum_{s\in  \bigcup_{\underset{j\neq r}{j=1}}^{k_{\alpha}} I(B_{\alpha j})}x_{s} +  g_{1}^{\prime}(x_{e_{1}}, \ldots, x_{e_{q}})\\ &
\qquad\qquad\qquad  -\sum_{s\in\mathcal{I}_{1}^{\prime}}u_{s}(x_{s}) \Bigg\}\\
&=\text{Argmax}\Bigg\{\left\lbrace x_{s}\right\rbrace_{s\in \mathcal{I}_{1}^{\prime}} \mapsto \Big (\sum_{{s\in I(A_{\alpha})}}x_{s}^{2}\Big )\cdot \sum_{s\in  \bigcup_{\underset{j\neq r}{j=1}}^{k_{\alpha}} I(B_{\alpha j})}x_{s} +  g_{1}^{\prime}(x_{e_{1}}, \ldots, x_{e_{q}})\\ &
\qquad\qquad\qquad  -\sum_{s\in\mathcal{I}_{1}^{\prime}}u_{s}(x_{s}) \Bigg\},
\end{align*}
and then, using the right-hand equality in   \eqref{eqn:1.2-56}, we get the equality \eqref{eqn:1.2-30} on $I(B_{\alpha r})$; that is, 
\begin{equation}\label{eqn:1.2-57}
 \sum_{{s\in I(B_{\alpha r})}}x_{s}^{2}=\sum_{{s\in  I(B_{\alpha r})}}x_{s}^{1}.
\end{equation}
Finally, we combine the above equality with the left-hand equality in \eqref{eqn:1.2-56} to get \eqref{eqn:1.2-1} for all $s\in I(B_{\alpha r})$, since $N(v_s)=V(S_{\alpha r })\setminus \{v_{s}\}$ for every $s\in I(B_{\alpha r})$ and $B_{\alpha j}=V(S_{\alpha j })\setminus A_{\alpha}$. Hence, by part \ref{Part 4} of Lemma \ref{Lemma:1}, $x_{s}^{1}=x_{s}^{2}$ for all $s\in I(B_{\alpha r})$. Thus, $x_{s}^{1}=x_{s}^{2}$ for all $s\in \bigcup_{j=2}^{k_{\alpha}}I(B_{\alpha j})=I(V(S_{\alpha})\setminus V(S_{\alpha 1})) $. Hence, from \eqref{eqn:1.2-59} we conclude  $x_{s}^{1}=x_{s}^{2}$ on  $I(V(S_{\alpha}))$, and so, $x_{s}^{1}=x_{s}^{2}$ for all $s\in \bigcup_{\alpha =1}^{l} I(V(S_{\alpha}))= \{1, \ldots, m\}$, completing the proof of the case $v_{1}\in B_{1}$.
\end{enumerate}

 For the case $v_{1}\in A_{1}$, every $s\in I(A_{1})$ satisfies $N(v_s)= V(S_{1})\setminus \{v_{s}\}$, then any $s\in I(A_{1})$ satisfies $\overline{N}(v_s)=\overline{N}(v_1)$. Therefore, by part \ref{Part a} of Lemma \ref{Lemma:1} we get \eqref{eqn:1.2-16}, and \eqref{eqn:1.2-48} reduces to \eqref{eqn:1.2-17}. For the rest of the proof we fix $\alpha \in\{1, \ldots, l\}$ and mimic the proof of the case $v_{1}\in B_{1}$, completing the proof of the theorem.
\end{proof}
\begin{remark} Th results developed in this section, for  graphs with inner hubs glued on their cliques, are neither more or less general than Theorem 3.1, which applies to graphs obtained by extracting subgraphs with inner hubs from complete graphs. To see this, note that in Theorem \ref{Theorem 2}, if $m=4$ and $l=2$, with $S_{1}=\{x_{1},x_{3}\}$ and $S_{2}=\{x_{2},x_{4}\}$, we get the surplus associated to the graph in Figure \ref{Graph cycle.}, which clearly cannot be obtained from the results of Section 4. On the other hand, we can find examples of surplus functions covered by the framework presented in Section 4, but not covered by Theorem \ref{Theorem 2}. For instance, Figure \ref{Path} and \ref{Tree} are graphs whose respective surplus are not covered by Theorem \ref{Theorem 2}, as we need more than two absolutely continuous measures and clearly, these conditions are necessary.
\end{remark}

 We next turn to a slight generalization of Proposition \ref{Proposition 2}, where, roughly speaking, any two graphs (with inner hubs) can be glued together (unlike in the preceding proposition, where each $S_\alpha$, $\alpha \in E$ was glued to $S_1$).  The proof is a straightforward modification of the proof of Proposition \ref{Proposition 2} and is therefore omitted.
 \par
In order to facilitate the description of the next Proposition we will introduce a natural higher level notion of graph. For this, we interpret any collection of graphs with inner hubs $\{G_\alpha\}_{\alpha=1}^{l}$, as the vertices of a graph $\mathcal{G}$, whose edges are glueings on cliques  between the $G_{\alpha}$ and $G_{\beta}$; that is,
 $$V(\mathcal{G})=\{G_\alpha\}_{\alpha=1}^{l}$$
 and
 $$E(\mathcal{G})=\left\lbrace \{ G_{\alpha}, G_{\beta}\}: G_{\alpha}\;\; \text{is glued on a clique to}\;\;G_{\beta} \right\rbrace.$$
 \begin{proposition}\label{Proposition 3}
Let $\{G_\alpha\}_{\alpha=1}^{l}$ be a collection of graphs with inner hubs $A_{\alpha}$, and $\mathcal{G}$ its associated higher order graph (described above).  Let $m=|\bigcup_{\alpha=1}^{l}V(G_{\alpha})|$ and $\mu_{i}$ be probability measures over $X_{i}$, $i=1,\ldots, m$, where without loss of generality $v_{1}\in V(G_{1})$. Assume:
\begin{enumerate}
\item For each distinct $\alpha \neq \beta$, $A_{\alpha}\cap A_{\beta}=\emptyset$ and $V(G_{\alpha})\cap V(G_{\beta})$ is either:
\begin{itemize}
\item empty,
\item the vertex set $V(S)$, where $S$ is a maximal clique $S$ of both $G_{\alpha}$ and $G_{\beta}$ (in this case $G_{\alpha}$ and $G_{\beta}$ are glued on a clique $S$), or
\item $A_\lambda$ for some other $G_\lambda$ (as when $G_\alpha$ and $G_\beta$ are both glued to $G_\lambda$).
\end{itemize}
\item  $\mu_1$ is  absolutely continuous  with respect to $\mathcal{L}^{n}$, and, for each $\alpha \in \{1, \ldots, l\}$, there exists $p_{\alpha}\in I(A_{\alpha})$ such that $\mu_{p_{\alpha}}$ is absolutely continuous  with respect to $\mathcal{L}^{n}$.
\item For at least one maximal clique $S$ having $v_1$ as one of its vertices, $G_1$ is not glued to any other $G_{\alpha}$ on $S$.
\item  $\mathcal{G}$ is a tree.
\end{enumerate}
Then every solution to the Kantorovich problem with surplus $\bigcup_{\alpha=1}^{l}G_{\alpha}$ is induced by a map.
\end{proposition}
\begin{remark}
Using the terminology developed  above, the assumptions in Proposition \ref{Proposition 2} are equivalent to the assumptions in Proposition \ref{Proposition 3}, except that the hypothesis that $\mathcal{G}$ is a tree is replaced with the hypothesis that $\mathcal{G}$ is a star with internal node $S_1$.  Therefore, Proposition \ref{Proposition 3} is a direct generalization of Proposition \ref{Proposition 2}.
%In the same setting of the above proposition, we can interpret the graph $G$ in Proposition \ref{Proposition 2} as a star. 
\end{remark}

\subsection{Examples}
Let us illustrate the results obtained in Section 4 throughout some examples.
In what follows, $\mu_{1}$ is absolutely continuous.
\begin{examples}
\begin{enumerate}[label=(\roman*)]
\item In Theorem \ref{Theorem 1}, if $S_{j}=S_{k}$ for every $j,k\in \{1, \ldots, l\}$, then $\bigcup_{j=1}^{l} S_{j}$ reduces to the Gangbo and \'{S}wi\c{e}ch surplus.

\item By Theorem \ref{Theorem 1}, the graph $S_{1}\cup S_{2}\cup S_{3}$ in Example \ref{Example 1}  provides a Monge solution, with $\mu_{p}$ absolutely continuous for some $p\in \{6,7,8\}$. 
\item In Example \ref{Example 2}, if there are $p_{1}\in I(A)$ and $p_{2}\in I(A^{\prime})$ such that $\mu_{p_{1}}$ and $\mu_{p_{2}}$ are absolutely continuous, then by Proposition \ref{Proposition 2} the graph $\left(\bigcup_{j=1}^{3}S_{j}\right)\bigcup \left(\bigcup_{j=1}^{4}S_{j}^{\prime}\right)$ provides a solution of Monge type .
\item By Theorem \ref{Theorem 1}, any graph of the form $K_{1,k}$ ( known as a star graph) provides a solution of Monge type, under at most two regularity conditions (see pictures below). Note that $|V(K_{1,k})|=k+1$ and there exists $v\in V(K_{1,k})$ such that $N(v)=\{v_{1}, \ldots, v_{k}\}$. Additionally, $N(v_s)=\{v\}$ for all $s\in \{1, \ldots, k\}$. This is one of the most simple graphs  providing Monge solutions that we could obtain, since  a graph with inner hub have in fact a "star shape". Note that, in the general setting, the single set $\{v\}$ is replaced by the inner hub $A$ and $\{v_{j}\}$ is replaced by $B_{j}:=V(S_{j})\setminus A$, $j=1, \ldots, k$, where $\{S_{j}\}_{j=1}^{l}$ is the collection of maximal cliques. See for instance Figure \ref{fig:13}.
\begin{figure}[H]   
\begin{center}
\begin{subfigure}[b]{0.4\linewidth}
\begin{tikzpicture}[x=1.5cm, y=1.5cm]
	\vertex (v1) at (60:1) [label=60:$v_{1}$]{};
	\vertex (v2) at (120:1) [label=120:$v_{2}$]{};
	\vertex (v3) at (180:1) [label=180:$v_{3}$]{};
	\vertex (v4) at (240:1) [label=240:$v_{4}$]{};
	\vertex (v5) at (300:1) [label=300:$v_{5}$]{};
	\vertex (v6) at (360:1) [label=360:$v_{6}$]{};
	\vertex (v7) at (360:0) [label=5:$v_{7}$]{};
	\path 
		(v6) edge (v7)
		(v7) edge (v1)
		(v7) edge (v2)
		(v7) edge (v4)
		(v7) edge (v3)
		(v7) edge (v5)
		;
\end{tikzpicture}
\caption{$K_{1,6}$, with $V_{1}=\{v_{7}\}$ and $V_{2}=\{v_{i}\}_{i=1}^{6}$. 
Here, we need regularity conditions on $\mu_{1}$ and $\mu_{7}$.} 
  \end{subfigure}
  \qquad \qquad
\begin{subfigure}[b]{0.4\linewidth}
\begin{tikzpicture}[x=1.5cm, y=1.5cm]
	\vertex (v7) at (60:1) [label=60:$v_{7}$]{};
	\vertex (v2) at (120:1) [label=120:$v_{2}$]{};
	\vertex (v3) at (180:1) [label=180:$v_{3}$]{};
	\vertex (v4) at (240:1) [label=240:$v_{4}$]{};
	\vertex (v5) at (300:1) [label=300:$v_{5}$]{};
	\vertex (v6) at (360:1) [label=360:$v_{6}$]{};
	\vertex (v1) at (360:0) [label=5:$v_{1}$]{};
	\path 
		(v6) edge (v1)
		(v1) edge (v7)
		(v1) edge (v2)
		(v1) edge (v4)
		(v1) edge (v3)
		(v1) edge (v5)
		;
\end{tikzpicture}
\caption{$K_{1,6}$, with $V_{1}=\{v_{1}\}$ and $V_{2}=\{v_{i}\}_{i=2}^{7}$.  Here, we only need a regularity condition on $\mu_{1}$.} 
\end{subfigure}
\caption{}
\end{center}
\end{figure}
\begin{figure}[H]  
\begin{center}
\begin{tikzpicture}[x=1.5cm, y=1.5cm]
	\vertex (v1) at (95:2) [label=95:$v_{1}$]{};
	\vertex (v7) at (85:2) [label=85:$v_{7}$]{};
	\vertex (v2) at (180:0.5) [label=20:$v_{2}$]{};
	\vertex (v3) at (270:0.5) [label=270:$v_{3}$]{};
	\vertex (v4) at (360:0.5) [label=160:$v_{4}$]{};
	\vertex (v5) at (370:3) [label=370:$v_{5}$]{};
\vertex (v9) at (375:3) [label=375:$v_{9}$]{};
	\vertex (v6) at (170:3) [label=170:$v_{6}$]{};
	\vertex (v8) at (165:3) [label=165:$v_{8}$]{};
	\vertex (v10) at (330:3) [label=330:$v_{10}$]{};
	\vertex (v13) at (335:3) [label=335:$v_{13}$]{};
	\vertex (v11) at (210:3) [label=210:$v_{11}$]{};
	\vertex (v12) at (205:3) [label=205:$v_{12}$]{};
	\path 
		(v1) edge (v2)
		(v2) edge (v3)
		(v3) edge (v4)
		(v4) edge (v1)
		(v2) edge (v4)
		(v5) edge (v2)
		(v2) edge (v6)
		(v3) edge (v6)
		(v3) edge (v5)
		(v3) edge (v1)
		(v4) edge (v5)
		(v4) edge (v6)
		(v4) edge (v7)
		(v3) edge (v7)
		(v2) edge (v7)
		(v1) edge (v7)
	    (v2) edge (v8)
	    (v3) edge (v8)
	   	(v4) edge (v8)
	   	(v6) edge (v8)
	   	(v5) edge (v9)
	   	(v2) edge (v9)
	   	(v4) edge (v9)
	   	(v3) edge (v9)
	   	(v3) edge (v11)
	   	(v3) edge (v12)
	   	(v2) edge (v11)
	   	(v2) edge (v12)
	   	(v4) edge (v11)
	   	(v4) edge (v12)
	   	(v11) edge (v12)
	   	(v10) edge (v13)
	   	(v10) edge (v2)
	   	(v10) edge (v3)
	   	(v10) edge (v4)
	   	(v13) edge (v2)
	   	(v13) edge (v3)
	   	(v13) edge (v4)

		;	
\end{tikzpicture}
\caption{Graph $G=\bigcup_{j=1}^{5}S_{j}$ generated by the collection of its maximal cliques $\{S_{j}\}_{j=1}^{5}$, where $V(S_{1})=\{v_{2},v_{3},v_{4},v_{1},v_{7}\}$, $V(S_{2})=\{v_{2},v_{3},v_{4},v_{6},v_{8}\}$, $V(S_{3})=\{v_{2},v_{3},v_{4},v_{11},v_{12}\}$, $V(S_{4})=\{v_{2},v_{3},v_{4},v_{10},v_{13}\}$ and $V(S_{5})=\{v_{2},v_{3},v_{4},v_{5},v_{9}\}$. Clearly,  $A=\{v_{2},v_{3},v_{4}\}$ is the inner hub of $G$.}\label{fig:13} 
\end{center}
\end{figure}
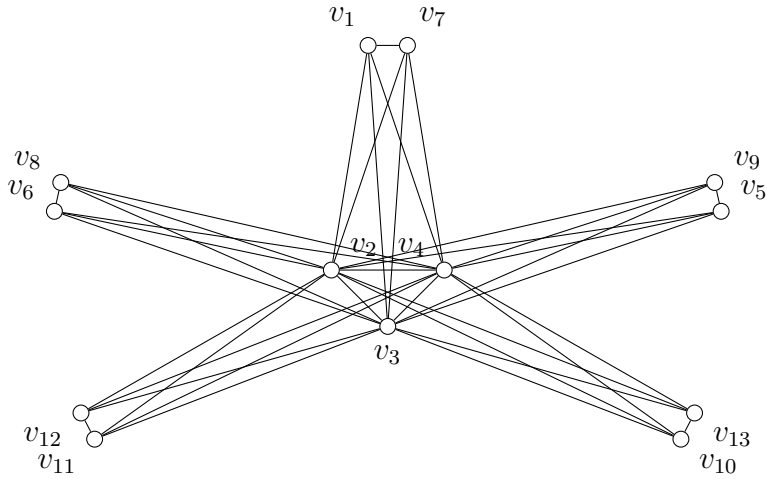   

\item Let $G$ be a graph tree with $V(G)=\{v_{1}, \ldots, v_{m}\}$ and $\mathcal{D}=\{s\in \{1, \ldots, m\}: |N(v_s)|=1\}$. Assume $\mu_{s}$ is absolutely continuous for every $s\in \{2, \ldots, m\} \setminus \mathcal{D}$. Monge solutions for these graphs could be easily deduced by adapting the reasoning presented in Example \ref{eqn:1.2-54}; the solution will be the composition of optimal maps for two marginal problems along any path.   Alternatively, these can be seen as special cases of Proposition \ref{Proposition 3}. 
%we deduce that $G$ gives a unique Monge solution. %In particular, any path give a unique Monge solution. See pictures below. 
\begin{figure}[H]   
\begin{center}
\begin{subfigure}[b]{0.4\linewidth}
\begin{tikzpicture}[x=1.5cm, y=1.5cm]
	\vertex (v1) at (60:1) [label=60:$v_{1}$]{};
	%\vertex (x2) at (120:1) [label=120:$x_{2}$]{};
	\vertex (v3) at (180:1) [label=180:$v_{3}$]{};
	\vertex (v4) at (240:1) [label=240:$v_{4}$]{};
	\vertex (v5) at (300:1) [label=300:$v_{5}$]{};
	\vertex (v6) at (360:1) [label=360:$v_{6}$]{};
	\vertex (v7) at (360:0) [label=5:$v_{7}$]{};
	\path 
		(v7) edge (v1)
		(v7) edge (v3)
		(v1) edge (v6)
		(v6) edge (v5)
		(v5) edge (v4)
		;
\end{tikzpicture}
\caption{Path with vertex sequence $(x_{3},x_{7},x_{1},x_{6},x_{5},x_{4})$. Here, we need regularity conditions on $\mu_{1}, \mu_{5},\mu_{6},\mu_{7}$.} \label{Path} 
  \end{subfigure}
  \qquad \qquad
\begin{subfigure}[b]{0.4\linewidth}
\begin{tikzpicture}[x=1.5cm, y=1.5cm]
	\vertex (v1) at (60:1) [label=60:$v_{1}$]{};
	\vertex (v2) at (120:1) [label=120:$v_{2}$]{};
	\vertex (v3) at (180:1) [label=180:$v_{3}$]{};
	\vertex (v4) at (240:1) [label=240:$v_{4}$]{};
	\vertex (v5) at (300:1) [label=300:$v_{5}$]{};
	\vertex (v6) at (360:1) [label=360:$v_{6}$]{};
	\vertex (v7) at (360:0) [label=5:$v_{7}$]{};
	\vertex (v8) at (360:2) [label=360:$v_{8}$]{};
	\vertex (v9) at (30:2) [label=30:$v_{9}$]{};
	\vertex (v10) at (340:3) [label=340:$v_{10}$]{};
	\path 
		(v7) edge (v1)
		(v7) edge (v3)
		(v1) edge (v6)
		(v6) edge (v5)
		(v5) edge (v4)
		(v1) edge (v2)
		(v6) edge (v9)
		(v9) edge (v8)
		(v8) edge (v10)
		;
\end{tikzpicture}
\caption{Here, we need regularity conditions on $\mu_{k}$, for every $k\in \{5,6,7,8,9\}$.}\label{Tree}
\end{subfigure}
\caption{}
\end{center}
\end{figure}
\end{enumerate}
\end{examples}

\section{ Uniqueness}
Here, we include  a standard argument, showing that in situations where all solutions are of Monge type, the solution to \eqref{KP} must be unique.
	
	\begin{corollary}
		Under the hypotheses in any of Theorem \ref{Theorem 2}, Theorem \ref{Theorem 1}, Proposition \ref{Proposition 2} or Proposition \ref{Proposition 3}, the solution to the Kantorovich problem \eqref{KP} is unique.  
	\end{corollary}
\begin{proof}
	If there are two such solutions, $\gamma_0$ and $\gamma_1$, linearity of the Kantorovich functional implies that their interpolant $\gamma_{1/2} =\frac{1}{2}\gamma_0 +\frac{1}{2}\gamma_1$ is also a solution;	under any of the collections of hypotheses listed in the statement of the corollary, the corresponding result then asserts that each of $\gamma_0, \gamma_1$ and $\gamma_{1/2}$  must concentrate on the graph of a function.  This is clearly not possible, as if $\gamma_0, \gamma_1$ concentrate on the graphs of $T_0$ and $T_1$, respectively, $\gamma_{1/2}$ concentrates on the union of these two graphs, which is itself a single graph only if $T_0=T_1$ $\mu_1$ almost everywhere, in which case $\gamma_0=\gamma_1$.
\end{proof}
\section{Discussion and negative examples}
This paper has identified a wide class of graphs leading to Monge solution and uniqueness results in the multi-marginal optimal transport problem \eqref{MP} with corresponding surplus \eqref{Main cost}, under appropriate conditions on the marginals; see Theorems \ref{Theorem 2} and \ref{Theorem 1} as well as Propositions \ref{Proposition 2} and \ref{Proposition 3}.   To the best of our knowledge, such results are not known for any graph which is not covered here. Furthermore,  Part 2 of Proposition 2.1 verifies that the extra regularity conditions on the marginals imposed here are necessary in order to obtain Monge solution and uniqueness results. %On the other hand, we have also shown that such results cannot hold unless basic structural properties are satisfied; in particular, the graph must be connected, see \ref{prop: non ToSS}

There are many graphs to which none of Theorem \ref{Theorem 2}, Theorem \ref{Theorem 1}, Proposition \ref{Proposition 2} or \ref{Proposition 3} apply, and for most of these we do not know whether or not Monge solution and uniqueness results might hold, assuming for simplicity that all the marginals are absolutely continuous.  A notable exception to this is the cycle graph for $m\geq 5$ (see Figure 2-(b) for the case $m=7$); in a recent work \cite{Pass5}, we showed the existence of absolutely continuous marginals generating non-Monge solutions for the corresponding surplus \eqref{cycle cost}.  For illustrative purposes, we close by mentioning a class of graphs falling outside the scope of this paper, for which Monge solution and uniqueness remain completely open. For this, recall that for graphs $G_1$ and $G_2$ with disjoint vertex sets $V_1$ and $V_2$, the \emph{graph join} $G_1 +G_2$ is defined as the graph union $G_1 \cup G_2$ together with all edges joining vertices in $V_1$ with vertices in $V_2$. Also, for any graph $G$, the \textit{graph complement} (denoted $\overline{G}$) is the graph with vertices $V(G)$ and set of edges $E(\overline{G})=\left\lbrace \{v, w\}:v,w\in V(G)\;\;\text{and}\;\;\{v, w\}\notin E(G)\right\rbrace$.
\begin{definition}
Let $P_n$ be a path with $n$ vertices and $\overline{C_k}$ the complement of the complete graph with $k$ vertices $C_k$ (so $N_{\overline{C_k}}(v)=\emptyset$ for every $v\in V(\overline{C_k})$).  The fan graph $F_{k,n}$ is defined as the graph $ \overline{C_k} + P_n$. %with $N_{F_{k,n}}(v)=V(P_n)$ for every $v\in V(\overline{C_k} )$.}
\end{definition}
\begin{example}
Let us illustrate the above definition with some basics examples.
\begin{itemize}
\item The graphs $F_{1,1}$ and $F_{1,2}$ reduce to complete graphs with two and three vertices respectively. 
\item The graph $F_{1,3}$ reduces to the extraction of the graph consisting of only one edge from the complete graph $C_4$.
\item The graphs $F_{1,6}$ and $F_{2,5}$, where for $F_{1,6}$  we denote the only vertex of $C_1$ as $v_1$, and for $F_{2,5}$ we denote the vertices of $C_2$  as $v_1$ and $v_7$. See figures below
\begin{figure}[H]   
\begin{center}
\begin{subfigure}[b]{0.4\linewidth}
\begin{tikzpicture}[x=1.5cm, y=1.5cm]
		\vertex (v1) at (51:1) [label=51:$v_1$]{};
		\vertex (v2) at (103:1) [label=103:$v_2$]{};
		\vertex (v3) at (154:1) [label=154:$v_3$]{};
		\vertex (v4) at (206:1) [label=206:$v_4$]{};
		\vertex (v5) at (257:1) [label=257:$v_5$]{};
		\vertex (v6) at (309:1) [label=309:$v_6$]{};
		\vertex (v7) at (360:1) [label=360:$v_7$]{};
		\path 
		
		(v1) edge (v6)
		(v1) edge (v4)
		(v6) edge (v7)
		%(v6) edge (v4)
		%(v6) edge (v2)
		%(v6) edge (v3)
		(v7) edge (v1)
		%(v7) edge (v2)
		%(v7) edge (v4)
		%(v7) edge (v3)
		%(v7) edge (v5)
		(v5) edge (v1)
		%(v5) edge (v3)
		(v5) edge (v4)
		%(v5) edge (v2)
		%(v4) edge (v2)
		(v4) edge (v3)
		(v2) edge (v3)
		(v1) edge (v2)
		(v1) edge (v3)
		(v5) edge (v6)
		;
		\end{tikzpicture}
\caption{$F_{1,6}$} 
  \end{subfigure}
  \qquad \qquad
\begin{subfigure}[b]{0.4\linewidth}
\begin{tikzpicture}[x=1.5cm, y=1.5cm]
		\vertex (v7) at (51:1) [label=51:$v_7$]{};
		\vertex (v2) at (103:1) [label=103:$v_2$]{};
		\vertex (v3) at (154:1) [label=154:$v_3$]{};
		\vertex (v4) at (206:1) [label=206:$v_4$]{};
		\vertex (v5) at (257:1) [label=257:$v_5$]{};
		\vertex (v6) at (309:1) [label=309:$v_6$]{};
		\vertex (v1) at (360:1) [label=360:$v_1$]{};
		\path 
		
		(v7) edge (v6)
		(v7) edge (v4)
		(v6) edge (v1)
		%(v6) edge (v4)
		%(v6) edge (v2)
		%(v6) edge (v3)
		%(v7) edge (v1)
		%(v7) edge (v2)
		%(v7) edge (v4)
		%(v7) edge (v3)
		%(v7) edge (v5)
		(v5) edge (v7)
		%(v5) edge (v3)
		(v5) edge (v4)
		%(v5) edge (v2)
		%(v4) edge (v2)
		(v4) edge (v3)
		(v2) edge (v3)
		(v7) edge (v2)
		(v7) edge (v3)
		(v5) edge (v6)
		(v2) edge (v1)
		(v3) edge (v1)
		(v4) edge (v1)
		(v5) edge (v1)
		;
		\end{tikzpicture}
\caption{$F_{2,5}$}
\end{subfigure}
\caption{}
\end{center}
\end{figure}
\end{itemize}
\end{example}
\begin{proposition}\label{Proposition 6.1}
Let $F_{k,n}$ be a fan graph.
\begin{enumerate}
\item If $n\geq 4$, then $F_{k,n}$ does not belong to the class of graphs in Theorem \ref{Theorem 2}, Theorem \ref{Theorem 1}, Proposition \ref{Proposition 2} or Proposition \ref{Proposition 3}.
 \item If $n< 4$, then $F_{k,n}$ belongs to the class of graphs considered in Theorem \ref{Theorem 2}.
\end{enumerate}
\end{proposition}
The proof of Part 1 of the above proposition will be divided into two cases. In both cases the next lemma  will be used during the proofs. 
\begin{lemma}\label{Lemma 6.1}
Assume $n\geq 4$. Then $F_{1,n}$  does not have an inner hub.
\end{lemma}
\begin{proof}
 Assume $F_{1,n}$ has an inner hub. Since $F_{1,n}$ is  connected, every vertex in the nonempty hub is adjacent to all the other vertices. Now, the only vertex of $F_{1,n}$ satisfying this property is the vertex of $\overline{C_1}=C_1$ (so $V(C_1)$ is the hub of $F_{1,n}$).  This implies by definition of inner hub that $P_n$ is complete or it is the disjoint union of complete graphs. This is a contradiction as  $n> 2$ (so $P_n$ can not be complete)   and it is connected, completing the proof of the lemma.
\end{proof}
\begin{proof}[Proof of Proposition \ref{Proposition 6.1}]
Since Proposition \ref{Proposition 3} generalizes Theorem \ref{Theorem 1} and Proposition \ref{Proposition 2}, it suffices to prove Part 1 for Theorem \ref{Theorem 2} and Proposition \ref{Proposition 3}. For this, we set $m=k+n$ and consider two cases.
\begin{enumerate}[label=\textbf{Case \arabic*.}]
\item Assume $k=1$.  If $F_{1,n}=C_m\setminus S$ for some subgraph $S$ of $C_m$, then $S=\overline{P_n}$ or $S=C_1 \cup \overline{P_n} $. Since  $n\geq 4$, $\overline{P_n} $ is connected and there is not a vertex in $V(\overline{P_n})$  adjacent to all the other vertices; that is, $\overline{P_n}$ can not have an inner hub. Also, the only way that the disconnected graph $C_1 \cup \overline{P_n} $ has an inner hub is when $\overline{P_n} $ is complete (as it is connected), which is clearly not the case. Hence, the structure of $F_{1,n}$ does not correspond to the graphs considered in Theorem \ref{Theorem 2}. On the other hand, note that the vertex of $\overline{C_1}=C_1$ is connected to all the other vertices of $F_{1,n}$, and the only case in Proposition \ref{Proposition 3} where a vertex of a graph $\bigcup_{\alpha=1}^{l}G_{\alpha}$ ( where $\{G_\alpha\}_{\alpha=1}^{l}$ is a collection of graphs with inner hubs $A_{\alpha}$ satisfying the conditions  in Proposition \ref{Proposition 3})   satisfies this condition is when $l=1$; that is, if there exists one of these collections satisfying   $\bigcup_{\alpha=1}^{l}G_{\alpha}=F_{1,n}$, then $F_{1,n}$ would be a graph with an inner hub,   contradicting Lemma \ref{Lemma 6.1}.  This proves that $F_{1,n}$ does not belong to the class of graphs in Proposition \ref{Proposition 3}, completing the proof of Case 1.
%Then if $F_{1,k}$ belongs to the class of graphs in Proposition \ref{Proposition 3}, it must have inner hub, in particular,  $P_n$ can be expressed as a disjoint union of complete graphs. Then we must have $n\in \{1,2\}$, contradicting the hypothesis.
\item Assume $k\geqslant 2$. If $F_{k,n}=C_m\setminus S$ for some subgraph $S$ of $C_m$, then  $S=C_k \cup \overline{P_n} $. Note that $S$ is disconnected with connected components $C_k$ and $\overline{P_n} $ (as $n\geqslant 4)$, so if  $S$ has inner hub then it must be empty, which implies $\overline{P_n}$ is complete. This clearly is not possible as $n>1$. Hence,  $F_{k,n}$ does not belong to the class of graphs in Theorem \ref{Theorem 2}. For the other part of the assertion, consider $\{G_\alpha\}_{\alpha=1}^{l}$ a collection of graphs with inner hubs $A_{\alpha}$ satisfying the conditions imposed in Proposition \ref{Proposition 3} and assume $F_{k,n}=\bigcup_{\alpha=1}^{l}G_{\alpha}$. Fix any vertex $v$ in $V(\overline{C_k})\subseteq V(F_{k,n})$, then  there exists $\beta$ such that $v\in A_{\beta}$ or $v\in V(S_{\beta})\setminus A_{\beta}$ for some maximal clique $S_{\beta}$ of $G_{\beta}$. If $v\in A_{\beta}$, then 
\begin{flalign*}
V(G_{\beta})& =\left(  V(G_{\beta})\setminus \{v\}\right) \cup \{v\}\\
& = N_{\bigcup_{\alpha=1}^{l}G_{\alpha}}(v) \cup \{v\}\nonumber\\
    &= N_{F_{k,n}}(v)\cup \{v\}\nonumber\\
    &= V(P_n)\cup \{v\}\qquad\qquad \text{ as $v\in V(\overline{C_k})$}\nonumber\\
\end{flalign*} 
This implies that  $G_{\beta}=P_n \cup K_{v, V(P_n)}$ where  $ K_{v, V(P_n)}$ is a bi-partite graph with set partition $\{ \{v\},V(P_n)\}$ (alternatively we can interpret it as a star graph with "center"  $v$); that is, $G_{\beta}$ is a graph of the form $F_{1,n}$ having an inner hub. This is a contradiction by Lemma \ref{Lemma 6.1}. This proves that $F_{k,n}$ does not satisfy the graph structure condition in Proposition \ref{Proposition 3}.
 % $P_n \cup K_{v, V(P_n)}$  has inner hub $A_{\beta}$. Since $v$ is the only vertex of  $P_n \cup K_{v, V(P_n)}$ adjacent to all the other vertices of  $P_n \cup K_{v, V(P_n)}$, we must have $A_{\beta}=\{v\}$. 
 Now, assume $v\in V(S_{\beta})\setminus A_{\beta}$ and without lost of generality assume $v\notin A_{\alpha}$ for any $\alpha \neq \beta$ (otherwise we apply the same arguments as in the case $v\in A_{\beta}$ above),  then $V(S_{\beta})=N_{F_{k,n}}(v) \cup \{v\}=V(P_n)\cup \{v\}$; that is, $P_n \cup K_{v, V(P_n)}$ is a complete graph (so it has inner hub $V(P_n \cup K_{v, V(P_n)}))$, contradicting Lemma \ref{Lemma 6.1}. Hence, $F_{k,n}$ does not belong to the class of graphs in Proposition \ref{Proposition 3}, completing the proof of Part 1.
\end{enumerate}
To prove Part 2, note that if $n\in \{1,2,3\}$ the graph $\overline{P_n}$ can be trivially expressed as a union of disjoint complete graphs, so $C_{k}\cup \overline{P_n}$ is a disjoint union of complete graphs and can be interpreted as a graph with empty inner hub. Since $F_{k,n}=C_{m}\setminus (C_{k}\cup \overline{P_n})$, we immediately conclude that $F_{k,n}$ belong to the class of graphs in Theorem \ref{Theorem 2}. This completes the Proof of part 2.
\end{proof}
 We note that the essential ideas in the proposition above can in fact be adapted to a more abstract class of graphs.  The next lemma describes such a class, which therefore also falls outside the scope of the results in this paper and for which the Monge solution and uniqueness questions remain open. %Following some of the essential ideas of the above proposition, fan graphs $F_{k,n}$ with $n\geq 4$ can be extended to a more abstract class of graphs. The next lemma describe one of these classes and states that any graph of such class neither is  enclosed in our results.
\begin{lemma}\label{lem: general negative examples}
Let $G$ be a connected graph satisfying $N_{G}(v)\cup \{v\}\neq V(G)$, for all $v\in V(G)$ and consider the graph $\mathcal{F}_{k, G}:= \overline{C_k} + G$. %with $N_{\mathcal{F}_{k, G}}(v)=V(G)$ for every $v\in V(\overline{C_k} )$. 
Then  $\mathcal{F}_{k, G}$  does not belong to the classes of graph considered in  Theorem \ref{Theorem 2} and Proposition \ref{Proposition 3}.
\end{lemma}
\begin{proof}
Note that the condition $N_{G}(v)\cup \{v\}\neq V(G)$, for all $v\in V(G)$ implies that $G$ does not have an inner hub (there is not a vertex in $V(G)$ adjacent to all the other vertices). Also, since $G$ is connected, $N_{\overline{G}}(v)\cup \{v\}\neq V(\overline{G})$ for all $v\in V(\overline{G})=V(G)$, so  $\overline{G}$ also has no inner hub. In particular, $G$ and $\overline{G}$ are not complete and can not be expressed as a disjoint union of complete graphs. Knowing this, it is not hard to follow the arguments of Lemma \ref{Lemma 6.1} to prove that $\mathcal{F}_{1, G}$ does not have inner hub, and then, by mimicking the proof of the above proposition the proof is completed.
\end{proof}
%We finish this section by remarking that the uniqueness and Monge solutions questions remain open for this class as well. See example below for a graph satisfying the above conditions and not being of the form $F_{k,n}$.
Lemma \ref{lem: general negative examples} allows one to construct many graphs for which Monge solution and uniqueness results are not known, with more adhoc structure than the fan graphs considered above.  One such possibility is illustrated in the figure below.
\begin{figure}[H]   
\begin{center}
\begin{tikzpicture}[x=1.5cm, y=1.5cm]
		\vertex (v7) at (51:1) [label=51:$v_7$]{};
		\vertex (v2) at (103:1) [label=103:$v_2$]{};
		\vertex (v3) at (154:1) [label=154:$v_3$]{};
		\vertex (v4) at (206:1) [label=206:$v_4$]{};
		\vertex (v5) at (257:1) [label=257:$v_5$]{};
		\vertex (v6) at (309:1) [label=309:$v_6$]{};
		\vertex (v1) at (360:1) [label=360:$v_1$]{};
		\path 
		
		(v7) edge (v6)
		(v7) edge (v4)
		(v6) edge (v1)
		%(v6) edge (v4)
		%(v6) edge (v2)
		%(v6) edge (v3)
		%(v7) edge (v1)
		%(v7) edge (v2)
		%(v7) edge (v4)
		%(v7) edge (v3)
		%(v7) edge (v5)
		(v5) edge (v7)
		%(v5) edge (v3)
		(v5) edge (v4)
		%(v5) edge (v2)
		%(v4) edge (v2)
		(v4) edge (v3)
		(v2) edge (v3)
		(v7) edge (v2)
		(v7) edge (v3)
		(v5) edge (v6)
		(v2) edge (v1)
		(v3) edge (v1)
		(v4) edge (v1)
		(v5) edge (v1)
		(v2) edge (v6)
		;
		\end{tikzpicture}
\end{center}
\end{figure}

\end{document}